\documentclass[final,leqno,onefignum,onetabnum]{siamltex1213}

\usepackage{tikz}
\usetikzlibrary{arrows,chains,matrix,positioning,scopes}
\makeatletter
\tikzset{join/.code=\tikzset{after node path={%
\ifx\tikzchainprevious\pgfutil@empty\else(\tikzchainprevious)%
edge[every join]#1(\tikzchaincurrent)\fi}}}
\makeatother
\tikzset{>=stealth',every on chain/.append style={join},
         every join/.style={->}}
\tikzstyle{labeled}=[execute at begin node=$\scriptstyle,
   execute at end node=$]

\usepackage[hyphenbreaks]{breakurl}
\usepackage{cite}
\usepackage{amsmath, amssymb, amsfonts}
\usepackage{array}
\newcommand{\qcqp}{\mbox{QCQP-$\mathbb{C}$}}
\newcommand{\csdpr}{\mbox{CSDP-$\mathbb{R}$}}
\newcommand{\sdpr}{\mbox{SDP-$\mathbb{R}$}}
\newcommand{\sdpc}{\mbox{SDP-$\mathbb{C}$}}
\newcommand{\socpc}{\mbox{SOCP-$\mathbb{C}$}}
\newcommand{\socpr}{\mbox{SOCP-$\mathbb{R}$}}
\newcommand{\csocpr}{\mbox{CSOCP-$\mathbb{R}$}}
\allowdisplaybreaks
\newtheorem{example}{\textit{Example}}[section]
\newtheorem{remark}{\textit{Remark}}[section]

\usepackage{rotating}
\usepackage{tablefootnote}

\usepackage{algorithm2e}
\usepackage[noend]{algpseudocode}

\usepackage[multiple]{footmisc}

\title{Moment/Sum-of-Squares Hierarchy for Complex Polynomial Optimization}

\author{C\'edric Josz\footnotemark[2]\ \footnotemark[4]\
\and Daniel~K. Molzahn\footnotemark[3]
}

\usepackage{soul}

\begin{document}
\maketitle

\renewcommand{\thefootnote}{\fnsymbol{footnote}}

\footnotetext[2]{French National Research Institute in Scientific Computing INRIA, Paris-Rocquencourt, BP 105, F-78153 Le Chesnay, France.}
\footnotetext[3]{Department of Electrical Engineering and Computer Science, University of Michigan, Ann Arbor, MI 48109, USA (\email{molzahn@umich.edu}). Support from Dow Sustainability Fellowship, ARPA-E grant DE-AR0000232 and Los Alamos National Laboratory subcontract 270958.}
\footnotetext[4]{French Transmission System Operator RTE, 9, rue de la Porte de Buc, BP 561, F-78000 Versailles, France (\email{cedric.josz@rte-france.com}). Support from CIFRE ANRT 
contract 2013/0179.}

\renewcommand{\thefootnote}{\arabic{footnote}}

\slugger{mms}{xxxx}{xx}{x}{x--x}

\begin{abstract}
We consider the problem of finding the global optimum of a real-valued complex polynomial on a compact set defined by real-valued complex polynomial inequalities. It reduces to solving a sequence of complex semidefinite programming relaxations that grow tighter and tighter thanks to D'Angelo's and Putinar's Positivstellenstatz discovered in 2008. In other words, the Lasserre hierarchy may be transposed to complex numbers. We propose an algorithm for exploiting sparsity and apply the complex hierarchy to problems with several thousand complex variables. They consist in computing optimal power flows in the European high-voltage transmission network. 
\end{abstract}

\begin{keywords}
Quillen property,
Lasserre hierarchy,
Shor relaxation,
complex moment problem,
sparse semidefinite programming,
optimal power flow.
\end{keywords}

\begin{AMS}\end{AMS}

\pagestyle{myheadings}
\thispagestyle{plain}
\markboth{C\'EDRIC JOSZ AND DANIEL~K. MOLZAHN}{Hierarchy for Complex Polynomial Optimization}
\section{Introduction}
Multivariate polynomial optimization where variables and data are complex numbers is a non-deterministic polynomial-time hard problem that arises in various applications such as electric power systems (Section \ref{sec:Application to Electric Power Systems}), imaging science~\cite{singer-2011,candes-2013,bandeira-2014,fogel-2014}, signal processing~\cite{maricic-2003,aittomaki-2009,chen-2009,luo-2010,li-2012,aubry-2013}, automatic control~\cite{toker-1998}, and quantum mechanics~\cite{hilling-2010}. Complex numbers are typically used to model oscillatory phenomena which are omnipresent in physical systems. Although complex polynomial optimization problems can readily be converted into real polynomial optimization problems, efforts have been made to find \textit{ad hoc} solutions~\cite{sorber-2012,jiang-2014,jiang-2015}. We observe that relaxing non-convex constraints and converting from complex to real numbers are two non-commutative operations. This leads us to transpose to complex numbers Lasserre's moment/sum-of-squares hierarchy~\cite{lasserre-2001} for real polynomial optimization.

In 1968, Quillen~\cite{quillen-1968} showed that a real-valued bihomogenous complex polynomial that is positive away from the origin can be decomposed as a sum of squared moduli of holomorphic polynomials when it is multiplied by $(|z_1|^2\! +\! \hdots\! +\! |z_{n}|^2)^r$ for some $r\in \mathbb{N}$. The result was rediscovered by Catlin and D'Angelo~\cite{catlin-1996} and ignited a search for complex analogues of Hilbert's seventeenth problem~\cite{angelo-2002,angelo-2010} and the ensuing Positivstellens\"atze~\cite{putinar-2006,putinar-2012,putinar-scheiderer-2012,putinar-2013}. Notably, D'Angelo and Putinar~\cite{angelo-2008} proved in 2008 that a positive complex polynomial on a sphere intersected by a finite number of polynomial inequality constraints can be decomposed as a weighted sum of the constraints where the weights are sums of squared moduli of holomorphic polynomials.
Similar to Lasserre~\cite{lasserre-2001} and Parrilo~\cite{parrilo-2003}, we use D'Angelo's and Putinar's Positivstellensatz to construct a complex moment/sum-of-squares hierarchy of semidefinite programs to solve complex polynomial optimization problems with compact feasible sets. To satisfy the assumption in the Positivstellensatz, we propose to add a slack variable $z_{n+1} \in \mathbb{C}$ and a redundant constraint $|z_1|^2\! +\! \hdots\! +\! |z_{n+1}|^2\! =\! R^2$ to the description of the feasible set when it is in a ball of radius $R$. The complex hierarchy is more tractable than the real hierarchy yet produces potentially weaker bounds. Computational advantages are shown using the optimal power flow problem in electrical engineering. In addition to global convergence of the bounds, the complex hierarchy is endowed with sufficient conditions for extracting feasible points that are globally optimal.

The theoretical contributions of this paper regarding the complex hierarchy are: 
\begin{enumerate}
\item its construction using real-valued Radon measures (Section~\ref{sec:Complex Moment/Sum-of-Squares Hierarchy}) leading to a new notion of complex moment matrix and localization matrix (Remark~\ref{rem:Ly}) different from existing literature~\cite{curto-2005}; the Lasserre hierarchy~\cite{lasserre-2001} can thus be viewed as a special case of the proposed complex hierarchy (Figure~\ref{fig:hankel});
\item a proof of global convergence (Proposition~\ref{prop:dualconverge}, Corollary~\ref{cor:convergence}); a sufficient condition for strong duality (Proposition~\ref{prop:duality}); Karush-Kuhn-Tucker conditions involving complex sums-of-squares (Corollary~\ref{cor:KKT}); a multi-ordered hierarchy to exploit sparsity while preserving global convergence (Section~\ref{subsec:Multi-Ordered Hierarchy});
\item a solution to a newly defined truncated complex moment problem (Theorem~\ref{th:trunc}) different from existing literature~\cite[Theorem 5.1]{curto-2005} which implies Curto and Fialkow's solution of the real truncated moment problem (Corollary~\ref{cor:trunc}); as a result, sufficient conditions for extracting global solutions from the complex hierarchy (Proposition~\ref{prop:rank});
\item an invariant complex hierarchy whose convergence can be deduced from an invariant version of D'Angelo's and Putinar's Positivstellensatz (Proposition~\ref{prop:invariance}); in particular, an action of the torus in the complex plane (Proposition~\ref{prop:torus}) and a subgroup of it (Proposition~\ref{prop:subtorus}) are considered.
\end{enumerate}

The paper is organized as follows. Section \ref{sec:Motivation} uses Shor and second-order conic relaxations to motivate the complex moment/sum-of-squares hierarchy in Section \ref{sec:Complex Moment/Sum-of-Squares Hierarchy}. Using a sparsity-exploiting algorithm, numerical experiments on the optimal power flow problem are presented in Section \ref{sec:Application to Electric Power Systems}. Section \ref{sec:Conclusion} concludes our work.

\section{Motivation}
\label{sec:Motivation}

Let $\mathbb{N}$, $\mathbb{N}^*$, $\mathbb{R}$, $\mathbb{R}_+$ and $\mathbb{C}$ denote the set of natural, positive natural, real, non-negative real, and complex numbers respectively. Also, let ``$\textbf{i}$'' denote the imaginary unit and $\mathbb{H}_n$ denote the set of Hermitian matrices of order $n\in \mathbb{N}^*$. Consider the subclass of complex polynomial optimization
\begin{equation}
\label{eq:qcqp}
\text{QCQP-}\mathbb{C}~\text{:} ~~~
\inf_{z \in \mathbb{C}^n}~ z^H H_0 z ~~~ \text{s.t.} ~~~
z^H H_i z \leqslant h_i,~~~ i=1, \hdots, m,
\end{equation}
where $m \in \mathbb{N}^*$, $H_0,\hdots,H_m \in \mathbb{H}_n$, $h_0,\hdots,h_m \in \mathbb{R}$, $\left(\cdot\right)^H$ denotes the conjugate transpose. The Shor~\cite{shor-1987b} and second-order conic relaxations of QCQP-$\mathbb{C}$ share the following property: it is better to relax non-convex constraints before converting from complex to real numbers rather than to do the two operations in the opposite order.

\subsection{Shor Relaxation}
\label{subsec:Shor Relaxation}
For $H \in \mathbb{H}_n$ and $z \in \mathbb{C}^n$, the relationship $z^H H z = \text{Tr}(H z z^H)$ holds where $\text{Tr}\left(\cdot\right)$ denotes the trace
\footnote{For all matrices $A,B \in \mathbb{C}^{n\times n}$, $\text{Tr}(AB) = \sum_{1 \leqslant i,j \leqslant n} A_{ij} B_{ji}$.} 
of a complex square matrix. Let $\succcurlyeq 0$ indicate positive semidefiniteness. Relaxing the rank of $Z = zz^H$ in $\eqref{eq:qcqp}$ yields 
\begin{subequations}
\begin{gather}
\text{SDP-}\mathbb{C}~\text{:} ~~~
\inf_{Z \in \mathbb{H}_n}~ \text{Tr}(H_0 Z) ~~~~~~~~~~~~~~~~~~~~~~ \label{eq:sdpC1} \\
~~~~~~~~~~~~~~~~~~~~~~\text{s.t.} ~~~
\text{Tr}(H_i Z) \leqslant h_i,~~~~ i=1, \hdots, m, \label{eq:sdpC1bis}\\
Z \succcurlyeq 0, \label{eq:sdpC2} ~
\end{gather}
\end{subequations}
Let $\text{Re}Z$ and $\text{Im}Z$ denote the real and imaginary parts of the matrix $Z\in \mathbb{C}^{n\times n}$ respectively. Consider the ring homomorphism $\Lambda : (\mathbb{C}^{n\times n},+,\times) \longrightarrow (\mathbb{R}^{2n\times 2n},+,\times)$
\begin{equation}
\label{eq:conversion}
\Lambda(Z) :=
\left( \begin{array}{cr}
\text{Re}Z & - \text{Im} Z \\
\text{Im}Z & \text{Re}Z
\end{array} 
\right).
\end{equation}
To convert \sdpc{} into real numbers, real and imaginary parts of the complex matrix variable are identified using two properties: (1) a complex matrix $Z$ is positive semidefinite if and only if the real matrix $\Lambda(Z)$ is positive semidefinite, and (2) if $Z_1,Z_2 \in \mathbb{H}_n$, then $\text{Tr}\left[\Lambda(Z_1)\Lambda(Z_2)\right] = \text{Tr}\left[\Lambda(Z_1Z_2)\right]  = 2\text{Tr}(Z_1 Z_2)$. 
This yields
\begin{subequations}
\begin{gather}
\text{CSDP-}\mathbb{R}~\text{:} ~~~ \inf_{X \in \mathbb{S}_{2n}} ~ \text{Tr}( \Lambda(H_0) X) ~~~~~~~~~~~~~~~~~~~~~~~~~ \label{eq:csdpR1} \\
~~~~~~~~~~~~~~~~~~~~~~\text{s.t.} ~~~
\text{Tr}( \Lambda(H_i) X) \leqslant h_i,~~~~ i=1, \hdots, m, \label{eq:csdpR2}
\\
X  \succcurlyeq 0, \label{eq:csdpR3} ~~~~~~
\\ ~~~~~~~~~~~~~~~~~~~~~~~~~~~~~~~
X = 
\left(
\begin{array}{cl}
A & B^T \\
B & C
\end{array}
\right) ~~~\&~~~ 
\begin{array}{lcr}
A & = & C, \\
B^T & = & -B,
\end{array}
\label{eq:csdpR4}
\end{gather}
\end{subequations}
where $\mathbb{S}_{2n}$ denotes the set of real symmetric matrices of order~$2n$ and $\left(\cdot\right)^T$ indicates the transpose. Note that the set of matrices satisfying \eqref{eq:csdpR4} is isomorphic to $\mathbb{C}^{n \times n}$. A global solution to \qcqp{} can be retrieved from \csdpr{} if and only if $\text{rank}(X)\in \{0,2\}$ at optimality (proof in Appendix \ref{app:Rank-2 Condition}).
In order to convert \qcqp{} into real numbers, real and imaginary parts of the complex vector variable are identified. This is done by considering a new variable $x = \left( ~ (\text{Re}z)^T ~ (\text{Im}z)^T ~ \right)^T $ and observing that if $H \in \mathbb{H}_n$, then $z^H H z = x^T \Lambda(H) x = \text{Tr}(\Lambda(H)xx^T)$. This gives rise to a problem which we will call QCQP-$\mathbb{R}$. Relaxing the rank of $X=xx^T$ yields
\begin{subequations}
\begin{gather}
\text{SDP-}\mathbb{R}~\text{:} ~~~ \inf_{X \in \mathbb{S}_{2n}} ~ \text{Tr}( \Lambda(H_0) X) ~~~~~~~~~~~~~~~~~~~  \label{eq:sdpR1} \\
~~~~~~~~~~~~~~~~~~~~~~~~~~\text{s.t.} ~~~
\text{Tr}( \Lambda(H_i) X) \leqslant h_i, ~~~~ i=1, \hdots, m, \label{eq:sdpR2}
\\
X  \succcurlyeq 0. \label{eq:sdpR3} ~~~ 
\end{gather}
\end{subequations}
A global solution to \qcqp{} can be retrieved from \sdpr{} if and only if $\text{rank}(X)\in \{0,1\}$ or $\text{rank}(X)=2$ and \eqref{eq:csdpR4} holds at optimality.
We have $\text{val}(\text{SDP-}\mathbb{C}) = \text{val}(\text{CSDP-}\mathbb{R}) = \text{val}(\text{SDP-}\mathbb{R})$
where ``val'' is the optimal value of a problem (proof in Appendix \ref{app:Invariance of Shor Relaxation Bound}).
The number of scalar variables of \csdpr{} is half that of \sdpr{} due to constraint \eqref{eq:csdpR4}. This constraint also halves the possible ranks of the matrix variable, which must be an even integer in \csdpr{} whereas it can be any integer between 0 and $2n$ in \sdpr{}. The number of variables in \sdpr{} can be reduced by a small fraction ($\frac{2}{2n+1}$ to be exact) by setting a diagonal element of $X$ to 0. This does not affect the optimal value (proof in Appendix \ref{app:Invariance of SDP-R Relaxation Bound}). See Figure~\ref{fig:commutation} for a summary.
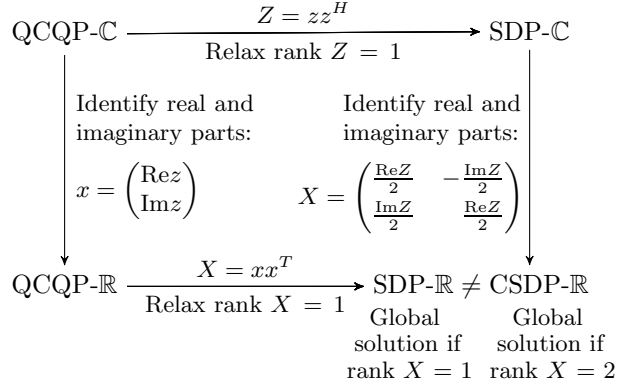
\begin{figure}[ht]
  \centering

\begin{tikzpicture}[scale=1, transform shape]
 \matrix (m) [matrix of math nodes, row sep=8em, column sep=9em]
	{ \text{QCQP-}\mathbb{C}  & \text{SDP-}\mathbb{C}  \\
     \text{QCQP-}\mathbb{R} & \text{SDP-}\mathbb{R} \neq \text{CSDP-}\mathbb{R} ~~~~~~~~~~~\\  };
{ [start chain] \chainin (m-1-1);
    \chainin (m-1-2) [join={node[above,labeled] {\mbox{\small $Z = zz^H$}}}];
        }
{ [start chain] \chainin (m-1-1);
    \chainin (m-1-2) [join={node[below,labeled] {\mbox{\small \text{Relax rank }$Z \,=\, 1$}}}];
        }
{ [start chain] \chainin (m-1-2);
    \chainin (m-2-2);
        }
{ [start chain] \chainin (m-1-1);
    \chainin (m-2-1);
        }
{ [start chain] \chainin (m-2-1);
    \chainin (m-2-2) [join={node[above,labeled] {\mbox{\small $X = xx^T$}}}];
        }
{ [start chain] \chainin (m-2-1);
    \chainin (m-2-2) [join={node[below,labeled] {\mbox{\small \text{Relax rank }$X \,=\, 1$}}}];
        }

\draw (-3.75,1) node[below right] {\small Identify real and};
\draw (-3.75,0.6) node[below right] {\small imaginary parts:};
\draw (-3.75,0.1) node[below right] {\mbox{\small $x = \begin{pmatrix} \text{Re} z \\ \text{Im} z\end{pmatrix}$}};

\draw (-0.2,1) node[below right] {\small Identify real and};
\draw (-0.2,0.6) node[below right] {\small imaginary parts:};
\draw (-0.8,0.1) node[below right] {\mbox{\small $X = \begin{pmatrix} \frac{\text{Re} Z}{2} & -\frac{\text{Im} Z}{2} \\[0.3em] \frac{\text{Im} Z}{2} & \hphantom{-}\frac{\text{Re} Z}{2}\end{pmatrix}$}};

\draw (0.15,-1.85) node[below right] {\small Global};
\draw (-0.05,-2.2) node[below right] {\small solution if};
\draw (-0.15,-2.55) node[below right] {\small rank $X$ = 1};

\draw (2.05,-1.85) node[below right] {\small Global};
\draw (1.85,-2.2) node[below right] {\small solution if};
\draw (1.75,-2.55) node[below right] {\small rank $X$ = 2};

\end{tikzpicture}
  \caption{Non-Commutativity of Complex-to-Real Conversion and Relaxation}
  \label{fig:commutation}
\end{figure} 
\subsection{Second-Order Conic Relaxation}
\label{subsec:Second-Order Conic Relaxation}
In \sdpc{} of Section~\ref{subsec:Shor Relaxation}, assume that the semidefinite constraint \eqref{eq:sdpC2} is relaxed to the second-order cones
\begin{equation}
\left(
\begin{array}{cc}
Z_{ii} & Z_{ij} \\
Z_{ij}^H & Z_{jj}
\end{array}
\right)
 \succcurlyeq 0 ~~,~~ 1 \leqslant i \neq j \leqslant n. \label{eq:2x2}
\end{equation}
Equation \eqref{eq:2x2} is equivalent to constraining the determinant $Z_{ii} Z_{jj} - Z_{ij} Z_{ij}^H$ and diagonal elements $Z_{ii}$ to be non-negative. This yields $\text{SOCP-}\mathbb{C}~\text{:}  \inf_{Z \in \mathbb{H}_n} \text{Tr}(H_0 Z) ~ \text{s.t.}$ \eqref{eq:sdpC1bis}, $|Z_{ij}|^2 \leqslant Z_{ii} Z_{jj}$ for $1 \leqslant i \neq j \leqslant n$, and $Z_{ii} \geqslant 0$ for $i=1, \hdots, n$ where $|\cdot|$ denotes the complex modulus.
Identifying real and imaginary parts of the matrix variable $Z$ leads to $\text{CSOCP-}\mathbb{R}~\text{:}  \inf_{X \in \mathbb{S}_{2n}} \text{Tr}( \Lambda(H_0) X) ~ \text{s.t.}$ \eqref{eq:csdpR2}, \eqref{eq:csdpR4}, $X_{ij}^2 + X_{n+i,j}^2 \leqslant X_{ii} X_{jj}$ for $1 \leqslant i \neq j \leqslant n$, and $X_{ii} \geqslant 0$ for $i=1, \hdots, n$.
In \sdpr{} of Section~\ref{subsec:Shor Relaxation}, assume that the semidefinite constraint \eqref{eq:sdpR3} is relaxed to the second-order cones
\begin{equation}
\left(
\begin{array}{cc}
X_{ii} & X_{ij} \\
X_{ij} & X_{jj}
\end{array}
\right)
 \succcurlyeq 0 ~~,~~ 1 \leqslant i \neq j \leqslant 2n.
\end{equation}
This leads to $\text{SOCP-}\mathbb{R}~\text{:} \inf_{X \in \mathbb{S}_{2n}} \text{Tr}( \Lambda(H_0) X) ~ \text{s.t.}$ \eqref{eq:sdpR2}, $X_{ij}^2 \leqslant X_{ii} X_{jj}$ for $1 \leqslant i \neq j \leqslant 2n$, and $X_{ii} \geqslant 0$ for $i=1, \hdots, 2n$.
We have
$
\text{val}(\text{SOCP-}\mathbb{C}) = \text{val}(\text{CSOCP-}\mathbb{R}) \geqslant \text{val}(\text{SOCP-}\mathbb{R})
$ (proof in Appendix \ref{app:Discrepancy Between Second-Order Conic Relaxation Bounds}).
The number of variables of \csocpr{} is half that of \socpr{} due to constraint \eqref{eq:csdpR4}. The number of second-order conic constraints in \csocpr{}, equal to $\frac{n(n-1)}{2}$, is roughly a fourth of that in \socpr{}, equal to $\frac{2n(2n-1)}{2}$.

\subsection{Exploiting Sparsity}

\label{subsec:Exploiting sparsity}

The properties of chordal graphs enable sparsity exploitation for the Shor relaxation~\cite{vandenberghe2014}. Given an undirected graph $(\mathcal{V},\mathcal{E})$ where $\mathcal{V} \subset \{1,\hdots,n\}$ and $\mathcal{E}\subset \mathcal{V} \times \mathcal{V}$, define for all $Z \in \mathbb{H}_n$
\begin{equation}
\Psi_{(\mathcal{V},\mathcal{E})}(Z)_{ij} := \left\{ \begin{array}{cl} Z_{ij} & \text{if} ~~ (i,j) \in \mathcal{E} ~~ \text{or} ~~ i = j \in \mathcal{V}, \\ 0 & \text{else}. \end{array} \right.
\end{equation}

We associate an undirected graph $\mathcal{G}$ to \qcqp{} whose nodes are $\{1,\hdots,n\}$ and that satisfies 
$H_i = \Psi_\mathcal{G}(H_i) $ for $i = 0, \hdots, m$. Let $ \mathbb{H}_n^+$ denote the set of positive semidefinite Hermitian matrices of size $n$ and let ``Ker'' denote the kernel of a linear application. Given the definition of $\mathcal{G}$, constraint~\eqref{eq:sdpC2} of \sdpc{} can be relaxed to 
$
Z \in \mathbb{H}_n^+ + \text{Ker} ~ \Psi_{\tilde{\mathcal{G}}}
$
without changing its optimal value for any graph $\tilde{\mathcal{G}}$ whose nodes are $\{1,\hdots,n\}$ and where $\mathcal{G} \subset \tilde{\mathcal{G}}$. 
Consider a chordal extension $\mathcal{G}\subset \mathcal{G}^\text{ch}$, that is to say that all cycles of length four or more have a chord (edge between two non-consecutive nodes of the cycle). Let $\mathcal{C}_1,\hdots,\mathcal{C}_p \subset \mathcal{G}^\text{ch}$ denote the maximal cliques of $\mathcal{G}^\text{ch}$. (A clique is a subgraph where all nodes are linked to one another. The set of maximally sized cliques of a chordal graph can be computed in linear time~\cite{tarjan}). A chordal extension has a useful property for exploiting sparsity \cite{grone}: for all $Z \in \mathbb{H}_n$, we have that $Z \in \mathbb{H}_n^+ + \text{Ker}~ \Psi_{\mathcal{G}^\text{ch}}$ if and only if $\Psi_{\mathcal{C}_i}(Z) \succcurlyeq 0$ for $i=1,\hdots, p$. Note that $\Psi_{\mathcal{C}_i}(Z) \succcurlyeq 0$ if and only if
$\Lambda \circ \Psi_{\mathcal{C}_i}(Z) \succcurlyeq 0$, where ``$\circ$'' is the composition of functions. Given a graph $(\mathcal{V},\mathcal{E})$, define for $X \in \mathbb{S}_{2n}$
\begin{equation}
\tilde{\Psi}_{(\mathcal{V},\mathcal{E})}(X) :=
\left(
\begin{array}{cl}
\Psi_{(\mathcal{V},\mathcal{E})}(A) & \Psi_{(\mathcal{V},\mathcal{E})}(B^T)  \\
\Psi_{(\mathcal{V},\mathcal{E})}(B)  & \Psi_{(\mathcal{V},\mathcal{E})}(C) 
\end{array}
\right),
\end{equation}
using the block decomposition in the left hand part of \eqref{eq:csdpR4}. Notice that $\Lambda \circ \Psi_{(\mathcal{V},\mathcal{E})}= \tilde{\Psi}_{(\mathcal{V},\mathcal{E})} \circ \Lambda$. As a result, \eqref{eq:csdpR3} can be replaced by
$ \tilde{\Psi}_{\mathcal{C}_i} (X) \succcurlyeq 0$ for $i = 1,\hdots, p$ without changing the optimal value of \csdpr{}, with an analogous replacement for constraint \eqref{eq:sdpR3} in \sdpr{}. If in \sdpr{} we exploit the sparsity of matrices $\Lambda(H_i)$ instead of that of $H_i$, the resulting graph has twice as many nodes. Computing a chordal extension and maximal cliques is hence more costly. Sparsity in the second-order conic relaxations is exploited using the fact that applying constraints only for $(i,j)$ that are edges of $\mathcal{G}$ does not change the optimal values of \csocpr{} and \socpr{}.

\section{Complex Moment/Sum-of-Squares Hierarchy} 
\label{sec:Complex Moment/Sum-of-Squares Hierarchy}
We transpose~\cite{lasserre-2001} from real to complex numbers.
Let $z^\alpha$ denote the monomial $z_1^{\alpha_1} \cdots z_n^{\alpha_n}$ where $z \in {\mathbb C}^n$ and $\alpha \in {\mathbb N}^n$ for some integer $n \in \mathbb N^*$. Let $|\alpha| := \alpha_1+\hdots+\alpha_n$ and define $\overline{w}$ as the conjugate of $w \in \mathbb{C}$. Define $\bar{z} := (\bar{z}_1,\hdots,\bar{z}_n)^T$ where $z \in {\mathbb C}^n$.
Consider the sets where $d\in \mathbb{N}$
\begin{equation}
\begin{array}{rl}
\mathbb{C}[z] := & \{ ~ p : \mathbb{C}^n \rightarrow \mathbb{C} ~|~ p(z) = \sum_{|\alpha|\leqslant l} p_{\alpha} z^\alpha, ~ l\in \mathbb{N}, ~ p_{\alpha} \in \mathbb{C}  ~ \}, \\
\mathbb{C}[\bar{z},z]  := & \{ ~ f : \mathbb{C}^n \rightarrow \mathbb{C} ~|~ f(z) = \sum_{|\alpha|,|\beta|\leqslant l} f_{\alpha,\beta} \bar{z}^\alpha z^\beta,~ l\in \mathbb{N}, ~ f_{\alpha,\beta} \in \mathbb{C} ~ \}, \\
\mathbb{R}[\bar{z},z]  := & \{ ~ f \in \mathbb{C}[\bar{z},z] ~|~ \overline{f(z)} = f(z), ~ \forall z \in \mathbb{C}^n ~ \}, \\[0.25em]
\Sigma[z] := & \{ ~ \sigma : \mathbb{C}^n \rightarrow \mathbb{C} ~|~ \sigma = \sum_{j=1}^r |p_j|^2, ~ r\in \mathbb{N}^*, ~ p_j \in \mathbb{C}[z] ~ \},
\end{array}
\end{equation}
\begin{equation}
\label{eq:complexsetsd}
\begin{array}{rl}
\mathbb{C}_d[z] := & \{ ~ p : \mathbb{C}^n \rightarrow \mathbb{C} ~|~ p(z) = \sum_{|\alpha|\leqslant d} p_{\alpha} z^\alpha, ~ p_{\alpha} \in \mathbb{C}  ~ \}, \\
\mathbb{C}_d[\bar{z},z] := & \{ ~ f : \mathbb{C}^n \rightarrow \mathbb{C} ~|~ f(z) = \sum_{|\alpha|,|\beta|\leqslant d} f_{\alpha,\beta} \bar{z}^\alpha z^\beta, ~ f_{\alpha,\beta} \in \mathbb{C}  ~ \}, \\
\mathbb{R}_d[\bar{z},z]  := & \{ ~ f \in \mathbb{C}_d[\bar{z},z] ~|~ \overline{f(z)} = f(z), ~ \forall z \in \mathbb{C}^n ~ \}, \\[0.25em]
\Sigma_d[z] := & \{ ~ \sigma : \mathbb{C}^n \rightarrow \mathbb{C} ~|~ \sigma = \sum_{j=1}^r |p_j|^2, ~ r\in \mathbb{N}^*, ~ p_j \in \mathbb{C}_d[z] ~ \}.
\end{array}
\end{equation}
Note that the coefficients of a function $f \in \mathbb{R}[\bar{z},z]$ satisfy $\overline{f_{\alpha,\beta}} = f_{\beta,\alpha}$ for all $|\alpha|,|\beta|\leqslant l$ for some $l \in \mathbb{N}$. The set of complex polynomials $\mathbb{C}[\bar{z},z]$ is a $\mathbb{C}$-algebra (i.e. commutative ring and vector space over $\mathbb{C}$) and the set of holomorphic polynomials $\mathbb{C}[z]$ is a subalgebra of it (i.e. subspace closed under sum and product). The set of real-valued complex polynomials $\mathbb{R}[\bar{z},z]$ is an $\mathbb{R}$-algebra. The set of sums of squared moduli of holomorphic polynomials $\Sigma[z]$ and the set $\Sigma_d[z] \subset \mathbb{R}_d[z]$ are pointed cones (i.e. closed under multiplication by elements of $\mathbb{R}_+$) that are convex (i.e. $tu+ (1-t)v$ with $0\leqslant t \leqslant 1$ belongs to them if $u$ and $v$ do). 
Let $C(K,\mathbb{C})$ denote the Banach (i.e. complete) $\mathbb{C}$-algebra of continuous functions from a compact set $K \subset \mathbb{C}^n$ to $\mathbb{C}$ equipped with the norm $\|\varphi\|_{\infty} := \sup_{z\in K} |\varphi(z)|$. Consider $R_K : \mathbb{C}[\bar{z},z] \longrightarrow C(K,\mathbb{C})$ defined by  $f \longmapsto f_{|K}$
where $f_{|K}$ denotes the restriction of $f$ to $K$. $R_K(\mathbb{C}[\bar{z},z])$ is a unital subalgebra of $C(K,\mathbb{C})$ (i.e. contains multiplicative unit) that separates points of $K$ (i.e. $u\neq v \in K \Longrightarrow \exists \varphi \in R_K(\mathbb{C}[\bar{z},z]): \varphi(u)\neq \varphi(v)$) and that is closed under complex conjugation. It is hence a dense subalgebra due to the Complex Stone-Weiestrass Theorem.
Likewise, $C(K,\mathbb{R}) := \{ \varphi \in C(K,\mathbb{C}) ~|~ \overline{\varphi(z)} = \varphi(z), ~ \forall z \in \mathbb{C}^n  \}$ is a Banach $\mathbb{R}$-algebra of which $R_K(\mathbb{R}[\bar{z},z])$ is a dense subalgebra.
In other words, a continuous real-valued function of multiple complex variables can be approximated as close as desired by real-valued complex polynomials when restricted to a compact set. They are hence a powerful modeling tool in optimization. Speaking of which, let $m \in \mathbb{N}^*$ and $k,k_1,\hdots,k_m \in \mathbb{N}$. Consider $(f,g_1,\hdots,g_m) \in \mathbb{R}_k[\bar{z},z] \times \mathbb{R}_{k_1}[\bar{z},z] \times \hdots \times \mathbb{R}_{k_m}[\bar{z},z]$ where there exists $|\alpha|=k$ and $|\beta| \leqslant k$ such that $f_{\alpha,\beta} \neq 0$. In addition, for $i=1,\hdots,m$, there exists $|\alpha|=k_i$ and $|\beta| \leqslant k_i$ such that $g_{i,\alpha,\beta} \neq 0$. Consider the problem
\begin{equation}
\boxed{
\label{eq:complexPOP}
\begin{array}{rcllll}
f^\text{opt} & := & \inf_{z \in {\mathbb C}^n} & f(z) & \mathrm{s.t.} & g_i(z) \geqslant 0,~~i=1,...,m,
\end{array}
}
\end{equation}
where $f^\text{opt} := + \infty$ if the feasible set is empty.
The feasible set $K:= \{ z \in \mathbb{C}^n ~|~ g_i(z) \geqslant 0,~i=1,...,m \}$ is assumed to be compact.
Let $K^\text{opt}$ denote the set of optimal solutions to \eqref{eq:complexPOP} and $\mathcal{M}(K)$ denote the Banach space over $\mathbb{R}$ of Radon measures on $K$. Since $K$ is compact, $\mathcal{M}(K)$ may be identified with the set of linear continuous applications from $C(K,\mathbb{R})$ to $\mathbb{R}$ equipped with the operator norm (Riesz Representation Theorem). For $\varphi \in C(K,\mathbb{C})$, define $\int_K \varphi d\mu := \int_K \text{Re}(\varphi)d\mu + \textbf{i} \int_K \text{Im}(\varphi)d\mu$~\cite[1.31 Definition]{rudin-1987}\footnote{We wish to thank Bruno Nazaret for bringing this reference to our attention.}. Consider the convex pointed cone 
$
\mathcal{P}(K) := \{ ~ \varphi \in C(K,\mathbb{R}) ~|~ \varphi(z) \geqslant 0,~\forall z\in K ~ \}$.
A Radon measure $\mu$ is positive (denoted $\mu \geqslant 0$) if $\varphi \in \mathcal{P}(K)$ implies that $\int_K \varphi d\mu \geqslant 0$. Let $\mathcal{M}_+(K)$ denote the set of positive Radon measures. 
We have
\begin{equation}
\label{eq:moment}
\begin{array}{rcllll}
f^\text{opt} & = & \inf_{\mu \in \mathcal{M}(K)}  & \int_K f d\mu & \mathrm{s.t.} & \int_K d\mu = 1 ~~\&~~ \mu \geqslant 0.
\end{array}
\end{equation}
Indeed, if $z \in K$, then the Dirac\footnote{The Dirac measure $\delta_z$ with $z\in K$ may be identified with the continuous linear application from $C(K,\mathbb{R})$ to $\mathbb{R}$ defined by $\varphi \longmapsto \varphi(z)$. This is one way to interpret the fact that $\int_K f d\delta_z = f(z)$.} 
measure $\delta_z$ is a feasible point of \eqref{eq:moment} for which the objective value is equal to $f(z)$. Hence the optimal value of \eqref{eq:moment} is less than or equal to $f^\text{opt}$. Conversly, if $\mu$ is a feasible point of \eqref{eq:moment}, then $\int_K (f - f^\text{opt})d\mu \geqslant 0$ and hence $\int_K f d\mu \geqslant \int_K f^\text{opt} d\mu = f^\text{opt} \int_K d\mu = f^\text{opt}$. 
\begin{proposition}
\label{prop:measure}
\normalfont
\textit{The set of optimal solutions to \eqref{eq:moment} is}
\begin{equation}
\label{eq:solutions}
\{ ~ \mu \in \mathcal{M}_+(K) ~|~ \mu(K^\text{opt}) = 1 ~~ \& ~~ \mu(K \setminus K^\text{opt}) = 0 ~ \}.
\end{equation}
\textit{As a consequence, if }$K^\text{opt}$\textit{ is a finite set of $S\in \mathbb{N}^*$ points $z(1),\hdots,z(S) \in \mathbb{C}^n$, then the optimal solutions to \eqref{eq:moment} are} $\{ ~ \sum_{j=1}^S \lambda_j \delta_{z(j)} ~ | ~ \sum_{j=1}^S \lambda_j = 1 ~~  \& ~~ \lambda_1,\hdots,\lambda_S \in \mathbb{R}_+ ~ \}$.
\end{proposition}
\begin{proof}
Consider $\mu$ an optimal solution to \eqref{eq:moment}. It must be that $\int_K (f-f^\text{opt})d\mu = 0$. Thus $\int_{K\setminus K^\text{opt}} (f-f^\text{opt}) d\mu = 0$ and $\mu(K \setminus K^\text{opt})  = \int_{K\setminus K^\text{opt}} d\mu = 0$. Therefore $\mu(K^\text{opt}) =  \int_{K^\text{opt}} d\mu = \mu(K) - \mu(K \setminus K^\text{opt}) =1$. Conversly, if $\mu$ belongs to the set in \eqref{eq:solutions}, then it is feasible for \eqref{eq:moment} and $\int_K (f - f^\text{opt}) d\mu = \int_{ K \setminus K^\text{opt} } (f - f^\text{opt}) d\mu = 0$. Hence $\int_K fd\mu=\int_K f^\text{opt} d\mu = f^\text{opt}\int_K d\mu = f^\text{opt}$.
\end{proof}

In order to dualize the equality constraint in \eqref{eq:moment}, consider the Lagrange function $\mathcal{L} : \mathcal{M}_+(K) \times \mathbb{R} \longrightarrow \mathbb{R}$ defined by
$
(\mu,\lambda) \longmapsto \int_K f d\mu + \lambda \left(1-\int_K d\mu \right)
$.
We have $\mathcal{L}(\mu,\lambda)  =  \lambda + \int_K (f - \lambda) d\mu$ and
\begin{equation}
\inf_{\mu \in \mathcal{M}_+(K)} \int_K (f- \lambda) d\mu =
\left\{
\begin{array}{cl}
\hphantom{-}0 & \text{if} ~ f(z)-\lambda \geqslant 0, ~~~ \forall z \in K, \\
- \infty & \text{else},
\end{array}
\right.
\end{equation}
since, in the second case, we may consider $t \delta_z$ for a $z\in K$ such that $f(z)-\lambda< 0$ and $t \rightarrow + \infty$.
This leads to the dual problem
\begin{equation}
\label{eq:sos}
\begin{array}{rcllll}
f^\text{opt} & = & \sup_{\lambda \in \mathbb{R}}  & \lambda & \mathrm{s.t.} & f(z)- \lambda \geqslant 0,~~ \forall z \in K.
\end{array}
\end{equation}
Primal problem \eqref{eq:moment} gives rise to the complex moment hierarchy in Section~\ref{subsec:Complex Moment Hierarchy}. Dual problem \eqref{eq:sos} gives rise to the complex sum-of-squares hierarchy in Section~\ref{subsec:Complex Sum-of-Squares Hierarchy}.

\subsection{Complex Moment Hierarchy} 
\label{subsec:Complex Moment Hierarchy}

Let $\mathcal{H}$ (respectively $\mathcal{H}_d$) denote the set of sequences of complex numbers $(y_{\alpha,\beta})_{\alpha,\beta \in \mathbb{N}^n}$ (respectively $(y_{\alpha,\beta})_{|\alpha|,|\beta| \leqslant d}$)
such that $\overline{y_{\alpha,\beta}} = y_{\beta,\alpha}$ for all $\alpha,\beta \in \mathbb{N}^n$ (respectively $|\alpha|,|\beta| \leqslant d$). An element $y \in \mathcal{H}$ is said to have a \textit{representing measure} $\mu$ on $K$ if $\mu \in \mathcal{M}_+(K)$ and $y_{\alpha,\beta} = \int_K \bar{z}^\alpha z^\beta d\mu$ for all $\alpha,\beta \in \mathbb{N}^n$.
When $y \in \mathcal{H}$ has a representing measure on $K$, the measure is unique because $R_K(\mathbb{C}[\bar{z},z])$ is dense in $C(K,\mathbb{C})$. The moment problem consists in characterizing the sequences that are representable by a measure on $K$.
For example, Atzmon~\cite[Theorem 2.1]{atzmon-1975} proved that when $K=\{ z \in \mathbb{C} ~|~ |z|=1\}$ the solutions are the sequences $y\in \mathcal{H}$ such that 
$\sum_{m,n,j,k \in \mathbb{N}} c_{n,j} ~ \overline{c}_{m,k} ~ y_{m+j,n+k} \geqslant 0$
and $\sum_{m,n \in \mathbb{N} } w_m \overline{w}_n ~ (y_{m,n} - y_{m+1,n+1}) \geqslant 0$
for all complex numbers $(c_{j,k})_{j,k \in \mathbb{N}}$ and $(w_n)_{n \in \mathbb{N}}$ with only finitely many non-zero terms. Theorem \ref{th:represent} below generalizes this result. 

Consider a feasible point $\mu$ of \eqref{eq:moment} and the sequence $y \in \mathcal{H}$ that has representation measure $\mu$ on $K$. Notice that 
$\int_K f d\mu = \int_K \sum_{|\alpha|,|\beta|\leqslant k} f_{\alpha,\beta} \bar{z}^\alpha z^\beta  d\mu =
\sum_{|\alpha|,|\beta|\leqslant k} f_{\alpha,\beta} \int_K \bar{z}^\alpha z^\beta d\mu 
= \sum_{|\alpha|,|\beta|\leqslant k} ~ f_{\alpha,\beta} y_{\alpha,\beta} =: L_y(f)$ and
$\int_K d\mu = \int_K \bar{z}^0 z^0 d\mu = y_{0,0} = 1$.
For all $p \in \mathbb{C}[z]$, we have $|p|^2 g_i \geqslant 0$ on $K$.  Since $\mu \geqslant 0$, this implies that $\int_K  |p|^2 g_i d\mu \geqslant 0$. Naturally, we also have $\int_K  |p|^2 g_0 d\mu \geqslant 0$ if we define $g_0 := 1$. Define $k_0 :=0$ and $d^\text{min}:=\max\{k, k_1 \hdots,k_m\}$. Consider $d\geqslant d^\text{min}$, $0 \leqslant i \leqslant m$, and $p \in \mathbb{C}_{d-k_i}[z]$. We have $\int_K |p|^2 g_i d\mu 
= \int_K  |\sum_{|\alpha|\leqslant d-k_i} p_{\alpha} z^{\alpha}|^2 ( \sum_{|\gamma|,|\delta| \leqslant k_i} g_{i,\gamma,\delta} \bar{z}^{\gamma} z^{\delta}) d\mu 
= \int_K (\sum_{|\alpha|,|\beta| \leqslant d-k_i} \overline{p}_{\alpha}  p_{\beta} \bar{z}^\alpha z^\beta) ( \sum_{|\gamma|,|\delta| \leqslant k_i} g_{i,\gamma,\delta} \bar{z}^{\gamma} z^{\delta}) d\mu 
= \int_K \sum_{|\alpha|,|\beta| \leqslant d-k_i} \overline{p}_\alpha  p_\beta$

\noindent $\sum_{|\gamma|,|\delta| \leqslant k_i}  g_{i,\gamma,\delta} \bar{z}^{\alpha+\gamma} z^{\beta+\delta} d\mu 
= \sum_{|\alpha|,|\beta| \leqslant d-k_i} \overline{p}_{\alpha}  p_{\beta} \sum_{|\gamma|,|\delta| \leqslant k_i}  g_{i,\gamma,\delta} \int_K \bar{z}^{\alpha+\gamma} z^{\beta+\delta} d\mu 
= \sum_{|\alpha|,|\beta| \leqslant d-k_i} \overline{p}_{\alpha}  p_{\beta} \sum_{|\gamma|,|\delta| \leqslant k_i}  g_{i,\gamma,\delta} ~ y_{\alpha+\gamma,\beta+\delta}
=: \sum_{|\alpha|,|\beta| \leqslant d-k_i} \overline{p}_{\alpha}  p_{\beta} M_{d-k_i}(g_iy)(\alpha,\beta)
= \vec{p}^H M_{d-k_i}(g_iy) \vec{p}$
where $\vec{p} := ( p_\alpha )_{|\alpha|\leqslant d-k_i}$ and $M_{d-k_i}(g_iy)$ is a Hermitian matrix indexed by $|\alpha|,|\beta| \leqslant d-k_i$.
To sum up, $y$ is a feasible point of 
\begin{equation}
\label{eq:momentinf}
\begin{array}{rclll}
\rho & := & \inf_{y\in \mathcal{H}} & L_y(f) & \\
 & & \text{s.t.} &  y_{0,0} = 1, &  \\
 & & &  M_{d-k_i}(g_iy) \succcurlyeq 0, &i = 0, \hdots, m, ~~ \forall d \geqslant d^\text{min},
\end{array}
\end{equation}
with same objective value as $\mu$ in \eqref{eq:moment}. Automatically, $\rho \leqslant f^\text{opt}$. Consider the relaxation of \eqref{eq:momentinf} defined by
\begin{equation}
\label{eq:momentd}
\boxed{
\begin{array}{rclll}
\rho_d & := & \inf_{y\in \mathcal{H}_d} & L_y(f)                                          & \\
           &     & \text{s.t.}                          & y_{0,0} = 1,                                 &  \\
           &     &                                         &  M_{d-k_i}(g_iy) \succcurlyeq 0, & i = 0, \hdots, m,
\end{array}
}
\end{equation}
which we name the \textit{complex moment relaxation of order} $d$ for reasons that will become clear with Theorem \ref{th:represent}. In Section \ref{subsec:Complex Sum-of-Squares Hierarchy}, we will introduce its dual counterpart.
\begin{remark}
\label{rem:Ly}
\normalfont
Given $y \in \mathcal{H}$, the function $L_y$ in this section can be formally be defined by the $\mathbb{C}$-linear operator $L_y : \mathbb{C}[\bar{z},z] \longrightarrow \mathbb{C}$ such that $L_y(\bar{z}^\alpha z^\beta) = y_{\alpha,\beta}$ for all $\alpha,\beta \in \mathbb{N}$ (i.e. Riesz functional). If $\varphi \in \mathbb{C}[\bar{z},z]$ and $\overline{\varphi} = \varphi$, then $\overline{L_y(\varphi)} = L_y(\varphi)$.
Given $l,d \in \mathbb{N}$ and $\varphi \in \mathbb{R}_{l}[\bar{z},z]$, the matrix $M_d(\varphi y)$ can be formally be defined as the Hermitian matrix indexed by $|\alpha|,|\beta| \leqslant d$ such that
$M_{d}(\varphi y)(\alpha,\beta) := L_y(\varphi(z) \bar{z}^\alpha z^\beta) = \sum_{|\gamma|,|\delta| \leqslant l}  ~ \varphi_{\gamma,\delta} ~ y_{\alpha+\gamma,\beta+\delta}$. Notice that $M_{d}(\varphi y)(0,0) = L_y(\varphi)$. Lastly, define $M_d(y) := M_d(g_0 y)$ which we refer to as the \textit{complex moment matrix of order} $d$.
\end{remark}

\subsection{Complex Sum-of-Squares Hierarchy} 
\label{subsec:Complex Sum-of-Squares Hierarchy} 
Given $l \in \mathbb{N}$ and $\varphi \in \mathbb{R}_l[\bar{z},z]$, define $\vec{\varphi} := ( \varphi_{\alpha,\beta})_{|\alpha|,|\beta|\leqslant l}$. This notation is well-defined due to the unicity of the coefficients of $\varphi$.\footnote{\label{fn:sos}The notation is ill-defined in the real case: if  $\varphi: x \in \mathbb{R}^n \longrightarrow \sum_{|\alpha|,|\beta|\leqslant l} \varphi_{\alpha,\beta} x^\alpha x^\beta \in \mathbb{R}$, then the coefficients $\varphi_{\alpha,\beta} \in \mathbb{R}$ are not unique. Thus $\sum_{|\alpha|\leqslant 2l} \sigma_\alpha x^\alpha$ is a real sum of squares if and only if there exists some real numbers $(\varphi_{\alpha,\beta})_{|\alpha|,|\beta|\leqslant l} \succcurlyeq 0$ such that $\sum_{|\alpha|\leqslant 2l} \sigma_\alpha x^\alpha = \sum_{|\alpha|,|\beta|\leqslant l} \varphi_{\alpha,\beta} x^\alpha x^\beta$.} Notice that $\varphi \in \Sigma_l[z]$ if and only if $\vec{\varphi} \succcurlyeq 0$. Also, define $\langle A , B \rangle_{\mathcal{H}_d} := \text{Tr} (AB)$ where $A,B \in \mathcal{H}_d$.
Given $d\geqslant d^\text{min}$, consider the Lagrange function
$\mathcal{L}_d : \mathcal{H}_d \times \mathbb{R} \times \Sigma_{d-k_0}[z] \times \hdots \times \Sigma_{d-k_m}[z] \longrightarrow \mathbb{R}$ defined by
$( y , \lambda, \sigma_0, \hdots , \sigma_m ) \longmapsto L_y(f) + \lambda(1-y_{0,0}) - \sum_{i=0}^m \langle M_{d-k_i}(g_iy) , \vec{\sigma}_i \rangle_{\mathcal{H}_{d-k_i}}$. Given $\sigma_i =: \sum_{j=1}^{r_i} |p_j^i|^2$, i.e. $\vec{\sigma_i} = \sum_{j=1}^{r_i} \vec{p_j}^i (\vec{p_j}^i)^H$, compute $\mathcal{L}_d(y , \lambda, \sigma_0, \hdots , \sigma_m) = \lambda + L_y(f - \lambda) - \sum_{i=0}^m \sum_{j=0}^{r_i} (\vec{p}_j^{\hphantom{.}i})^H M_{d-k_i}(g_iy) \vec{p}_j^{\hphantom{.}i} = \lambda + L_y(f - \lambda) -  \sum_{i=0}^m \sum_{j=0}^{r_i} L_y(|p_j^i|^2 g_i) = \lambda + L_y(f - \lambda - \sum_{i=0}^m \sigma_i g_i)$.
Observe that 
\begin{equation}
\inf_{y \in \mathcal{H}} ~~ L_y\left(f - \lambda - \sum_{i=0}^m \sigma_i g_i\right) =
\left\{
\begin{array}{cl}
\hphantom{-}0 & \text{if} ~ f(z) - \lambda - \sum_{i=0}^m \sigma_i(z) g_i(z) = 0, \\
& \text{for all} ~ z \in \mathbb{C}^n, \\
- \infty & \text{else}.
\end{array}
\right.
\end{equation}
Indeed, in the second case, there exists $z\in \mathbb{C}^n$ such that $f(z) - \lambda - \sum_{i=0}^m \sigma_i(z) g_i(z) \neq 0$. With $(y_{\alpha,\beta})_{\alpha,\beta \in \mathbb{N}} := (\bar{z}^\alpha z^\beta)_{\alpha,\beta \in \mathbb{N}}$, $L_{ty}(f - \lambda - \sum_{i=0}^m \sigma_i g_i) \longrightarrow - \infty$ for either $t\longrightarrow -\infty$ or $t\longrightarrow +\infty$.
The associated dual problem of \eqref{eq:momentd} is thus
\begin{equation}
\label{eq:sosd}
\boxed{
\begin{array}{rcll}
\rho_d^* &:= & \sup_{\lambda,\sigma} & \lambda \\
              &    & \text{s.t.} &  f-\lambda = \sum_{i=0}^m \sigma_i g_i, \\ 
              &    &                &  \lambda \in \mathbb{R},~ \sigma_i \in \Sigma_{d-k_i}[z] ,~ i=0, \hdots, m,
\end{array}
}
\end{equation}
which we name the \textit{complex sum-of-squares relaxation of order} $d$. Consider
\begin{equation}
\label{eq:sosinf}
\begin{array}{rcll}
\rho^* &:= & \sup_{\lambda,\sigma} & \lambda \\
              &    & \text{s.t.} &  f-\lambda = \sum_{i=0}^m \sigma_i g_i, \\ 
              &    &                & \lambda \in \mathbb{R},~ \sigma_i \in \Sigma[z] ,~ i=0, \hdots, m.
\end{array}
\end{equation}
\begin{proposition}
\label{prop:dualconverge}
\normalfont
\textit{We have} $\rho_d^* \leqslant \rho_d$ for all $d\geqslant d^\text{min}$ and $\rho_d^* \longrightarrow \rho^* \leqslant \rho \leqslant f^\text{opt}$.
\end{proposition}
\begin{proof}
The sequence $(\rho_d^*)_{d\geqslant d^\text{min}}$ is non-decreasing and upper bounded by $\rho^* \in \mathbb{R}\cup\{\pm \infty\}$. Thus it converges towards some limit $\rho^*_\text{lim} \in \mathbb{R}\cup\{\pm \infty\}$ such that $\rho^*_\text{lim} \leqslant \rho^*$. If $\rho^* = - \infty$, then $\rho_d^* = - \infty$ for all $d\geqslant d^\text{min}$ and $\rho_d^* \longrightarrow \rho^*$. If not, by definiton of the optimum $\rho^*$, there exists a sequence $(\lambda^l,\sigma_0^l,\hdots,\sigma_m^l)$ of feasible points such that $\lambda^l \leqslant \rho^*$ and $\lambda^l \longrightarrow \rho^*$. To each $l\in\mathbb{N}$, we may associate an integer $d(l)\in\mathbb{N}$ such that $(\lambda^l,\sigma_0^l,\hdots,\sigma_m^l)$ is a feasible point of the complex sum-of-squares relaxation of order $d(l)$. Thus $\lambda^l \leqslant \rho_{d(l)}^* \leqslant \rho^*$. As a result, $\rho_\text{limit}^* = \rho^*$. Moreover, $(\rho_d)_{d\geqslant d^\text{min}}$ is non-decreasing and upper bounded by $\rho \in \mathbb{R}\cup\{\pm \infty\}$. Thus it converges towards some limit $\rho_\text{lim} \in \mathbb{R}\cup\{\pm \infty\}$ such that $\rho_\text{lim} \leqslant \rho$. Moreover, weak duality implies that $\rho_d^* \leqslant \rho_d ~ (\leqslant \rho)$. Thus $\rho^* \leqslant \rho_\text{lim} \leqslant \rho$. It was shown in Section \ref{subsec:Complex Moment Hierarchy} that $\rho \leqslant f^\text{opt}$.
\end{proof}
\begin{remark}
\label{rem:duality}
\normalfont
Problems \eqref{eq:sosinf} and \eqref{eq:momentinf} may be interpreted as a pair of primal-dual linear programs in infinite-dimensional spaces~\cite{anderson-1987}. Consider the duality bracket $\langle .,. \rangle$ defined from $\mathbb{R}[\bar{z},z] \times \mathcal{H}$ to $\mathbb{R}$ by $\langle \varphi , y \rangle := L_y(\varphi)$. A sequence $(\varphi^n)_{n\in \mathbb{N}}$ in $\mathbb{R}[\bar{z},z]$ is said to converge weakly towards $\varphi \in \mathbb{R}[\bar{z},z]$ if for all $y \in \mathcal{H}$, we have $\langle \varphi^n , y \rangle \longrightarrow \langle \varphi , y \rangle$. Consider the weakly continuous $\mathbb{R}$-linear operator $A : \mathbb{R}[\bar{z},z] \longrightarrow \mathbb{R}[\bar{z},z]$ defined by $\varphi \longmapsto \varphi - \varphi_{0,0}$. Its dual $A^* : \mathcal{H} \longrightarrow \mathcal{H}$ is defined by $y \longmapsto y - y_{0,0} \delta_{0,0}$ where $(\delta_{0,0})_{0,0} = 1$ and $(\delta_{0,0})_{\alpha,\beta} = 0$ if $(\alpha,\beta)\neq(0,0)$. Indeed, $\langle A\varphi , y \rangle = \langle \varphi , A^* y \rangle$ for all $(\varphi,y) \in \mathbb{R}[\bar{z},z] \times \mathcal{H}$. Consider the convex pointed cone defined by $C:= \Sigma[z] g_0 + \hdots + \Sigma[z]g_m$ and its dual cone $C^* := \{ y \in \mathcal{H} ~ | ~ \forall \varphi \in C, ~ \langle \varphi , y \rangle \geqslant 0 \}$. If $b:=Af$, then 
\begin{equation}
\label{eq:pd}
\begin{array}{rcllll}
f_{0,0} - \rho^* & = & \inf_{\varphi \in \mathbb{R}[\bar{z},z]} & \langle \varphi , \delta_{0,0} \rangle & \text{s.t.} & A\varphi = b  ~~ \& ~~ \varphi \in C, \\
f_{0,0} - \rho\hphantom{^*}  & = & \sup_{y \in \mathcal{H}} & \langle b , y \rangle & \text{s.t.} & \delta_{0,0} - A^*y \in C^*.
\end{array}
\end{equation}
Let $\text{cl}(C)$ denote the weak closure of $C$ in $\mathbb{R}[\bar{z},z]$. \cite[5.91 Bipolar Theorem]{aliprantis-1999}\footnote{We wish to thank Jean-Bernard Baillon for bringing this reference to our attention.} implies that $\text{cl}(C) = C^{**}$. Below, Theorem \ref{th:angelo} and Theorem \ref{th:represent} provide a sufficient condition ensuring no duality gap in \eqref{eq:pd} and $\text{cl}(C) = \{ \varphi \in \mathbb{R}[\bar{z},z] ~|~ \varphi_{|K} \geqslant 0 \}$ respectively.
\end{remark}

\subsection{Convergence of the Complex Hierarchy} 
\label{subsec:Convergence of the Complex Hierarchy} 
We turn our attention to a result from algebraic geometry discovered in 2008.
\begin{theorem}[D'Angelo's and Putinar's Positivstellenstatz~\cite{angelo-2008}]
\label{th:angelo}
If one of the constraints that define $K$ is a sphere constraint $|z_1|^2 + \hdots + |z_n|^2 = 1$, and if $f_{|K} > 0$, then there exists $\sigma_0, \hdots, \sigma_m \in \Sigma[z]$ such that $f = \sum_{i=0}^m \sigma_i g_i$.
\end{theorem}
\begin{proof}
D'Angelo and Putinar wrote the theorem slightly differently. Say that constraints $g_{m-1}$ and $g_m$ are such that $g_{m-1}=s$ and $g_m=-s$ where $s(z):= 1 - |z_1|^2 - \hdots - |z_n|^2$. With the assumptions of Theorem \ref{th:angelo}, the authors of \cite[Theorem 3.1]{angelo-2008} show that there exists $\sigma_0, \hdots, \sigma_{m-2} \in \Sigma[z]$ and $r \in \mathbb{R}[\bar{z},z]$ such that
$f(z) =\sum_{i=0}^{m-2} \sigma_i(z) g_i(z) +  r(z)s(z)$ for all $z\in \mathbb{C}^n$. Thanks to~\cite[Proposition 1.2]{angelo-2010}, there exists $\sigma_{m-1},\sigma_m \in \Sigma[z]$ such that $r = \sigma_{m-1} - \sigma_m$ hence the desired result. 
\end{proof}

Theorem \ref{th:angelo} can easily be generalized to any sphere $|z_1|^2 + \hdots + |z_n|^2 = R^2$ of radius $R > 0$. With scaled variable $w = \frac{z}{R} \in \mathbb{C}^n$, the sphere constraint has radius 1 and a monomial of \eqref{eq:complexPOP} with coefficient $c_{\alpha,\beta}\in \mathbb{C}$  reads $c_{\alpha,\beta} \bar{z}^\alpha z^\beta = c_{\alpha,\beta} (R\overline{w})^\alpha (Rw)^\beta = R^{|\alpha|+|\beta|} c_{\alpha,\beta} \overline{w}^\alpha w^\beta$. With the scaled coefficients $R^{|\alpha|+|\beta|} c_{\alpha,\beta}$, Theorem \ref{th:angelo} can then be applied. Reverting back to the old scale $z = Rw$ 
leads to the desired result. Accordingly, we define the following statement which is true only when stated:
\begin{equation}
\label{eq:sa}
\textbf{Sphere Assumption:} ~~~
\boxed{
\begin{array}{l}
\text{One of the constraints of \eqref{eq:complexPOP} is a sphere} \\
|z_1|^2 + \hdots + |z_n|^2 = R^2 ~ \text{for some} ~ R > 0.
\end{array}
}
\end{equation}
\begin{corollary}
\label{cor:convergence}
Under the sphere assumption \eqref{eq:sa}, $\rho_d^* \rightarrow f^\text{opt}$ and $\rho_d \rightarrow f^\text{opt}$.
\end{corollary}
\begin{proof} Theorem \ref{th:angelo} implies that $\rho^* = f^\text{opt}$
because for all $\epsilon > 0$, function \mbox{$f-(f^\text{opt}-\epsilon)$} is positive on $K$. The sequences $(\rho_d^*)_{d\geqslant d^\text{min}}$ and $(\rho_d)_{d\geqslant d^\text{min}}$ converge towards $f^\text{opt}$ due to Proposition \ref{prop:dualconverge}.
\end{proof}

To require a sphere constraint in a complex polynomial optimization problem seems very restrictive and irrelevant for many problems. But in fact, a sphere constraint can be applied to any complex polynomial optimization problem \eqref{eq:complexPOP} with a feasible set contained in a ball $|z_1|^2 + \hdots + |z_{n}|^2 \leqslant R^2$ of known radius $R > 0$. Indeed, simply add a slack variable $z_{n+1} \in \mathbb{C}$ and the constraint $|z_1|^2 + \hdots + |z_{n+1}|^2 = R^2$.
Let $\hat{K}$ denote the feasible set of the problem in $n+1$ variables.
If $(z_1,\hdots,z_{n+1}) \in \hat{K}$, then $(z_1,\hdots,z_n) \in K$  and has the same objective value. Conversly, if $(z_1,\hdots,z_n) \in K$, then $(z_1,\hdots,z_{n+1}) \in \hat{K}$ for all $z_{n+1} \in \mathbb{C}$ such that $|z_{n+1}|^2 = R^2 - |z_1|^2  \hdots - |z_n|^2$. Again, the objective value is unchanged.
To ensure a bijection between $K$ and $\hat{K}$, add yet two more constraints $\textbf{i}z_{n+1}  - \textbf{i}\overline{z}_{n+1} =  0$ and $z_{n+1} + \overline{z}_{n+1} \geqslant 0$, thereby preserving the number of global solutions. In that case, the application from $K$ to $\hat{K}$ defined by $(z_1,\hdots,z_n) \longmapsto (z_1,\hdots,z_n,\sqrt{R^2-|z_1|^2-\hdots-|z_n|^2})$ is a bijection. Adding the two extra constraints is optional and not required for convergence of optimal values. 

As seen in Theorem \ref{th:angelo}, an equality constraint may be enforced via two opposite inequality constraints. Let $h_1,\hdots,h_e$ denote $e\in \mathbb{N}^*$ equality constraints in polynomial optimization problem \eqref{eq:complexPOP}.
Putinar and Scheiderer~\cite[Propositions 6.6 and 3.2 (iii)]{putinar-2013} show that the sphere assumption in D'Angelo's and Putinar's Positivstellensatz may be weakened to the existence of $r_1,\hdots,r_e \in \mathbb{R}[\bar{z},z]$, $\sigma \in \Sigma[z]$, and $a\in \mathbb{R}$ such that
\begin{equation}
\label{eq:weak}
\sum_{j=1}^e r_j(z) h_j(z) = \sum_{i=1}^n |z_i|^2 + \sigma(z)+a, ~~~~~~ \forall z\in \mathbb{C}^n.
\end{equation}
If the constraints include $|z_1|^2 - 1 = \hdots = |z_n|^2 - 1 = 0$, the assumption is satisfied by $r_1 = \hdots = r_n = 1$, $\sigma = 0$ and $a=-n$. In particular, there is no need to add a slack variable in the non-bipartite Grothendieck problem over the complex numbers~\cite{bandeira-2014}.
\begin{example}
\label{ex:angelo}
\normalfont
D'Angelo and Putinar~\cite{angelo-2008} consider $\frac{1}{3} < a < \frac{4}{9}$ and problem
\begin{equation}
\label{eq:ex1}
\begin{array}{llcl}
 \inf_{z \in \mathbb{C}} & f(z) & := & 1-\frac{4}{3}|z|^2+a|z|^4 \\
            \text{s.t.}          & g(z) & := & 1 - |z|^2 \geqslant 0,
\end{array}
\end{equation}
whose set of global solutions is $K^\text{opt} = \{ z \in \mathbb{C}~|~ |z| =1\}$ and $f^\text{opt} = a-\frac{1}{3}>0$. They prove that the decomposition $f = \sigma_0 + \sigma_1g ~ (\sigma_0,\sigma_1 \in \Sigma[z])$ of Theorem \ref{th:angelo} does not hold. As a result, the optimal values of the complex sum-of-squares relaxations cannot exceed 0 even though $f^\text{opt}>0$. Indeed, if $\rho_d^* > 0$ for some order $d\geqslant d^\text{min}$, then there exists $\lambda \geqslant \frac{\rho_d^*}{2}$ and $\sigma_0,\sigma_1 \in \Sigma_d[z]$ such that $f - \lambda = \sigma_0 + \sigma_1 g$. Thus $f = \lambda + \sigma_0 + \sigma_1 g$ where $\lambda + \sigma_0 \in \Sigma_d[z]$, which is a contradiction. We suggest solving
\begin{equation}
\label{eq:ex1sphere}
\begin{array}{llcl}
 \inf_{z_1,z_2 \in \mathbb{C}} & \hat{f}(z_1,z_2) & := & 1-\frac{4}{3}|z_1|^2+a|z_1|^4 \\
            \text{s.t.}          & \hat{g}(z_1,z_2) & := & 1 - |z_1|^2 - |z_2|^2 = 0.
\end{array}
\end{equation}
For all $\lambda < f^\text{opt}$, there exists $\hat{\sigma}_0 \in \Sigma[z_1,z_2]$ and $\hat{r}\in \mathbb{R}[\overline{z}_1,\overline{z}_2,z_1,z_2]$ such that $\hat{f}(z_1,z_2) - \lambda = \hat{\sigma}_0(z_1,z_2) + \hat{r}(z_1,z_2) \hat{g}(z_1,z_2)$ for all $z_1,z_2 \in \mathbb{C}$.
Plug in $z_1=z$ and $z_2=0$ and obtain $ f(z) - \lambda = \hat{\sigma}_0(z,0) + \hat{r}(z,0) g(z)$ for all $z \in \mathbb{C}$.
While function $z \longmapsto \hat{\sigma}_0(z,0)$ belongs to $\Sigma[z]$, function $z \longmapsto \hat{r}(z,0)$ does not! Hence we do not contradict the fact that $f = \sigma_0 + \sigma_1g ~ (\sigma_0,\sigma_1 \in \Sigma[z])$ is impossible. Consider $a = \frac{1}{2}(\frac{1}{3} + \frac{4}{9}) = \frac{7}{18}$ so that $f^\text{opt} = \frac{1}{18}$. Notice that $d^\text{min} = 2$ for \eqref{eq:ex1} and \eqref{eq:ex1sphere}. The complex relaxations of orders $2 \leqslant d \leqslant 3$ of \eqref{eq:ex1} yield\footnote{MATLAB 2013a, YALMIP \mbox{2015.06.26}~\cite{yalmip}, and MOSEK are used for the numerical experiments.} the value $-0.3333$. The complex relaxation of order 2 of \eqref{eq:ex1sphere} yields the value $0.0556 ~ (\approx f^\text{opt})$ and optimal polynomials $\hat{\sigma}_0(z_1,z_2) = 0.2780 |z_2|^2 + 0.2776 |z_1 z_2|^2 + 0.6667 |z_2|^4$ and $\hat{r}(z_1,z_2) = 0.9444 - 0.3889|z_1|^2 + 0.6665 |z_2|^2$.
\end{example}

\begin{proposition}
\label{prop:rank}
\normalfont
\textit{Assume that the sphere assumption~\eqref{eq:sa} holds, that $n>1$, and that $y \in \mathcal{H}_d$ is an optimal solution to the complex moment relaxation of order} 
$d \geqslant d^\text{min}$. 
\textit{With $d_K := \max_{1 \leqslant i \leqslant m} k_i$} (\textit{$k_i$ is defined above}~\eqref{eq:complexPOP})\textit{ and $d^\text{min} \leqslant t \leqslant d$, if}
\begin{enumerate}
\item $\text{rank} ~ M_{t} (y) = \text{rank} ~ M_{t-d_K} (y) ~ (=: S)$,
\item $ \begin{pmatrix} 
M_{t-d_K}(y) & M_{t-d_K}(\bar{z}_i y) & M_{t-d_K}(\bar{z}_j y) \\ 
M_{t-d_K}(z_i y) & M_{t-d_K}(|z_i|^2 y)  & M_{t-d_K}(\bar{z}_j z_i y) \\
M_{t-d_K}(z_j y) & M_{t-d_K}(\bar{z}_i z_j y) & M_{t-d_K}(|z_j|^2 y)
\end{pmatrix} 
\succcurlyeq 0,
$
\textit{for all} $1 \leqslant i < j \leqslant n$,
\end{enumerate}
\textit{then $\rho_d = f^\text{opt}$ and complex polynomial problem \eqref{eq:complexPOP} has at least $S$ global solutions.}
\end{proposition}
\begin{proof}
Thanks to Theorem~\ref{th:trunc} below, $y \in \mathcal{H}_t$ can be represented by a measure $\mu$ on $K$ (i.e. $y_{\alpha,\beta} = \int_K \bar{z}^\alpha z^\beta d\mu, ~\forall |\alpha|,|\beta| \leqslant t$) and can thus be extended to $y \in \mathcal{H}$. The same theorem implies that $\mu = \sum_{j=1}^S \lambda_j \delta_{z(j)}$ for some $S$ different point $z(1),\hdots,z(S)$ in $K$ and some $\lambda_1,\hdots,\lambda_S > 0$. In addition, $y_{0,0} = \int_K \bar{z}^0 z^0 d\mu = \sum_{j=1}^S \lambda_j = 1$ and thus $f^\text{opt} \geqslant \rho_d = L_y(f) = \int_K f d\mu = \sum_{j=1}^S \lambda_j f(z(j)) \geqslant \sum_{j=1}^S \lambda_j f^\text{opt} = f^\text{opt}$. We simultaneously deduce that $ \rho_d = f^\text{opt} = f(z(1)) = \hdots = f(z(S))$. 
\end{proof}

In particular, if $S=1$ in Proposition \ref{prop:rank}, then Point 2 in Proposition~\ref{prop:rank} need not be checked for (see comment under \eqref{eq:uni}) and $y_{\alpha,\beta} = \int_K \bar{z}^\alpha z^\beta d\delta_z =  \bar{z}^\alpha z^\beta, ~\forall  |\alpha|,|\beta| \leqslant d^\text{min}$ for some $z \in K^\text{opt}$. A global solution can be read from $y$ because $z = (y_{0,\beta})_{|\beta|=1}$. 

\begin{example}
\label{ex:putinar}
\normalfont
Putinar and Scheiderer~\cite{putinar-scheiderer-2012} consider parameters $0 < a < \frac{1}{2}$ and $C>\frac{1}{1-2a}$, and problem
\begin{equation}
\label{eq:ex2}
\begin{array}{llcl}
 \inf_{z \in \mathbb{C}} & f(z) & := & C - |z|^2 \\
            \text{s.t.}          & g(z) & := & |z|^2 - a z^2 - a \bar{z}^2 - 1 = 0,
\end{array}
\end{equation}
whose set of global solutions is $K^\text{opt} = \left\{ \pm \frac{1}{\sqrt{1-2a}} \right\}$ and $f^\text{opt} = C-\frac{1}{1-2a}>0$. They prove that the decomposition of Theorem \ref{th:angelo} does not hold. Since the feasible set is included in the Euclidean ball of radius $\sqrt{C}$, we suggest solving
\begin{equation}
\label{eq:ex2sphere}
\begin{array}{llcl}
 \inf_{z_1,z_2 \in \mathbb{C}} & \hat{f}(z_1,z_2) & := & C - |z_1|^2 \\
            \text{s.t.}          & \hat{g}_1(z_1,z_2) & := & |z_1|^2 - a z_1^2 - a \bar{z}_1^2 - 1 = 0, \\
                                    & \hat{g}_2(z_1,z_2) & := & C - |z_1|^2 - |z_2|^2 = 0, \\
                                    & \hat{g}_3(z_1,z_2) & := & \textbf{i} z_2 - \textbf{i} \overline{z}_2 = 0, \\
                                    & \hat{g}_4(z_1,z_2) & := & z_2 + \overline{z}_2 \geqslant 0.
\end{array}
\end{equation}
Consider $a = \frac{1}{4}$ and $C = 3$ so that $f^\text{opt} = 1$. Notice that $d^\text{min} = 2$ for \eqref{eq:ex2} and \eqref{eq:ex2sphere}. The complex relaxations of orders $2\leqslant d \leqslant 3$ of \eqref{eq:ex2} are unbounded. The complex relaxation of order 2 of \eqref{eq:ex2sphere} yields the value 0.6813. That of order 3 yields 1.0000, $\text{rank} ~ M_3(y) = \text{rank} ~ M_1(y) = 2$, and Point 2 in Proposition~\ref{prop:rank}. Thus $f^\text{opt} \approx 1.000$ and there exists at least 2 global solutions to \eqref{eq:ex2sphere}, and hence to \eqref{eq:ex2}.
\end{example}

We next transpose~\cite[Lemma 3]{josz-2015} from real to complex numbers.
\begin{lemma}
\normalfont
\label{lemma:sphere}
Define $s(z):= R^2 - |z_1|^2 - \hdots - |z_n|^2$. Given $d \in \mathbb{N}^*$ and $y\in \mathcal{H}_d$, if $M_d(y) \succcurlyeq 0$ and $M_{d-1}(sy) = 0$, then $\text{Tr} (M_{d}(y)) \leqslant y_{0,0} \sum_{l=0}^d R^{2l}$.
\end{lemma}
\begin{proof}
Given $1 \leqslant l \leqslant d$, we have $\text{Tr}(M_{l-1}(sy)) = \sum_{|\alpha|\leqslant l-1}  M_{l-1}(sy)(\alpha,\alpha)  = \sum_{|\alpha|\leqslant l-1}  L_y(s(z) \bar{z}^\alpha z^\alpha) =  \sum_{|\alpha|\leqslant l-1} \sum_{|\gamma| \leqslant 1} s_{\gamma,\gamma}  y_{\gamma+\alpha,\gamma+\alpha} = \sum_{|\alpha|\leqslant l-1,|\gamma| = 0}  s_{\gamma,\gamma}  y_{\gamma+\alpha,\gamma+\alpha}$  $+ \sum_{|\alpha|\leqslant l-1,|\gamma| = 1}  s_{\gamma,\gamma}  y_{\gamma+\alpha,\gamma+\alpha} = \sum_{|\alpha|\leqslant l-1} R^2  y_{\alpha,\alpha} - \sum_{|\alpha|\leqslant l-1,|\gamma| = 1} y_{\gamma+\alpha,\gamma+\alpha}$.
We have $M_{d-1}(sy) = 0$ so $M_{l-1}(sy) = 0$ for all $1 \leqslant l \leqslant d$ and hence $ \text{Tr}(M_{l-1}(sy)) = 0 $. In addition, $\sum_{0<|\alpha|\leqslant l}  y_{\alpha,\alpha} \leqslant \sum_{|\alpha|\leqslant l-1,|\gamma| = 1} y_{\gamma+\alpha,\gamma+\alpha}$. Thus $\sum_{|\alpha|\leqslant l}  y_{\alpha,\alpha}  \leqslant y_{0,0} + R^2  \sum_{|\alpha|\leqslant l-1} y_{\alpha,\alpha}$ for $1 \leqslant l \leqslant d$,
which proves the lemma.
\end{proof}

\begin{theorem}[Putinar and Scheiderer~\cite{putinar-2013}]
\label{th:represent} 
Under assumption \eqref{eq:sa}, $y\in \mathcal{H}$ has a representing measure on $K$ if and only if $M_{d}(g_iy) \succcurlyeq 0,~i = 0,\hdots,m, \forall d \in \mathbb{N}$.
\end{theorem}
\begin{proof} We provide an alternative proof using Lemma~\ref{lemma:sphere}.
The ``only if'' part is a consequence of Section \ref{subsec:Complex Moment Hierarchy}.
Concerning the ``if'' part, if $y_{0,0} = 0$, then Lemma \ref{lemma:sphere} implies that $y = 0$ which can be represented by $\mu = 0$ on $K$. Otherwise $y_{0,0}>0$ and $y/y_{0,0}$ is a feasible point of problem \eqref{eq:momentinf} whose optimal value is $f^\text{opt}$ for all $f \in \mathbb{R}[\bar{z},z]$ according to Corollary \ref{cor:convergence}. If moreover $f_{|K} \geqslant 0$, then $L_{ y/y_{0,0} } (f) \geqslant f^\text{opt} \geqslant 0$. In particular, if $f_{|K} = 0$, then $L_{ y/y_{0,0} } (f) = 0$. We may therefore define $\tilde{L}_{ y/y_{0,0} } : R_K( \mathbb{C}[\bar{z},z] ) \longrightarrow \mathbb{C}$ such that $\tilde{L}_{ y/y_{0,0} } ( \varphi_{|K} ) :=  L_{ y/y_{0,0} } (\varphi)$ (similar to Schweigh\"ofer~\cite[Proof of Theorem 2]{schweighofer-2005}). If $\varphi \in R_K(\mathbb{R}[\bar{z},z])$, then $\tilde{L}_{ y/y_{0,0} }(\|\varphi\|_\infty - \varphi) \geqslant 0$ and $\tilde{L}_{ y/y_{0,0} }(\varphi) \leqslant \|\varphi\|_\infty$. Linearity implies that $|\tilde{L}_{ y/y_{0,0} }(\varphi)| \leqslant \|\varphi\|_\infty$. 
As a result, for all $\varphi \in R_K(\mathbb{C}[\bar{z},z])$, we have
$ |\tilde{L}_{ y/y_{0,0} }(\varphi)| = | \tilde{L}_{ y/y_{0,0} }(\text{Re}(\varphi) + \textbf{i} \text{Im}(\varphi)) | = | \tilde{L}_{ y/y_{0,0} }(\text{Re}(\varphi)) + \textbf{i} \tilde{L}_{ y/y_{0,0} }(\text{Im}(\varphi)) | \leqslant | \tilde{L}_{ y/y_{0,0} }(\text{Re}(\varphi)) | + | \tilde{L}_{ y/y_{0,0} }(\text{Im}(\varphi)) | \leqslant \| \text{Re}(\varphi) \|_\infty + \| \text{Im}(\varphi) \|_\infty 
\leqslant 2 \| \varphi \|_\infty
$.
Moreover, $R_K(\mathbb{C}[\bar{z},z])$ is dense in $C(K,\mathbb{C})$. Therefore $\tilde{L}_{y/y_{0,0}}$ may be extended to a continous linear functional on $C(K,\mathbb{C})$ (we preserve the same name for the extension). $K$ is compact thus the Riesz Representation Theorem implies that there exists a unique Radon measure $\mu$ such that
$\tilde{L}_{y/y_{0,0}}(\varphi) = \int_K \varphi d\mu$ for all $\varphi \in C(K,\mathbb{C})$ and $\mu \geqslant 0$ because $\varphi \in \mathcal{P}(K)$ implies that $\tilde{L}_{ y/y_{0,0} }(\varphi) \geqslant 0$ (density argument). Finally, if $\alpha,\beta \in \mathbb{N}^n$, $y_{\alpha,\beta} / y_{0,0}  = L_{y/y_{0,0}} ( \bar{z}^\alpha z^\beta )$ (Remark \ref{rem:Ly}) so $y$ has representing measure $y_{0,0}\mu$ on $K$.
\end{proof}
\begin{theorem}
\label{th:trunc}
\normalfont
\textit{
Let $n>1$ and $y \in \mathcal{H}_d$ with $d \geqslant d_K = \max_{1 \leqslant i \leqslant m} k_i$} (\textit{$k_i$ is defined above~\eqref{eq:complexPOP}})\textit{. Assume that $K$ contains the constraints $|z_k|^2 \leqq R_k^2, ~ k=1\hdots n$, for some radii $R_k \geqslant 0$ or the constraint $\sum_{k=1}^n |z_k|^2 \leqq R^2$ for some radius $R \geqslant 0$} (\textit{where $\leqq$ is an equality or an inequality}).
\textit{Then there exists a positive} $ \text{rank} M_{d-d_K}(y)$\textit{-atomic measure} $\mu$ \textit{supported on} $K$ \textit{such that:}
\begin{equation}
\label{eq:atomic_rep}
 y_{\alpha,\beta} = \int_{\mathbb{C}^n} \bar{z}^\alpha z^\beta d\mu ~, ~~~~ \text{for all} ~~ |\alpha|,|\beta| \leqslant d
\end{equation}
\textit{if and only if:}
\begin{enumerate}
\item $M_d(y) \succcurlyeq 0$ and $M_{d-k_i}(g_iy) \succcurlyeq 0, ~ i = 1 \hdots m$;
\item $\text{rank} M_{d}(y) = \text{rank} M_{d-d_K}(y)$;
\item $ 
\begin{pmatrix} 
M_{d-d_K}(y) & M_{d-d_K}(\bar{z}_i y) & M_{d-d_K}(\bar{z}_j y) \\ 
M_{d-d_K}(z_i y) & M_{d-d_K}(|z_i|^2 y)  & M_{d-d_K}(\bar{z}_j z_i y) \\
M_{d-d_K}(z_j y) & M_{d-d_K}(\bar{z}_i z_j y) & M_{d-d_K}(|z_j|^2 y)
\end{pmatrix} 
\succcurlyeq 0 , ~ \forall 1 \leqslant i < j \leqslant n$.
\end{enumerate}
\textit{Moreover, for each $1\leqslant i \leqslant m$, the measure $\mu$ has exactly} $\text{rank} M_{d}(y) - \text{rank} M_{d-d_K}(g_i y)$ \textit{atoms that are zeros of $g_i$.}
\end{theorem}
\begin{proof}
($\Longleftarrow$) Point 1 implies that $(y_{\alpha,\beta})_{|\alpha|,|\beta|\leqslant d} \succcurlyeq 0$. Thus there exists a complex matrix $x$ of the same size as $(y_{\alpha,\beta})_{|\alpha|,|\beta|\leqslant d}$ such that we have the Cholesky factorization $(y_{\alpha,\beta})_{|\alpha|,|\beta|\leqslant d} = x^H x$. Let $(x_\alpha)_{|\alpha| \leqslant d}$ denote the columns of $x$. Also, let $\mathcal{C}_d$ denote the column space and consider the inner product $\langle u , v \rangle_{\mathcal{C}_d} := u^H v$ and its induced norm $\| . \|_{\mathcal{C}_d}$. We have $y_{\alpha,\beta} = \langle x_\alpha , x_\beta \rangle_{\mathcal{C}_d}$ for all $|\alpha|,|\beta|\leqslant d$. Let $V := \text{span} (x_\alpha)_{|\alpha| \leqslant d} \subset \mathcal{C}_d$. Point 2 implies that $V = \text{span} (x_\alpha)_{|\alpha| \leqslant d-1}$. Given $1\leqslant k \leqslant n$, define the $\mathbb{C}$-linear operator $T_k : V \longrightarrow V$ such that $T_k x_\alpha = x_{\alpha+e_k}$ for all $|\alpha| \leqslant d-1$ where $e_k$ is the row vector of size $n$ that contains only zeros apart from 1 in position $k$. This shift operator is well defined because each element of $V$ has a unique image by $T_k$. Indeed, consider some complex numbers $(u_\alpha)_{|\alpha| \leqslant d-1}$. The assumption on $K$ in the case of multiple constraints and Point 1 imply that $M_{d-1}[(R_k^2-|z_k|^2)y] \succcurlyeq 0$. Thus
$ \| \sum_{|\alpha| \leqslant d-1} u_\alpha x_{\alpha+e_k} \|_{\mathcal{C}_d}^2 = \sum_{|\alpha|,|\beta|\leqslant d-1} \langle x_{\alpha+e_k} , x_{\beta+e_k} \rangle_{\mathcal{C}_d} \bar{u}_\alpha u_\beta = \sum_{|\alpha|,|\beta|\leqslant d-1} y_{\alpha+e_k,\beta+e_k}  \bar{u}_\alpha u_\beta \leqslant R_k^2 \sum_{|\alpha|,|\beta|\leqslant d-1} y_{\alpha,\beta} \bar{u}_\alpha u_\beta = R_k^2 \sum_{|\alpha|,|\beta|\leqslant d-1} \langle x_\alpha , x_\beta \rangle_{\mathcal{C}_d}$\\$ \bar{u}_\alpha u_\beta = R_k^2 \| \sum_{|\alpha| \leqslant d-1} u_\alpha x_{\alpha} \|_{\mathcal{C}_d}^2$. Thus $T_k$ is well-defined and bounded by $R_k$. The assumption on $K$ in the case of a single constraint and Point 1 imply that $M_{d-1}[(R^2-\sum_{j=1}^n |z_j|^2)y] \succcurlyeq 0$. Thus $\|\sum_{|\alpha| \leqslant d-1} u_\alpha x_{\alpha+e_k}\|_{\mathcal{C}_d}^2 \leqslant \sum_{j=1}^n ~ \|\sum_{|\alpha| \leqslant d-1} u_\alpha x_{\alpha+e_j}\|_{\mathcal{C}_d}^2 \leqslant R^2 \|\sum_{|\alpha| \leqslant d-1} u_\alpha x_{\alpha}\|_{\mathcal{C}_d}^2$. Hence $T_k$ is well-defined and bounded by $R$.

Clearly, $(T_1,\hdots,T_n)$ is a pair-wise commuting tuple of operators on $V$. Let's now prove that $(T_1^*, \hdots T_n^*,T_1, \hdots, T_n)$ is a pair-wise commuting tuple of operators, which reduces to showing that $T_i^*T_j-T_jT_i^*=0$ for all $1 \leqslant i \leqslant j \leqslant n$ (where $(\cdot)^*$ stands for adjoint). To do so, consider $1 \leqslant i < j \leqslant n$ and $u,v,w \in V$. Point 2 implies that $V = \text{vec} (x_\alpha)_{|\alpha| \leqslant d-d_K}$. Thus there exists some complex numbers $(u_\alpha)_{|\alpha|\leqslant d-d_K}$, $(v_\alpha)_{|\alpha|\leqslant d-d_K}$, and $(w_\alpha)_{|\alpha|\leqslant d-d_K}$ such that $u = \sum_{|\alpha|\leqslant d-d_K} u_\alpha x_\alpha$, $v = \sum_{|\alpha|\leqslant d-d_K} v_\alpha x_\alpha$ and $w = \sum_{|\alpha|\leqslant d-d_K} w_\alpha x_\alpha$. Given $k \in \mathbb{N}$ and $\varphi \in \mathbb{C}_{k}[\bar{z},z]$, notice that $\langle u , \varphi (T)v \rangle_{\mathcal{C}_d} = \sum_{|\gamma|,|\delta|\leqslant k} \varphi_{\gamma,\delta} \langle T^\gamma u , T^\delta v \rangle_{\mathcal{C}_d} 
= \sum_{|\alpha|,|\beta| \leqslant d-d_K} \sum_{|\gamma|,|\delta|\leqslant k} \varphi_{\gamma,\delta} \langle T^\gamma x_\alpha , T^\delta x_\beta \rangle_{\mathcal{C}_d} \bar{u}_\alpha v_\beta =$\\$ \sum_{|\alpha|,|\beta| \leqslant d-d_K} \sum_{|\gamma|,|\delta|\leqslant k} \varphi_{\gamma,\delta} \langle x_{\alpha+\gamma} , x_{\beta+\delta} \rangle_{\mathcal{C}_d} \bar{u}_\alpha v_\beta= \hdots$\\$ \sum_{|\alpha|,|\beta| \leqslant d-d_K} ( \sum_{|\gamma|,|\delta|\leqslant k} \varphi_{\gamma,\delta} y_{\alpha+\gamma,\beta+\delta} ) \bar{u}_\alpha v_\beta = \vec{u}^H M_{d-d_K}(\varphi y) \vec{v}.$ 
As a result,
\begin{equation}
\left<
\begin{pmatrix}
u \\
v \\
w
\end{pmatrix},
\begin{pmatrix}
I        & T_i^* & T_j^* \\
T_i    & T_i^* T_i  & T_j^* T_i  \\
T_j    & T_i^* T_j  & T_j^* T_j
\end{pmatrix} 
\begin{pmatrix}
u \\
v \\
w
\end{pmatrix}
\right>_{ \mathcal{C}_d \times \mathcal{C}_d \times \mathcal{C}_d}= \hdots
\end{equation}
\begin{equation}
\begin{pmatrix}
\vec{u} \\
\vec{v} \\
\vec{w}
\end{pmatrix}^H
\begin{pmatrix} 
M_{d-d_K}(y) & M_{d-d_K}(\bar{z}_i y) & M_{d-d_K}(\bar{z}_j y) \\ 
M_{d-d_K}(z_i y) & M_{d-d_K}(|z_i|^2 y)  & M_{d-d_K}(\bar{z}_j z_i y) \\
M_{d-d_K}(z_j y) & M_{d-d_K}(\bar{z}_i z_j y) & M_{d-d_K}(|z_j|^2 y)
\end{pmatrix}
\begin{pmatrix}
\vec{u} \\
\vec{v} \\
\vec{w}
\end{pmatrix}
\end{equation}
Point 3 implies that 
\begin{equation}
\begin{pmatrix}
I        & T_i^* & T_j^* \\
T_i    & T_i^* T_i  & T_j^* T_i  \\
T_j    & T_i^* T_j  & T_j^* T_j
\end{pmatrix}  \succcurlyeq 0
\end{equation}
which is equivalent to the fact that Schur complement satisfies
\begin{equation}
\label{eq:schur}
\begin{pmatrix}
T_i^* T_i  & T_j^* T_i  \\
T_i^* T_j  & T_j^* T_j
\end{pmatrix} 
-
\begin{pmatrix}
T_i T_i^*  & T_i T_j^* \\
T_j T_i^*  & T_j T_j^*
\end{pmatrix}
\succcurlyeq 0.
\end{equation}
Thus $T_i^*T_i-T_iT_i^* \succcurlyeq 0$ and $T_j^*T_j-T_jT_j^* \succcurlyeq 0$. Since their trace is zero, we in fact have that $T_i^*T_i-T_iT_i^* = 0$ and $T_j^*T_j-T_jT_j^* = 0$. Going back to the Schur complement \eqref{eq:schur}, we thus have $T_i^* T_j - T_j T_i^* = 0$.

Having proven that $(T_1^*, \hdots T_n^*,T_1, \hdots, T_n)$ is a pair-wise commuting tuple of operators, it follows that they are commonly diagonizable. In other words, there exists orthogonal projectors $E_1, \hdots, E_p$ of $V$ such that $E_i E_j = 0$ for all $1 \leqslant i \neq j \leqslant p$ and there exists some complex numbers $( \lambda_{k,j})_{1 \leqslant k \leqslant n}^{1 \leqslant j \leqslant p}$ such that $ T_k = \sum_{j=1}^{p} \lambda_{k,j} E_j $ for all $1 \leqslant k \leqslant n$ (and thus $T_k^* = \sum_{j=1}^{p} \overline{\lambda_{k,j}} E_j$). For all $|\alpha|,|\beta| \leqslant d$, we thus have $y_{\alpha,\beta} = \langle x_\alpha , x_\beta \rangle_{\mathcal{C}_d} = \langle T^\alpha x_0 , T^\beta x_0 \rangle_{\mathcal{C}_d} = \langle x_0 , (T^*)^\alpha T^\beta x_0 \rangle_{\mathcal{C}_d} = \langle x_0 , \sum_{j=1}^p \bar{\lambda}_j^\alpha \lambda_j^\beta E_j x_0 \rangle_{\mathcal{C}_d} =$\\ $\sum_{j=1}^p \bar{\lambda}_j^\alpha \lambda_j^\beta \langle x_0 , E_j x_0 \rangle_{\mathcal{C}_d}$. Naturally, the number of projectors satisfies $p \leqslant \text{dim}(V) = \text{rank}M_d(y)$. Conversly, $ \text{rank}M_d(y) = \text{rank}(y_{\alpha,\beta})_{|\alpha|,|\beta|\leqslant d} \leqslant p$. Hence $p = \text{rank}M_d(y)$, the elements $\lambda_1, \hdots , \lambda_p$ are all distinct, and $\langle x_0 , E_j x_0 \rangle_{\mathcal{C}_d} > 0$ for all $1 \leqslant j \leqslant p$. Thus $\mu:=\sum_{j=1}^p  \langle x_0 , E_j x_0 \rangle_{\mathcal{C}_d} \delta_{\lambda_j}$ is a positive $\text{rank} M_{d-d_K}(y)$-atomic measure that satisfies \eqref{eq:atomic_rep}. Given $1\leqslant j \leqslant p$, let's show that $\lambda_j := (\lambda_{k,j})_{1\leqslant k \leqslant n} \in K$. There exists $u \in V \setminus \{0\}$ such that $T_k u = \lambda_{k,j} u$ for all $1 \leqslant k \leqslant n$. Normality implies that $T_k^* u = \overline{\lambda_{k,j}} u$. Hence $(T^*)^\alpha T^\beta u = \bar{\lambda}_j^\alpha \lambda_j^\beta  u $ for all $\alpha, \beta \in \mathbb{N}$. As a result, $ g_i(\lambda_j) \|u\|_{\mathcal{C}_d}^2 =  \langle u , g_i(\lambda_j) u \rangle_{\mathcal{C}_d} = \langle u , g_i(T) u \rangle_{\mathcal{C}_d} = \vec{u}^H M_{d-d_K}(g_iy) \vec{u} \geqslant 0 $. Thus $ \lambda_j \in K$.

$(\Longrightarrow)$ Let $p:= \text{rank} M_{d-d_K}(y)$, and let $(\lambda_j)_{1\leqslant j \leqslant p}$ and $(m_{j})_{1\leqslant j \leqslant p}$ denote the distinct atoms and their positive weights respectively. Let $x_\alpha := ( \sqrt{m_{j}} \lambda_j^\alpha )_{1\leqslant j \leqslant p} \in \mathbb{C}^p$ for all $|\alpha|\leqslant d$ and $V := \text{vec}(x_\alpha)_{|\alpha|\leqslant d}$. With these notations, we have $y_{\alpha,\beta} = \langle x_\alpha , x_\beta \rangle_{\mathcal{C}_d}$ for all $|\alpha|,|\beta| \leqslant d$. Notice that $p = \text{rank} M_{d-d_K}(y) \leqslant \text{rank} M_{d}(y) = \text{dim} V \leqslant p$, thus Point 2 holds. Given $1\leqslant k \leqslant n$, let $T_k := \text{diag}(\lambda_{k,1}, \hdots, \lambda_{k,p})$. It satisfies the shift property $T_k x_\alpha = x_{\alpha + e_k}$ for all $|\alpha|\leqslant d-1$ and $T_k^* = \text{diag}(\bar{\lambda}_{k,1}, \hdots, \bar{\lambda}_{k,p})$. Moreover, the shifts and their adjoints are pair-wise commuting so \eqref{eq:schur} holds and thus Point 3 does too. Let's now prove Point 1. Consider $1\leqslant i \leqslant m$ and some complex numbers $(u_{\alpha})_{|\alpha|\leqslant d-k_i} =: \vec{u}$. Let $u:= \sum_{|\alpha|\leqslant d-k_i} u_\alpha x_\alpha$ and $u=:(u_j)_{1 \leqslant j \leqslant p} \in \mathbb{C}^p$. We have $\vec{u}^H M_{d-k_i}(g_iy) \vec{u} = \langle u , g_i(T)u \rangle_{\mathcal{C}_d} = \langle u , [g_i(\lambda_j)u_j]_{j=1}^p \rangle_{\mathcal{C}_d} = \sum_{j=1}^p g_i(\lambda_j) |u_j|^2 \geqslant 0$. Hence $\text{dim} ~ \text{Ker} ~ g_i(T)$ ($=p - \text{rank} ~g_i(T)$ due to the rank-nullity theorem) is equal to number of atoms that are zeros of $g_i$. To conclude, notice that $\text{rank} ~ g_i(T) = \text{rank} ~ (\langle x_{\alpha},g_i(T)x_{\beta} \rangle_{\mathcal{C}_d})_{|\alpha|,|\beta| \leqslant d-d_K} = \text{rank} M_{d-d_K} (g_iy)$.
\end{proof}

In the univariate case $n=1$, Theorem~\ref{th:trunc} holds when Point 3 is replaced by
\begin{equation}
\label{eq:uni}
\begin{pmatrix} 
M_{d-d_K}(y) & M_{d-d_K}(\bar{z} y) \\ 
M_{d-d_K}(z y) & M_{d-d_K}(|z|^2 y)
\end{pmatrix} 
\succcurlyeq 0.
\end{equation}
In Theorem~\ref{th:trunc}, if we assume that $y_{0,0} > 0$, then Point 2 and Point 3 may be replaced by rank$M_{d}(y) = 1$. Indeed, in that case, the shift operators act on a one dimensional space, so they and their adjoints must commute pair-wise. For previous work on the link between linear functionals that are nonnegative on a quadratic module and bounded operators that admit a cyclic vector, see~\cite{putinar-1993} and~\cite[Theorem 2.3]{curto-2010}.

\begin{corollary}
\label{cor:trunc}
\normalfont
\textit{
Let $y \in \mathcal{H}_d$ be a Hankel matrix (i.e. $y_{\alpha,\beta} = y_{\gamma,\delta}$ for all $|\alpha|,|\beta|,|\gamma|,|\delta| \leqslant d$ such that $\alpha+\beta=\gamma+\delta$).}
\textit{Then there exists a positive} $ \text{rank} M_{d-d_K}(y)$\textit{-atomic measure} $\mu$ \textit{supported on} $K$ \textit{such that:}
\begin{equation}
 y_{\alpha,\beta} = \int_{\mathbb{C}^n} \bar{z}^\alpha z^\beta d\mu ~, ~~~~ \text{for all} ~~ |\alpha|,|\beta| \leqslant d
\end{equation}
\textit{if and only if:}
\begin{enumerate}
\item $M_d(y) \succcurlyeq 0$ and $M_{d-d_K}(g_iy) \succcurlyeq 0, ~ i = 1 \hdots m$;
\item $\text{rank} M_{d}(y) = \text{rank} M_{d-d_K}(y)$.
\end{enumerate}
\textit{Moreover, for each $1\leqslant i \leqslant m$, the measure $\mu$ has exactly} $\text{rank} M_{d}(y) - \text{rank} M_{d-d_K}(g_i y)$ \textit{atoms that are zeros of $g_i$.}
\end{corollary}
\begin{proof}
($\Longrightarrow$) Same as in proof of Theorem~\ref{th:trunc}. ($\Longleftarrow$) The Hankel property implies that the shifts in the proof of Theorem~\ref{th:trunc} are well-defined and self-adjoint. Indeed, consider $1 \leqslant k \leqslant n$ and $u,v \in V$. According the Point 2, there exists some complex numbers $(u_\alpha)_{|\alpha|\leqslant d-1}$ and $(v_\alpha)_{|\alpha|\leqslant d-1}$ such that $u = \sum_{|\alpha|\leqslant d-1} u_\alpha x_\alpha$ and $v = \sum_{|\alpha|\leqslant d-1} v_\alpha x_\alpha$. If $\sum_{|\alpha|\leqslant d-1} u_\alpha x_\alpha=0$, then for all $|\beta|\leqslant d-1$, we have
$ \langle \sum_{|\alpha|\leqslant d-1} u_{\alpha} x_{\alpha+e_k} , x_\beta \rangle_{\mathcal{C}_d} = \sum_{|\alpha| \leqslant d-1}  \bar{u}_{\alpha} \langle x_{\alpha+e_k} , x_\beta \rangle_{\mathcal{C}_d} = \sum_{|\alpha| \leqslant d-1}  \bar{u}_{\alpha} y_{ \alpha + e_k , \beta } = $\\ $\sum_{|\alpha| \leqslant d-1}  \bar{u}_{\alpha} y_{ \alpha , \beta + e_k } = \langle \sum_{|\alpha|\leqslant d-1} u_{\alpha} x_{\alpha} , x_{\beta+e_k} \rangle_{\mathcal{C}_d} = 0 $, and hence $\sum_{|\alpha|\leqslant d-1} u_{\alpha} x_{\alpha+e_k} = 0$. $T_k$ is thus well defined. Moreover, we have $T_k^* = T_k$ because $\langle T_k u , v \rangle_{\mathcal{C}_d} = \langle \sum_{\alpha} u_\alpha x_{\alpha+e_k} , \sum_{\alpha} v_\alpha x_\alpha \rangle_{\mathcal{C}_d} = \sum_{\alpha,\beta} \bar{u}_\alpha v_\beta \langle  x_{\alpha+e_k} ,  x_\beta \rangle_{\mathcal{C}_d} = \sum_{\alpha,\beta} \bar{u}_\alpha v_\beta y_{\alpha+e_k,\beta} =$\\ $\sum_{\alpha,\beta} \bar{u}_\alpha v_\beta y_{\alpha,\beta+e_k} = \langle u , T_k v \rangle_{\mathcal{C}_d}$.
\end{proof}

A Hermitian matrix that is a Hankel matrix is real symmetric. Hence the atoms in Corollary~\ref{cor:trunc} lie in $K \cap \mathbb{R}^n$. Corollary~\ref{cor:trunc} is thus the same as~\cite[Theorem 3.11]{lasserre-2010} due to Curto and Fialkow~\cite[Theorem 1.1]{curto-2005}. This observation leads to Figure~\ref{fig:hankel}.
%
%
%
%
\begin{figure}[ht]
  \centering
  
\begin{tikzpicture}[scale=1, transform shape]
 \matrix (m) [matrix of math nodes, row sep=6em, column sep=9em]
	{ \begin{array}{c} \inf_{z\in \mathbb{C}^n} f(z) \\[.5em] \text{s.t.}~ g_i(z) \geqslant 0 \end{array} & \begin{array}{c} \inf_{x\in \mathbb{R}^n} f(x) \\[.5em] \text{s.t.}~ g_i(x) \geqslant 0 \end{array}  \\
     \begin{array}{c} \inf_{y_{\alpha,\beta}} L_y(f) ~~\text{s.t.} \\[.5em] M_{d-k_i}(g_iy) \succcurlyeq 0 \\[.5em] y_{0,0} = 1 \end{array} & \begin{array}{c} \inf_{y_\alpha} L_y(f) ~~\text{s.t.} \\[.5em]  M_{d-k_i}(g_iy) \succcurlyeq 0 \\[.5em] y_{0} = 1 \end{array} \\  };
{ [start chain] \chainin (m-1-1);
    \chainin (m-1-2) [join={node[above,labeled] {\mbox{ Hankel property }}}];
        }
{ [start chain] \chainin (m-1-1);
    \chainin (m-1-2) [join={node[below,labeled] {\mbox{$\begin{array}{c} \bar{z}^{\alpha} z^{\beta} = \bar{z}^{\gamma} z^{\delta} \\ \forall \alpha+\beta=\gamma+\delta \end{array}$ }}}];
        }
{ [start chain] \chainin (m-1-2);
    \chainin (m-2-2);
        }
{ [start chain] \chainin (m-1-1);
    \chainin (m-2-1);
        }
{ [start chain] \chainin (m-2-1);
    \chainin (m-2-2) [join={node[above,labeled] {\mbox{  Hankel property}}}];
        }
{ [start chain] \chainin (m-2-1);
    \chainin (m-2-2) [join={node[below,labeled] {\mbox{ $\begin{array}{c} y_{\alpha,\beta} = y_{\gamma,\delta} \\ \forall \alpha+\beta=\gamma+\delta \end{array}$ }}}];
        }

\draw (-4.7,1) node[below right] {Complex};
\draw (-4.8,.61) node[below right] {Hierarchy};
\draw (-6.05,0.23) node[below right] {\mbox{$y_{\alpha,\beta} = \int_K \bar{z}^\alpha z^\beta d\mu$}};

\draw (3.16,1) node[below right] {Real};
\draw (3.16,0.61) node[below right] {Hierarchy};
\draw (3.16,0.23) node[below right] {\mbox{$y_{\alpha} = \int_K x^\alpha d\mu$}};

\end{tikzpicture}
\caption{Commutativity of Relaxation and Hankel Property}
\label{fig:hankel}
\end{figure}
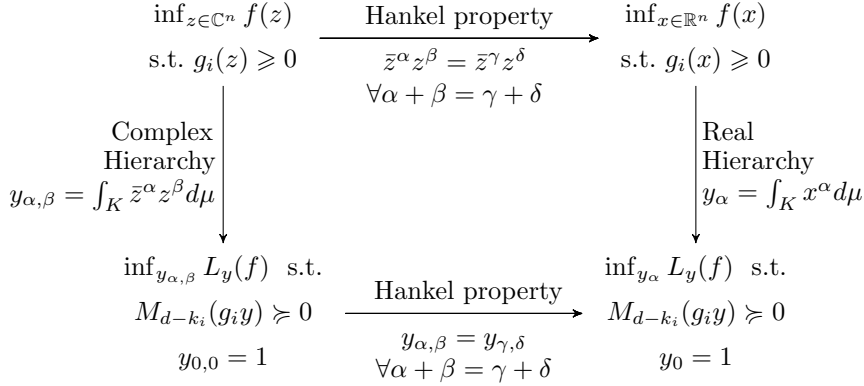 

Next we transpose~\cite[Theorem 1]{josz-2015} from real to complex numbers.
\begin{proposition}
\label{prop:duality}
\normalfont
\textit{Under assumption \eqref{eq:sa},} $\rho_d^* = \rho_d \in \mathbb{R}\cup\{+\infty\}, \forall d \geqslant d^\text{min}$.
\end{proposition}
\begin{proof}
Given $A \in \mathcal{H}_d$, consider the operator norm $\| A \| $, the largest eigenvalue of $A$ in absolute value, and the Frobenius norm $\| A \|_{\mathbb{F}} := \sqrt{ \langle A , A \rangle_{\mathcal{H}_d} }$. Consider $d \geqslant d^\text{min}$. Two cases can occur. Case 1: the feasible set of the complex moment relaxation of order $d$ is non-empty. All norms are equivalent in finite dimension so there exists a constant $C_d \in \mathbb{R}$ such that for all feasible points $(y_{\alpha,\beta})_{|\alpha|,|\beta|\leqslant d}$ we have
$\sqrt{ \sum_{|\alpha|,|\beta|\leqslant d} |y_{\alpha,\beta}|^2 }  = \| M_{d}(y) \|_{\mathbb{F}} \leqslant C_d ~ \| M_{d}(y) \| \leqslant C_d ~ \sum_{l=0}^d R^{2l}$ according to Lemma \ref{lemma:sphere}.                                                          
As a result, the feasible set of the complex moment relaxation of order $d$ is a non-empty compact set and so is its image by $\Lambda$ (defined in \eqref{eq:conversion}). We can thus apply Trnovsk\'a's result \cite{maria-2005} which states that in a semidefinite program in real numbers, if the primal feasible set is non-empty and compact, then there exists a dual interior point and there is no duality gap. Case 2: the feasible set of the complex moment relaxation of order $d$ is empty, i.e. $\rho_d = + \infty$. It must be strongly infeasible because it cannot be weakly infeasible (see~\cite[Section 5.2]{klerk-2000} for definitions). Indeed, if it is weakly infeasible, then there exists a sequence $(y^j)_{j\in \mathbb{N}}$ of elements of $\mathcal{H}$ such that for all $j \in \mathbb{N}$, we have $|y^j_{0,0}-1| \leqslant \frac{1}{j+1}$ and
$\lambda_\text{min} ( M_{d-k_i} (g_iy^j) ) \geqslant - \frac{1}{j+1}$ where $i = 0,\hdots,m$.
Define $c:= (n+d)!/(n!d!)$. We now mimick the computations in Lemma \ref{lemma:sphere} using $y^j_{0,0} \leqslant 1 + \frac{1}{j+1} \leqslant 2$ and $|\text{Tr}(M_{l-1} (sy^j))|  \leqslant \frac{c}{j+1} \leqslant c$ if $1\leqslant l \leqslant d$. Consider $j_0 \in \mathbb{N}$ such that for all $j\geqslant j_0$ and $1\leqslant l \leqslant d$, we have $\sum_{|\alpha|\leqslant l-1,|\gamma| = 1} y_{\gamma+\alpha,\gamma+\alpha}^j - \sum_{0<|\alpha|\leqslant l}  y_{\alpha,\alpha}^j \geqslant -1$.The concluding equation in the proof of Lemma \ref{lemma:sphere} then becomes $\sum_{|\alpha|\leqslant l}  y_{\alpha,\alpha}^j \leqslant 2 + R^2 \left(\sum_{|\alpha|\leqslant l-1}  y_{\alpha,\alpha}^j \right) + c+1$.
As a result, $\text{Tr}(M_{d} (y^j)) = \sum_{|\alpha|\leqslant d}  y_{\alpha,\alpha}^j \leqslant ( 3 + c ) \sum_{l=0}^d R^{2l}$, which, together with $\lambda_\text{min} ( M_{d} (y^j) ) \geqslant -\frac{1}{j+1} \geqslant- 1$, yields
$\lambda_\text{max} ( M_{d} (y^j) ) \leqslant ( 3 + c ) \sum_{l=0}^d R^{2l} + c - 1$. Hence for all $j \geqslant j_0$, the spectrum of $M_{d} (y^j)$ is lower bounded by $-1$ and upper bounded by $B_d :=  ( 3 + c ) \sum_{l=0}^d R^{2l} + c - 1 \geqslant 1$. We therefore have
$\sqrt{ \sum_{|\alpha|,|\beta|\leqslant d} |y_{\alpha,\beta}^j|^2 }  \leqslant C_d ~ \| M_{d}(y) \| \leqslant  C_d \times B_d$.
The sequence $(y^j)_{j\geqslant j_0}$ is thus included in a compact set. Hence there exists a subsequence that converges towards a limit $y^\text{lim}$ which satisfies $y^\text{lim}_{0,0} =1$ and the constraints $\lambda_\text{min} ( M_{d-k_i} (g_iy^\text{lim}) ) \geqslant 0, ~ i = 0,\hdots,m$. Therefore $y^\text{lim}$ is a feasible point of the complex moment relaxation of order $d$, which is a contradiction. Strong infeasibility means that the dual feasible set contains an improving ray~\cite[Definition 5.2.2]{klerk-2000}. Moreover, $\inf_{y\in \mathcal{H}_d} L_y(f)$ subject to $y_{0,0} = 1, ~ M_d(y) \succcurlyeq 0, ~\text{and} ~ M_{d-1}(sy) = 0$ is a semidefinite program with a non-empty compact feasible set hence the dual feasible set contains a point $(\lambda,\sigma_0,\sigma_1)$. As result $(\lambda,\sigma_0,\sigma_1,0,\hdots,0)$ is a feasible point of the complex sum-of-squares relaxation of order $d$. Together with the improving ray, this means that $\rho_d^* = + \infty$.
To conclude, $\rho_d^* = \rho_d$ in both cases.
\end{proof}
\begin{proposition}
\label{prop:saddle}
\normalfont
\textit{Assume that \eqref{eq:complexPOP} satisfies \eqref{eq:weak} and has a global solution} $z^\text{opt} \in K^\text{opt}$. \textit{In addition, assume that} $(\sigma_0^\text{opt},\hdots,\sigma_m^\text{opt}) \in \Sigma[z]^{m+1}$ \textit{is an optimal solution to the sum-of-squares problem \eqref{eq:sosinf}. 
Then} $(z^\text{opt},\sigma_1^\text{opt},\hdots,\sigma_m^\text{opt})$ \textit{is a saddle point of $\phi : \mathbb{C}^n \times \Sigma[z]^m \longrightarrow \mathbb{R}$ defined by $(z,\sigma) \longmapsto f(z) - \sum_{i=1}^m \sigma_i(z) g_i(z)$.}
\end{proposition}
\begin{proof}
The optimality of $(\sigma_0^\text{opt},\hdots,\sigma_m^\text{opt})$ means that $f - f^\text{opt} = \sum_{i=0}^m \sigma_i^\text{opt} g_i$. With $f(z^\text{opt}) - f^\text{opt} = \sum_{i=0}^m \sigma_i^\text{opt}(z^\text{opt}) g_i(z^\text{opt}) = 0$, $\sigma_i^\text{opt}(z^\text{opt}) \geqslant 0$, and $g_i(z^\text{opt}) \geqslant 0$, we have $\sigma_i^\text{opt}(z^\text{opt}) g_i(z^\text{opt}) = 0$ for $i=0,\hdots,m$. It follows that $\phi(z^\text{opt},\sigma) \leqslant \phi(z^\text{opt},\sigma^\text{opt})$ for all $\sigma \in \Sigma[z]$. For all $z \in \mathbb{C}^n$, $\phi(z^\text{opt},\sigma^\text{opt}) \leqslant \phi(z,\sigma^\text{opt})$ because $f(z) - f^\text{opt} - \sum_{i=1}^m \sigma_i^\text{opt}(z) g_i(z) = \sigma_0^\text{opt}(z) \geqslant 0$.
\end{proof}

Given an application $\varphi : \mathbb{C}^n \longrightarrow \mathbb{R}$, define $\tilde{\varphi} : \mathbb{R}^{2n} \longrightarrow \mathbb{R}$ by $(x,y) \longmapsto \varphi(x+\textbf{i}y)$. 
If $\tilde{\varphi}$ is $\mathbb{R}$-differentiable at point $(x,y) \in \mathbb{R}^{2n}$, consider the Wirtinger derivative~\cite{wirtinger-1927} defined by $\nabla \varphi (x+\textbf{i}y):= \frac{1}{2}(\nabla_x \tilde{\varphi} (x,y) - \textbf{i} \nabla_y \tilde{\varphi} (x,y) ) \in \mathbb{C}^{n}$.
\begin{corollary}
\label{cor:KKT}
\normalfont
\textit{With the same assumptions as in} Proposition \ref{prop:saddle}\textit{, we have}
\begin{equation}
\boxed{
\begin{array}{l}
\nabla f (z^\text{opt}) = \sum_{i=1}^m \sigma_i^\text{opt}(z^\text{opt}) \nabla g_i(z^\text{opt}), \\
\sigma_i^\text{opt}(z^\text{opt}), g_i(z^\text{opt}) \geqslant 0, ~~~ i = 1,\hdots,m, \\
\sigma_i^\text{opt}(z^\text{opt})  g_i(z^\text{opt}) = 0, ~~~ i = 1,\hdots,m.
\end{array}
}
\end{equation}
\end{corollary}
\begin{proof}
$z^\text{opt}$ is a minimizer of $z \in \mathbb{C}^n \longmapsto \phi(z,\sigma^\text{opt})$ thus $\nabla_z \phi(z^\text{opt},\sigma^\text{opt}) = \nabla f (z^\text{opt}) - \sum_{i=1}^m \nabla \sigma_i^\text{opt}(z^\text{opt}) g_i(z^\text{opt}) - \sum_{i=1}^m \sigma_i^\text{opt}(z^\text{opt}) \nabla g_i(z^\text{opt}) = 0$. Consider $1\leqslant i \leqslant m$. Since $\sigma_i^\text{opt}(z^\text{opt})= 0$ and $\sigma_i^\text{opt} \in \Sigma[z]$, it must be that $|z_k-z_k^\text{opt}|^2$ divides $\sigma_{i,k}^\text{opt}: z_k \in \mathbb{C} \longmapsto \sigma_i^\text{opt}(z_1^\text{opt},\hdots,z_{k-1}^\text{opt},z_k,z_{k+1}^\text{opt},\hdots,z_n^\text{opt})$. With $z_k^\text{opt} =: x_k^\text{opt}+\textbf{i}y_k^\text{opt}$, the real number $x_k^\text{opt}$ is a root of multiplicity 2 of $x_k \in \mathbb{R} \longmapsto \sigma_{i,k}^\text{opt} (x_k+\textbf{i}y_k^\text{opt})$, with an analogous remark for $y_k^\text{opt}$. Thus $\nabla \sigma_i^\text{opt}(z^\text{opt}) = 0$ which leads to the desired result.
\end{proof}
\subsection{Comparison of Real and Complex Hierarchies} 
\label{subsec:Comparison of Real and Complex Hierarchies} 
\begin{figure}[ht]
  \centering

\begin{tikzpicture}[scale=1, transform shape]
 \matrix (m) [matrix of math nodes, row sep=2em, column sep=.0em]
	{\small \boxed{\begin{array}{c} \sup_{\lambda \in \mathbb{R}} \lambda \\[.7em] \text{s.t.}~\forall z \in K, \\[.7em] f(z) - \lambda \geqslant 0 \\ \end{array} }
	& 
	\small \boxed{ \begin{array}{c} \sup_{\lambda \in \mathbb{R}} \lambda \\[.7em] \text{s.t.}~ \forall x+\textbf{i}y \in K, \\[.7em] f(x+\textbf{i}y) - \lambda \geqslant 0 \end{array}}  \\
    & 
    \small \boxed{\begin{array}{c} \sup_{\lambda \in \mathbb{R}} \lambda ~~ \text{s.t.} \\\\ \forall x,y \in \mathbb{R}^n, ~~ f(x+\textbf{i}y) - \lambda = \\[.5em] \sum\limits_{i=0}^m \left(\sum\limits_{j=1}^{r_i} \left|\sum\limits_{|\alpha+\beta|\leqslant d-k_i} p_{j,\alpha,\beta}^i (x-\textbf{i}y)^\alpha (x+\textbf{i}y)^\beta\right|^2\right) g_i(x+\textbf{i}y) \end{array}} ~~~~~~~~~~~~~~~~~~\\
    \small \boxed{\begin{array}{c} \sup_{\lambda \in \mathbb{R}} \lambda ~~ \text{s.t.} \\\\ \forall z \in \mathbb{C}^n, ~~ f(z) - \lambda = \\[.5em] \sum\limits_{i=0}^m \left(\sum\limits_{j=1}^{r_i} \left|\sum\limits_{|\alpha|\leqslant d-k_i} p_{j,\alpha}^i z^\alpha\right|^2\right) g_i(z) \end{array}} 
    & \small \boxed{\begin{array}{c} \sup_{\lambda \in \mathbb{R}} \lambda ~~ \text{s.t.} \\\\ \forall x,y \in \mathbb{R}^n, ~~ f(x+\textbf{i}y) - \lambda = \\[.5em] \sum\limits_{i=0}^m \left(\sum\limits_{j=1}^{r_i} \left|\sum\limits_{|\alpha|\leqslant d-k_i} p_{j,\alpha}^i(x+\textbf{i}y)^\alpha\right|^2\right) g_i(x+\textbf{i}y) \end{array}} ~\\};

{ [start chain] \chainin (m-1-1);
    \chainin (m-3-1);
        }

{ [start chain] \chainin (m-1-2);
    \chainin (m-2-2);
        }
        
{ [start chain] \chainin (m-1-1);
    \chainin (m-1-2) [join={node[above,labeled] {\mbox{\small $\begin{array}{c} \text{Identify} \\ \text{real and} \end{array}$}}}];
        }
{ [start chain] \chainin (m-1-1);
    \chainin (m-1-2) [join={node[below,labeled] {\mbox{\small $\begin{array}{c} \text{imaginary} \\ \text{parts} \end{array}$}}}];
        }
        
{ [start chain] \chainin (m-3-1);
    \chainin (m-3-2) [join={node[above,labeled] {\mbox{\small $\begin{array}{c} \text{Identify} \\ \text{real and} \end{array}$}}}];
        }
{ [start chain] \chainin (m-3-1);
    \chainin (m-3-2) [join={node[below,labeled] {\mbox{\small $\begin{array}{c} \text{imaginary} \\ \text{parts} \end{array}$}}}];
        }
\draw (-6.96,0.8) node[below right] {\small Complex};
\draw (-5.525,0.8) node[below right] {\small Hierarchy};

\draw (2.375,-1.15) node[below right] {\large $\geqslant$};

\draw (1.9,2.4) node[below right] {\small Real};
\draw (2.67,2.4) node[below right] {\small Hierarchy};

\end{tikzpicture}
\caption{Comparison of Real and Complex Hierarchies}
\label{fig:comparison}
\end{figure}
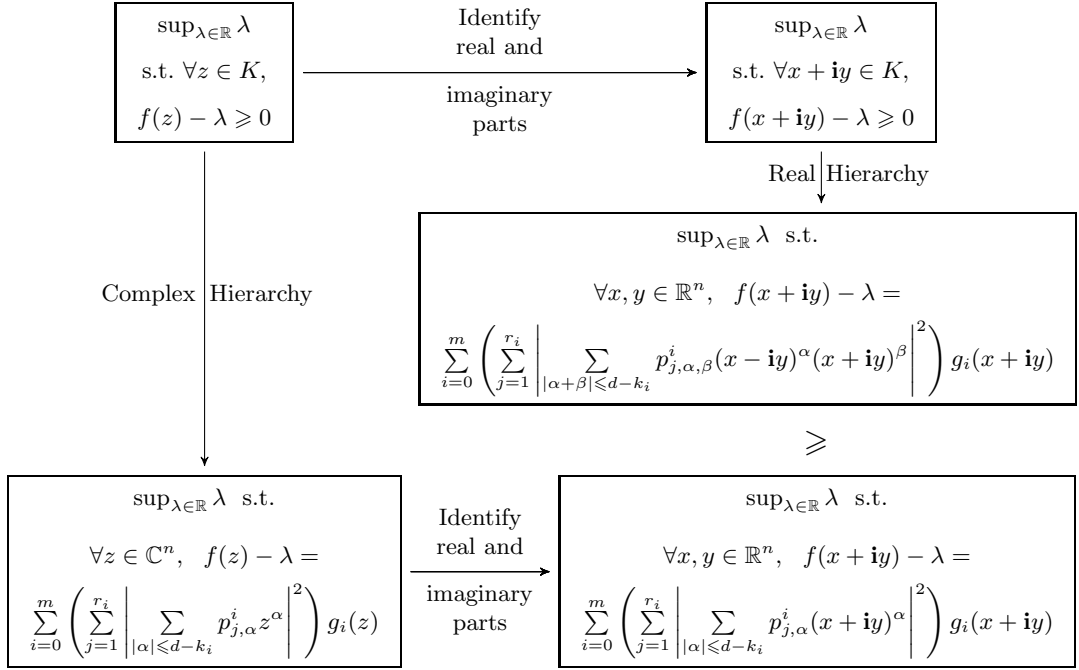

In Figure \ref{fig:comparison} where $p_{j,\alpha}^i, p_{j,\alpha,\beta}^i \in \mathbb{C}$, the real sum-of-squares hierarchy is artificially written using squares of moduli of complex polynomials. It thus yields bounds superior or equal to the complex hierarchy. For example, at order 2, the real hierarchy yields $1.0000$ while the complex hierarchy yields $0.6813$ for \eqref{eq:ex2sphere}.
However, the size of the largest semidefinite constraint in the complex hierarchy when converted to real numbers, i.e. $2\times \text{card}\{\alpha \in \mathbb{N}^n ~\text{s.t.}~ |\alpha|\leqslant d\} = 2(n+d)!/(n!d!)$, is far inferior to that of the real hierarchy, i.e. $\text{card}\{\alpha,\beta \in \mathbb{N}^n ~\text{s.t.}~ |\alpha+\beta|\leqslant d\} =(2n+d)!/((2n)!d!)$. 
At fixed $d$, the size reduction converges towards $2^{d-1}$ as $n \rightarrow \infty$. Further reduction is possible (Section~\ref{subsec:Invariant Hierarchy}).

\subsection{Invariant Hierarchy} 
\label{subsec:Invariant Hierarchy} 
We generalize and transpose to complex numbers the work in~\cite{riener-2013} (see also~\cite{cimpric-2009}). Let $(G,\times)$ denote a compact group whose unit we denote 1. First, consider the continuous action of $G$ on $\mathbb{C}^n$ via $\mathcal{A} : G \times \mathbb{C}^n \longrightarrow \mathbb{C}^n$ such that $\mathcal{A}(1,z) = z$, $\mathcal{A}(g_1 \times g_2,z) = \mathcal{A}(g_1 ,\mathcal{A}(g_2,z))$ for all $z \in \mathbb{C}^n$ and $g_1,g_2 \in G$. Second, consider the action of $G$ on $\mathbb{R}[\bar{z},z]$ via $\mathcal{A}' : G \times \mathbb{R}[\bar{z},z] \longrightarrow \mathbb{R}[\bar{z},z]$ defined by $\mathcal{A}'(g,\varphi) := \varphi(\mathcal{A}(g,z))$. Third, consider the action of $G$ on the set $\mathcal{B}(K)$ of Borel subsets of $K$ via $\mathcal{A}'' : G \times \mathcal{B}(K) \longrightarrow \mathcal{B}(K)$ defined by $\mathcal{A}''(g,B):= \{ z \in K ~|~ \mathcal{A}(g,z) \in B \}$. Last, consider the action of $G$ on $\mathcal{M}(K)$ via $\mathcal{A}''' : G \times \mathcal{M}(K) \longrightarrow \mathcal{M}(K)$ defined by $\mathcal{A}'''(g,\mu)(\cdot):=\mu(\mathcal{A}''(g,\cdot))$. Given a set $S$ on which $G$ is acting via $\mathcal{T}$ and $Y \subset S$, let $Y^G := \{ y \in Y ~|~ \forall g \in G, ~ \mathcal{T}(g,y) = y \}$. If $f,g_1,\hdots,g_m \in \mathbb{R}[\bar{z},z]^G$, then:
\begin{equation}
\label{eq:momentinv}
\begin{array}{rcllll}
f^\text{opt} & = & \inf_{\mu \in \mathcal{M}(K)^G}  & \int_K f d\mu & \mathrm{s.t.} & \int_K d\mu = 1 ~~\&~~ \mu \geqslant 0.
\end{array}
\end{equation}
If $\mu$ is feasible for \eqref{eq:momentinv}, then $\int_K |p|^2 g_i d\mu \geqslant 0$ for all $d\in \mathbb{N}$ and  $p \in \mathbb{C}_d[z]$ such that $|p|^2 \in \Sigma_{d}[z]^G$. Given $A \in \mathcal{H}_d$, let  $A \succcurlyeq^G 0$ therefore denote $ \vec{p}^H A \vec{p} \geqslant 0$ for all $p \in \mathbb{C}_d[z]$ such that $|p|^2 \in \Sigma_{d}[z]^G$. This yields a \textit{G-invariant hierarchy} for all $d\geqslant d^\text{min}$
\begin{equation}
\label{eq:momentdinv}
\begin{array}{rclll}
\rho^G_d & := & \inf_{y\in \mathcal{H}_d} & L_y(f)                                          & \\
           &     & \text{s.t.}                          & y_{0,0} = 1,                                 &  \\
           &     &                                         &  M_{d-k_i}(g_iy) \succcurlyeq^G 0, & i = 0, \hdots, m,
\end{array}
\end{equation}
\begin{equation}
\label{eq:sosdinv}
\begin{array}{rcll}
(\rho^G_d)^* &:= & \sup_{\lambda,\sigma} & \lambda \\
              &    & \text{s.t.} &  f-\lambda = \sum_{i=0}^m \sigma_i g_i, \\ 
              &    &                &  \lambda \in \mathbb{R},~ \sigma_i \in \Sigma_{d-k_i}[z]^G,~ i=0, \hdots, m,
\end{array}
\end{equation}
whose convergence we now discuss. Assume that the first $2e$ ($e \in \mathbb{N}^*$) constraint functions $g_1,\hdots, g_m$ form equality constraints (i.e. $g_{2i-1} = - g_{2i} =: h_i , ~ i=1\hdots e$). Define $S:=\Sigma[z]^G + \sum_{i=1}^e \mathbb{R}[\bar{z},z]^G h_i$ (and $S:= \Sigma[z]^G$ if there are no equality constraints).
\begin{proposition}
\normalfont
\label{prop:invariance}
\textit{Assume that $f,g_1,\hdots,g_m \in \mathbb{R}[\bar{z},z]^G$ and that $\mathbb{R}[\bar{z},z]^G = S + \mathbb{R}$. If $f_{|K} > 0$, then there exists $\sigma_0, \hdots, \sigma_m \in \Sigma[z]^G$ such that $f = \sum_{i=0}^m \sigma_i g_i$.}
\end{proposition}
\begin{proof}
By definition of $\mathcal{A}'$, $\mathbb{R}[\bar{z},z]^G$ is an $\mathbb{R}$-algebra. As a result, $S$ is a semiring of $\mathbb{R}[\bar{z},z]^G$ (i.e. contains $\mathbb{R}_+$ and is closed in $\mathbb{R}[\bar{z},z]^G$ under taking sums and products) and $M:= S + \sum_{i=2e+1}^m \Sigma[z]^G g_i$ is an $S$-module of $\mathbb{R}[\bar{z},z]^G$ (i.e. contains 1 and satisfies $M+M \subset M$ and $SM \subset M$). The conclusion then follows from \cite[Theorem 2.6]{putinar-scheiderer-2012}.
\end{proof}
\begin{proposition}
\normalfont
\label{prop:torus}
\textit{The torus $G=\mathbb{T}$ in $\mathbb{C}$ with the action $\mathcal{A}(g,z) := gz$ satisfies $\rho^{\mathbb{T}}_d = \rho_d$ and $(\rho^{\mathbb{T}}_d)^* = \rho_d^*$ for all  $d\geqslant d^\text{min}$ if $f,g_1,\hdots,g_m \in \mathbb{R}[\bar{z},z]^{\mathbb{T}}$.}
\end{proposition}
\begin{proof}
Firstly, $\varphi \in \mathbb{R}[\bar{z},z]^{\mathbb{T}}$ if and only if $\forall \alpha,\beta \in \mathbb{N}^n, ~ |\alpha-\beta|\varphi_{\alpha,\beta} = 0$. Indeed, for all $\theta \in \mathbb{R}$ and $z\in \mathbb{C}^n$, $\varphi(e^{\textbf{i}\theta} z) = \sum_{\alpha,\beta \in \mathbb{N}^n} \varphi_{\alpha,\beta} \overline{(e^{\textbf{i}\theta} z)}^\alpha (e^{\textbf{i}\theta} z)^\beta = \sum_{\alpha,\beta \in \mathbb{N}^n} \varphi_{\alpha,\beta} e^{\textbf{i}(|\beta|-|\alpha|)\theta}$ \\ $\bar{z}^\alpha z^\beta$ is equal to $\varphi(z)$ if and only if $\varphi_{\alpha,\beta} = 0$ or $|\beta-\alpha|\theta \equiv 0 [2\pi]$ (i.e. $|\beta-\alpha|=0$). Secondly, if $\sigma \in \Sigma[z]$, i.e. $\sigma = \sum_{j=1}^r |p_j|^2$, then $\sum_{|\alpha|=|\beta|} \sigma_{\alpha,\beta} \bar{z}^\alpha z^\beta = \sum_{l \in \mathbb{N}} \sum_{j=1}^r |\sum_{|\alpha|=l} p_{j,\alpha}z^\alpha |^2 \in \Sigma[z]^\mathbb{T}$. Thirdly, if $(\lambda,\sigma_0,\hdots,\sigma_m)$ is feasible for \eqref{eq:sosd}, then $(\lambda,\sum_{|\alpha|=|\beta|} \sigma_{0,\alpha,\beta} \bar{z}^\alpha z^\beta,\hdots,\sum_{|\alpha|=|\beta|} \sigma_{m,\alpha,\beta} \bar{z}^\alpha z^\beta)$ is feasible for \eqref{eq:sosdinv}. Thus $(\rho^{\mathbb{T}}_d)^* = \rho_d^*$. Lastly, if $y$ is feasible for \eqref{eq:momentdinv}, then $(y_{\alpha,\beta} \delta_{|\alpha|=|\beta|})_{|\alpha|,|\beta|\leqslant d}$ is feasible for \eqref{eq:momentd} (where $\delta$ is the Kronecker symbol). Hence $\rho^{\mathbb{T}}_d = \rho_d$.
\end{proof}

If $f,g_1,\hdots,g_m \in \mathbb{R}[\bar{z},z]^{\mathbb{T}}$, then the minimum order $d^\text{min}$ of the complex hierarchy, i.e. $\max \{ |\alpha| , |\beta| ~\text{s.t.}~ |f_{\alpha,\beta}| + |g_{1,\alpha,\beta}| + \hdots + |g_{m,\alpha,\beta}| \neq 0 \}$, is equal to that of the real hierarchy, i.e. $\max \{ \lceil (|\alpha| + |\beta|) / 2 \rceil ~\text{s.t.}~ |f_{\alpha,\beta}| + |g_{1,\alpha,\beta}| + \hdots + |g_{m,\alpha,\beta}| \neq 0 \}$, where $\lceil . \rceil$ denotes the ceiling of a real number.
\begin{proposition}
\normalfont
\label{prop:subtorus}
\textit{The subgroup $G = \{-1,1\}$ of $\mathbb{T}$ with $\mathcal{A}(g,z) := gz$ satisfies $\rho^{\{-1,1\}}_d = \rho_d$ and $(\rho^{\{-1,1\}}_d)^* = \rho_d^*$ for all $d\geqslant d^\text{min}$ if $f,g_1,\hdots,g_m \in \mathbb{R}[\bar{z},z]^{\{-1,1\}}$.}
\end{proposition}
\begin{proof}
Firstly, $\varphi \in \mathbb{R}[\bar{z},z]^{\{-1,1\}}$ if and only if $\forall |\alpha+\beta|~\text{odd}, ~ \varphi_{\alpha,\beta} = 0$. Secondly, if $\sigma \in \Sigma[z]$, i.e. $\sigma = \sum_{j=1}^r |p_j|^2$, then $\sum_{|\alpha+\beta|\text{even}} \sigma_{\alpha,\beta} \bar{z}^\alpha z^\beta = \sum_{j=1}^r |\sum_{|\alpha|\text{even}} p_{j,\alpha} z^\alpha|^2 + |\sum_{|\alpha|\text{odd}} p_{j,\alpha} z^\alpha|^2 \in \Sigma[z]^{\{-1,1\}}$. Thirdly, if $(\lambda,\sigma_0,\hdots,\sigma_m)$ is feasible for \eqref{eq:sosd}, then $(\lambda,\sum_{|\alpha+\beta|\text{even}} \sigma_{0,\alpha,\beta} \bar{z}^\alpha z^\beta,\hdots,\sum_{|\alpha+\beta|\text{even}} \sigma_{m,\alpha,\beta} \bar{z}^\alpha z^\beta)$ is feasible for \eqref{eq:sosdinv}. Lastly, if $y$ is feasible for \eqref{eq:momentdinv}, then $(y_{\alpha,\beta} \delta_{|\alpha+\beta|\text{even}})_{|\alpha|,|\beta|\leqslant d}$ is feasible for \eqref{eq:momentd}.
\end{proof}

A problem with $\mathbb{T}$-invariance in complex numbers converts in real numbers to a problem with $\{-1,1\}$-invariance. If $\sigma \in \Sigma_d[z]^\mathbb{T}$, then $(\sigma_{\alpha,\beta})_{|\alpha|,|\beta|\leqslant d}$ has a $(d+1)$-block-diagonal structure, whereas if $\sigma \in \Sigma_d[x,y]^{\{-1,1\}}$ with $z=:x+\textbf{i}y$, then it has a 2-block-diagonal structure (after permutation) whose 2 blocks are much bigger.

\subsection{Multi-Ordered Relaxation}
\label{subsec:Multi-Ordered Relaxation}
We generalize and transpose to complex numbers the work in \cite{mh_sparse_msdp}. The idea is to associate a relaxation order to each constraint. In addition, we consider the coupling of the variables induced by the monomials present in the optimization problem, to which we add the coupling of the variables induced by only some constraints (those with a ``high-order''). For instance, the coupling induced by the mononials in $g_1(z_1,z_2,z_3) := \Re\{ z_1(z_2+z_3)\} \geqslant 0$ is $\{(1,2),(2,3)\}$, while the coupling induced by the constraint is $\{(1,2),(1,3),(2,3)\}$.

Given $\alpha \in \mathbb{N}^n$ with $n>1$, let $\text{supp}(\alpha) := \{ 1 \leqslant s \leqslant n ~|~ \alpha_s \neq 0 \}$. Consider the coupling induced by monomials defined by $\mathcal{E}^\text{mono} := \{ (l,m) ~|~ l \neq m ~\text{s.t.} ~ \exists \alpha,\beta \in \mathbb{N}^n ~\text{s.t.}~ \{l,m\} \subset \text{supp}(\alpha+\beta) ~\text{and}~ |f_{\alpha,\beta}| + |g_{1,\alpha,\beta}| + \hdots + |g_{m,\alpha,\beta}| \neq 0 \}$. 
Given $I \subset \{1,\hdots,n\}$, let $z(I) := \{ z_{i} ~|~ i \in I \}$ if $I \neq \emptyset$, else $z(C) := 1$.
Given $y \in \mathcal{H}$, $d \in \mathbb{N}$, and $ \varphi \in \mathbb{R}[ \overline{z(I)} , z(I) ]$, let $M_{d}(\varphi y,I) := ( ~ M_{d}(\varphi y)(\alpha,\beta ) ~ )_{\text{supp}(\alpha),\text{supp}(\beta) \subset I}$.
Let $G_1,\hdots,G_m \subset \{1,\hdots,n\}$ denote the minimal sets in terms of inclusion such that $(g_1,\hdots, g_m) \in \mathbb{R}_{k_1} [ \overline{z(G_1)} , z(G_1) ] \times \hdots \times \mathbb{R}_{k_m} [ \overline{z(G_m)} , z(G_m) ]$.

Let $(d_1,\hdots,d_m) \in \mathbb{N}^m$ be such that $d_i - k_i \geqslant 0$ for all $1 \leqslant i \leqslant m$. Consider the coupling induced by monomials and high-order constraints defined by $\mathcal{E}^\text{con} := \mathcal{E}^\text{mono} ~\cup~ \bigcup_{d_i>k_i} \{ (l,m) ~|~ l \neq m ~\text{s.t.} ~ \{l,m\} \subset G_i \}$. Let $C_1,\hdots, C_p \subset \{1,\hdots,m\}$ denote the maximal cliques of a chordal extension of $(\{1, \hdots , n \}, \mathcal{E}^\text{con})$. Given $1\leqslant i \leqslant m$, let $I_i := \cup_{l \in L} C_l$ where $L \in \text{argmin} \{ \sum_{l\in L} |C_l| ~|~ G_i \subset \cup_{l\in L} C_l,~ L \subset \{1,\hdots,m\}~ \}$. For $i$ such that $d_i>k_i$, $L$ is a singleton due to the definition of $\mathcal{E}^\text{con}$. 
Define $(d^\text{cl}_1,\hdots,d^\text{cl}_p) \in \mathbb{N}^p$ such that $d^\text{cl}_l := \text{min} \{ d_1, \hdots, d_m \}$ if $C_l \neq I_i$ for all $1\leqslant i \leqslant m$; if not, let $d^\text{cl}_l := \text{max} \{ d_i ~|~ I_i = C_l\}$. Define the relaxation of order $(d_1,\hdots,d_m)$ by
\begin{equation}
\label{eq:momentm}
\boxed{
\begin{array}{rlll}
\rho_{d_1,\hdots,d_m} := & \inf_{y\in \mathcal{H}} & L_y(f)                                          & \\
               & \text{s.t.}                          & y_{0,0} = 1,                                 &  \\
                &                                         &  M_{d^\text{cl}_l}(y,C_l) \succcurlyeq 0, & l = 1, \hdots, p,\\
                &                                         &  M_{d_i-k_i}(g_iy,I_i) \succcurlyeq 0, & i=1,\hdots,m,

\end{array}
}
\end{equation}
\begin{equation}
\label{eq:sosm}
\boxed{
\begin{array}{rll}
\rho_{d_1,\hdots,d_m}^* := & \sup_{\lambda,\sigma} & \lambda \\
                  & \text{s.t.} &  f-\lambda = \sum_{l=1}^p \left(\sigma_{0,l} + \sum_{C_l \subset I_i} \sigma_{i} g_i\right), \\ 
                  &                &  \sigma_{0,l} \in \Sigma_{d^\text{cl}_l}[z(C_l)], ~ l=1,\hdots,p,\\
                  &                & \sigma_{i} \in \Sigma_{d_i-k_i}[z(I_i)], ~ i=1,\hdots,m.

\end{array}
}
\end{equation}

\subsection{Multi-Ordered Hierarchy}
\label{subsec:Multi-Ordered Hierarchy}
Given $H: \mathbb{N} \longrightarrow [k_1,+\infty[ \times \hdots \times [k_m,+\infty[$ such that $\min_{d \rightarrow +\infty} H(d) = + \infty$, consider the sequence indexed by $d \in \mathbb{N}$ of relaxations of order $H(d)$. We refer to such a sequence as \textit{multi-ordered hierarchy}. The uniform case where $H(d) := (d+d^\text{min},\hdots,d+d^\text{min})$ is a special case of the sparse real hierarchy of~\cite{waki-2006} when transposed to complex numbers. The hierarchy of~\cite{waki-2006} converges to the global value of a real polynomial optimization problem if a ball constraint is added for each clique of a chordal extension of the sparsity pattern~\cite[equation (2.29)]{lasserre-2010}. The same holds in the complex case if a slack variable and redundant sphere constraint is added for each clique. The proof is the same as in the real case~\cite[Lemma B.13 and 4.10.2 Proof of Theorem 4.7]{lasserre-2010} once the real vector spaces on which measures are defined are replaced by complex vector spaces. (For other proofs of~\cite[Theorems 2.28 and 4.7]{lasserre-2010}, see~\cite{schweighofer-2007} and~\cite{kuhlmann-2007}.) For $d \in \mathbb{N}$ great enough, i.e. once $d_i>k_i$ for all $1\leqslant i \leqslant m$, the relaxation of order $H(d) = (d_1,\hdots,d_m)$ is at least as tight as the complex sparse relaxation of~\cite{waki-2006} of order $\min H(d)$. Any multi-ordered hierarchy thus globally converges (if a slack variable and sphere constraint is added for each clique).

\subsection{Example of Multi-Ordered Hierarchy: the Mismatch Hierarchy}
\label{subsec:Example of Multi-Ordered Hierarchy: the Mismatch Hierarchy}
Conceptually, the \textit{mismatch hierarchy} is defined by the following procedure. Until a measure can be extracted from a solution $y$ to the multi-ordered relaxation,
\begin{enumerate}
\item compute a solution $y$ to the moment relaxation of order $(d_1,\hdots,d_m)$;
\item find a closest measure $\mu$ to $y$ not necessarily supported on $K$:
\begin{equation}
\underset{\mu ~ \text{Dirac}}{\arg\min} ~ \left\| \left( y_{\alpha,\beta} - \int_{\mathbb{C}^n} \bar{z}^\alpha z^\beta d\mu \right)_{|\alpha|,|\beta|=1} \right\|_{\mathbb{F}}
\end{equation}
\item increment $d_i = d_i+1$ at the highest mismatch, that is to say:
\begin{equation}
\underset{1\leqslant i \leqslant m}{\arg\max} ~ \left| \sum_{\alpha,\beta} g_{i,\alpha,\beta} \left( y_{\alpha,\beta} - \int_{\mathbb{C}^n} \bar{z}^\alpha z^\beta d\mu \right) \right|.
\end{equation}
\end{enumerate}
Stricly speaking, we refer to the \textit{mismatch hierarchy} as the following recursively defined multi-ordered hierarchy $H$. It depends on 3 parameters: a mismatch tolerance $\epsilon > 0$; the number $h \in \mathbb{N}^*$ of highest mismatches considered at each iteration; and an upper bound $\Delta^{\text{max}}_{\text{min}}$ on the difference between maximum and minimum orders, i.e. $\{ \max H(d) - \min H(d) ~|~ d \in \mathbb{N} \}$.

Initialize by $H(0) := k_1 \times \hdots \times k_m$ and let's define $H(d+1)$ in function of $H(d)$. We distinguish two cases. Case 1: if there exists no solution to the moment relaxation of order $H(d)$, then let $H(d+1) := H(d) + (1,\hdots,1)$. Case 2: if not, consider a solution $y$. For $1\leqslant l \leqslant p$, consider some complex numbers $(u(l)_j)_{j \in C_l}$ such that $\overline{u(l)} ~ \overline{u(l)}^H$ is the closest rank 1 matrix to $y(l) := (y_{\alpha,\beta})^{|\alpha|=|\beta|=1}_{\text{supp}(\alpha),\text{supp}(\beta) \subset C_l}$ with respect to the Frobenius norm. Let $\lambda_1(l) \geqslant \lambda_2(l) \geqslant 0$ respectively denote the first and second largest eigenvalues of $y(l)$. Let $\theta \in \mathbb{R}^p$ be a minimizer of $\sum_{l,m=1}^p \sum_{j \in C_l \cap C_m} (\arg u(l)_j + \theta_l - \arg u(m)_j - \theta_m )^2~\text{s.t.}~ \theta \in [ 0 , 2\pi ]^p$. Let $z \in \mathbb{C}^n$ be a minimizer of $\sum_{\lambda_2(l) \neq 0} \lambda_1(l) / \lambda_2(l) \| z(C_l) - u(l) e^{\textbf{i}\theta_l} \|_2^2 + 2 \max \{ \lambda_1(l) / \lambda_2(l) ~|~ \lambda_2(l) \neq 0 \} \times \sum_{\lambda_2(l) = 0} \| z(C_l) - u(l)e^{\textbf{i}\theta_l} \|_2^2$.
We distinguish 3 cases:
\begin{itemize}
\item Case 2.1: $\mathcal{M} := \{ 1 \leqslant i \leqslant m ~|~ |L_y(g_i)-g_i(z)|> \epsilon ~\text{and}~ H_i(d) < \max H(d) \} \neq \emptyset$
\item Case 2.2: $\mathcal{M} = \emptyset$ and $\mathcal{M}' := \{ i \in \mathcal{S} ~|~ |L_y(g_i)-g_i(z)|> \epsilon \} \neq \emptyset $
\item Case 2.3: $ \mathcal{M} = \mathcal{M}' = \emptyset$
\end{itemize}
In Case 2.1, let $H_j(d+1) := H_j(d)+1$ if $I_j \subset I_i$ and $i$ has one of the $h$ highest mismatches $|L_y(g_i)-g_i(z)|$ among $i \in \mathcal{M}$. For all other $1 \leqslant j \leqslant m$, let $H_{j}(d+1) := H_{j}(d)$ unless the bound $\Delta^{\text{max}}_{\text{min}}$ is violated, in which case for all $1 \leqslant j \leqslant m$ such that $H_j(d) = \min H(d)$, let $H_{j}(d+1) := H_{j}(d)+1$. In Case 2.2, apply instructions of Case 2.1 where $\mathcal{M}$ is replaced by $\mathcal{M}'$. In Case 2.3, let $H(d+1) := H(d) + (1,\hdots,1)$. Observe that $\min H(d) \geqslant \max H(d) - \Delta^{\text{max}}_{\text{min}} \rightarrow +\infty$ as $d\rightarrow +\infty$.

\section{Application to Electric Power Systems}
\label{sec:Application to Electric Power Systems}

The optimal power flow is a central problem in power systems introduced half a century ago in~\cite{carpentier-1962}. 
While many non-linear methods~\cite{murillosanchez-thomas-zimmerman-2011,castillo-2013} have been developed to solve this difficult problem, there is a strong motivation for producing more reliable tools. 
Since 2006, the ability of the Shor and second-order conic relaxations to find global solutions~\cite{jabr2006, bai-fujisawa-wang-wei-2008, low_tutorial, andersen2014, coffrin2015, taylor2015, dan2015} has been studied. Some relaxations are presented in real numbers~\cite{lavaei-low-2012, dan2013} and some in complex numbers~\cite{zhang2013, bose2014, bose-chandy-gayme-low-2015}. However, in all numerical applications, standard solvers such as SeDuMi, SDPT3, and MOSEK are used which currently handle only real numbers. Modeling languages such as YALMIP and CVX do handle inputs in complex numbers, but the data is transformed into real numbers before calling the solver~\cite[Example 4.42]{boyd2009}.
We use the European network to illustrate that it is beneficial to relax non-convex constraints before converting from complex to real numbers. 

\subsection{Optimal Power Flow} 
\label{subsec:Optimal Power Flow} 
A transmission network can be modeled using an undirected graph $\mathcal{G} = ( \mathcal{B} , \mathcal{L} )$ where buses $\mathcal{B} = \{1,\hdots n\}$ are linked to one another via lines $\mathcal{L} \subset \mathcal{B} \times \mathcal{B}$. 
Power flows are governed by the admittance matrix $Y \in \mathbb{C}^{n \times n}$ whose extra diagonal terms $(l,m)\in\mathcal{L}$ are equal to $y_{lm}/(\rho_{ml} \rho_{lm}^H)$ and whose diagonal terms $(l,l)$ are equal to $\sum_{(l,m) \in \mathcal{L}} (y_{lm} + y_{lm}^\text{gr})/|\rho_{lm}|^2$. All others terms are equal to zero. Here, $y_{lm} \in \mathbb{C}$ denotes the mutual admittance between
buses $(l,m)\in \mathcal{L}$, $y^\text{gr}_{lm} \in \mathbb{C}$ denotes the
admittance-to-ground at end $l$ of line $(l,m)\in \mathcal{L}$, and $\rho_{lm} \in \mathbb{C}$ denotes the ratio of the
ideal phase-shifting transformer at end $l$ of line $(l,m) \in
\mathcal{L}$.

Each bus injects power $p_k^\text{gen}+ \mathbf{i}q_k^\text{gen}$ into the network with capacity limits $p_k^\text{min},p_k^\text{max},$ $q_k^\text{min},q_k^\text{max}$ (potentially all equal to 0) and extracts power demand $p_k^\text{dem}+ \mathbf{i}q_k^\text{dem}$ from the network. Each bus operates at a voltage $v_k  \in \mathbb{C}$. Finding power flows that minimize active power loss is a problem that can be cast as an instance of \qcqp{}:
\begin{gather}
\inf_{v \in \mathbb{C}^n}~ v^H \frac{Y^H+Y}{2} v \label{eq:opf1} \\
\label{eq:opfP} \text{s.t.} ~~~
\forall k \in \mathcal{B},~~~ p_k^{\text{min}} - p^\text{dem}_k \leqslant v^H H_k v \leqslant p_k^{\text{max}} - p^\text{dem}_k, ~~~~~~
\\ \label{eq:opfQ}
\forall k \in \mathcal{B},~~~ q_k^{\text{min}} - q^\text{dem}_k \leqslant v^H \tilde{H}_k v \leqslant q_k^{\text{max}} - q^\text{dem}_k,
\\
\forall k \in \mathcal{B},~~~ (v_k^{\text{min}})^2 \leqslant v^H e_k e_k^T v \leqslant (v_k^{\text{max}})^2, \label{eq:opf2}
\end{gather}
where $H_k := \frac{Y^H e_k e_k^T + e_k e_k^T Y}{2}$ and $\tilde{H}_k := \frac{Y^H e_k e_k^T - e_k e_k^T Y}{2\textbf{i}}$ are Hermitian and $e_k$ is the $k^{th}$ column of the identity matrix.
In Section~\ref{subsec:Numerical Results}, power flows are computed that seek to minimize either power loss or generation costs $\sum_{k\in \mathcal{B}} a_k (p_k^\text{gen})^2 + b_k p_k^\text{gen} + c_k$ where $a_k,b_k,c_k \in \mathbb{R}$, $a_k \geqslant 0$, and $p_k^\text{gen} =  v^H H_k v+p_k^\text{dem}$. In the case of generation costs, new real variables $(t_k)_{k \in \mathcal{B}}$ are introduced, objective \eqref{eq:opf1} is replaced by $\sum_{k\in\mathcal{B}} t_k$, and new constraints are added for all $k \in \mathcal{B}$: $a_k (v^H H_k v+p_k^\text{dem})^2 + b_k (v^H H_k v+p_k^\text{dem}) + c_k \leqslant t_k$.
In Section~\ref{subsec:Numerical Results} apparent power flow limits $|v_l i_{lm}^H| \leqslant s_{lm}^\text{max}$ are enforced where $v_l i_{lm}^H = v^H F_{lm} v$ and $F_{lm} := a_{lm}^H e_l e_l^T + b_{lm}^H e_m e_l^T$, with $a_{lm} := (y_{lm} + y_{lm}^\text{gr})/|\rho_{lm}|^2$ and $b_{lm} := -y_{lm}/(\rho_{ml} \rho_{lm}^H)$. These can be written for all $(l,m) \in \mathcal{L}$: $(v^H\frac{F_{lm} + F_{lm}^H}{2} v )^2 + (v^H \frac{F_{lm} - F_{lm}^H}{2\textbf{i}} v )^2 \leqslant (s_{lm}^\text{max})^2$.
Note that generation cost and line flow constraints yield second-order conic constraints for all the relaxations considered in this paper as well as semidefinite constraints for higher orders of the moment/sum-of-squares hierarchies. The optimal power flow problem is invariant under the action of the torus (Section \ref{subsec:Invariant Hierarchy}) due to alternating current. We thus implement invariant hierarchies in Section~\ref{subsubsec:Moment/Sum-of-Squares Hierarchy NUM}.
\subsection{Numerical Results} 
\label{subsec:Numerical Results} 

We consider large test cases representing portions of European power systems: Great Britain~(GB)~\cite{gbnetwork}, Poland~(PL)~\cite{murillosanchez-thomas-zimmerman-2011}, and systems from the PEGASE project~\cite{josz-2016,pegase}. They were preprocessed (see Table~\ref{tab:size}) to remove low-impedance lines in order to improve the solver's numerical convergence, which is a typical procedure in power system analysis. A \mbox{$1\times 10^{-3}$}~per~unit low-impedance line threshold was used for all test cases except for PEGASE-1354 and PEGASE-2869 which use a \mbox{$3\times 10^{-3}$}~per~unit threshold. Table~\ref{tab:size} includes the at-least-locally-optimal objective values obtained from the interior point solver in M{\sc atpower}~\cite{murillosanchez-thomas-zimmerman-2011} for the problems after preprocessing. Note that the PEGASE systems specify generation costs that minimize active power losses, so the objective values in both columns are the same. Implementations use YALMIP \mbox{2015.06.26}~\cite{yalmip}, Mosek 7.1.0.28, and MATLAB 2013a on a computer with a quad-core 2.70~GHz processor and 16~GB of RAM. The results do not include the typically small formulation times.


\begin{table}[ht]
\centering
\caption{Size of Data (After Low-Impedance Line Preprocessing)}
\begin{tabular}{|l|c|c|c|c|}
\hline
\multicolumn{1}{|c|}{\textbf{Test}}  & \multicolumn{1}{c|}{\textbf{Number of}} & \multicolumn{1}{c|}{\textbf{Number of}} & \multicolumn{2}{c|}{M{\sc atpower} \textbf{Solution~\cite{murillosanchez-thomas-zimmerman-2011}}}\\ \cline{4-5} 
\multicolumn{1}{|c|}{\textbf{Case}}  & \multicolumn{1}{c|}{\textbf{Complex}} & \multicolumn{1}{c|}{\textbf{Edges}} & \multicolumn{1}{c|}{\textbf{Gen. Cost}} & \multicolumn{1}{c|}{\textbf{Loss Min.}}\\
\multicolumn{1}{|c|}{\textbf{Name}}  & \multicolumn{1}{c|}{\textbf{Variables}} & \multicolumn{1}{c|}{\textbf{in Graph}} & \multicolumn{1}{c|}{\textbf{(\$/hr)}} & \multicolumn{1}{c|}{\textbf{(MW)}}\\
\hline
GB-2224            & 2,053 & \hphantom{1}2,581 & 1,942,260  & \hphantom{1}60,614\\
PL-2383wp        & 2,177 & \hphantom{1}2,651 &   1,868,350 & \hphantom{1}24,991\\
PL-2736sp         & 2,182 & \hphantom{1}2,675 &  1,307,859 & \hphantom{1}18,336\\
PL-2737sop       & 2,183 & \hphantom{1}2,675 &    \hphantom{1,}777,617 & \hphantom{1}11,397\\
PL-2746wop      & 2,189 & \hphantom{1}2,708 &    1,208,257 & \hphantom{1}19,212\\
PL-2746wp        & 2,192 & \hphantom{1}2,686 &   1,631,737 & \hphantom{1}25,269\\
PL-3012wp        & 2,292 & \hphantom{1}2,805 &   2,592,462 & \hphantom{1}27,646\\
PL-3120sp         & 2,314 & \hphantom{1}2,835 &  2,142,720 & \hphantom{1}21,513\\
PEGASE-89      & \hphantom{1,1}70 & \hphantom{11,}185 & \hphantom{1,11}5,819 & \hphantom{11}5,819\\
PEGASE-1354  & \hphantom{1,}983 & \hphantom{1}1,526 &          \hphantom{1,1}74,043 & \hphantom{1}74,043\\
PEGASE-2869  & 2,120 & \hphantom{1}3,487 &         \hphantom{1,}133,945 & 133,945\\
PEGASE-9241  & 7,154 & 12,292 &                    \hphantom{1,}315,749 & 315,749\\
PEGASE-9241R\tablefootnote{\label{pegase9241R}PEGASE-9241 contains negative resistances to account for generators at lower voltage levels. In PEGASE-9241R these are set to 0.} 
		        & 7,154 & 12,292 & \hphantom{1,}315,785 & 315,785\\
\hline
\end{tabular}
\label{tab:size}
\end{table}

\subsubsection{Shor Relaxation} 
\label{subsubsec:Shor Relaxation NUM} 

Table~\ref{tab:SDP} shows the results of applying \sdpr{} and \sdpc{}. They yield global decision variables and the global objective value for the cases marked an asterisk (*) in Table~\ref{tab:SDP}. For those cases, the eigenvector associated to the largest eigenvalue is feasible up to 0.005 p.u. at voltage constraints and 1~MVA at all other constraints, and the objective evaluated in the eigenvector matches the bound within 0.05\% relative to the bound. The lower bounds in Table~\ref{tab:SDP} suggest that the corresponding M{\sc atpower} solutions in Table~\ref{tab:size} are at least very close to being globally optimal. The gap between the M{\sc atpower} solutions and the lower bounds from \sdpc{} for the generation cost minimizing problems are less than 0.72\% for GB-2224, 0.29\% for the Polish systems, and 0.02\% for the PEGASE systems with the exception of PEGASE-9241. The non-physical negative resistances in \mbox{PEGASE-9241} result in weaker lower bounds, yielding a gap of 1.64\%.
\begin{table}[ht]
\centering
\caption{Real and Complex SDP (Generation Cost Minimization)}
\begin{tabular}{|l|c|c|c|c|}
\hline
\multicolumn{1}{|c|}{\textbf{Case}} & \multicolumn{2}{c|}{\textbf{SDP}-$\mathbb{R}$} & \multicolumn{2}{c|}{\textbf{SDP}-$\mathbb{C}$} \\
\cline{2-5}
\multicolumn{1}{|c|}{\textbf{Name}} & \multicolumn{1}{c|}{Val. (\$/hr)} & \multicolumn{1}{c|}{\!\! Time (sec)\!\!} & \multicolumn{1}{c|}{Val. (\$/hr)} & \multicolumn{1}{c|}{\!\!Time (sec)\!\!} \\
\hline
GB-2224       & 1,928,194 &  \hphantom{1}10.9 & 1,928,444 &  \hphantom{11}6.2 \\
PL-2383wp     & 1,862,979 &  \hphantom{1}48.1 & 1,862,985 &  \hphantom{1}23.0 \\
PL-2736sp*    & 1,307,749 &  \hphantom{1}35.7 & 1,307,764 &  \hphantom{1}22.0 \\
PL-2737sop*   &   \hphantom{1,}777,505 &  \hphantom{1}41.7 &   \hphantom{1,}777,539 &  \hphantom{1}19.5 \\
PL-2746wop*   & 1,208,168 &  \hphantom{1}51.1 & 1,208,182 &  \hphantom{1}22.8 \\
PL-2746wp     & 1,631,589 &  \hphantom{1}43.8 & 1,631,655 &  \hphantom{1}20.0 \\
PL-3012wp     & 2,588,249 &  \hphantom{1}52.8 & 2,588,259 &  \hphantom{1}24.3 \\
PL-3120sp     & 2,140,568 &  \hphantom{1}64.4 & 2,140,605 &  \hphantom{1}25.5 \\
PEGASE-89*    &     \hphantom{1,11}5,819 &   \hphantom{11}1.5 &     \hphantom{1,11}5,819 &   \hphantom{11}0.9 \\
PEGASE-1354   &    \hphantom{1,1}74,035 &  \hphantom{1}11.2 &    \hphantom{1,1}74,035 &   \hphantom{11}5.6 \\
PEGASE-2869   &   \hphantom{1,}133,936 &  \hphantom{1}38.2 &   \hphantom{1,}133,936 &  \hphantom{1}20.6 \\
PEGASE-9241   &   \hphantom{1,}310,658 & 369.7 &   \hphantom{1,}310,662 & 136.1 \\
PEGASE-9241R &    \hphantom{1,}315,848 & 317.2 &   \hphantom{1,}315,731 &  \hphantom{1}95.9 \\
\hline
\end{tabular}
\label{tab:SDP}
\end{table}
In accordance with Appendices~\ref{app:Invariance of Shor Relaxation Bound} and~\ref{app:Invariance of SDP-R Relaxation Bound}, all objective values in Table~\ref{tab:SDP} match within $0.037\%$. \sdpc{} is faster (between a factor of 1.60 and 3.31) than \sdpr{}. Exploiting the isomorphic structure of complex matrices in \sdpc{} is thus better than eliminating a row and column in \sdpr{}.

\subsubsection{Second-Order Conic Relaxation}
\label{subsubsec:Second-Order Conic Relaxation NUM}

Table~\ref{tab:SOCP} shows the results of applying \socpr{} and \socpc{}. Unlike the Shor relaxation, they do not yield the global solution to any of the test cases.\footnote{\socpc{} generally does not provide a global solution with the exception of radial systems when certain non-trivial technical conditions are satisfied~\cite{low_tutorial}.} \socpc{} provides better lower bounds and is faster than \socpr{}. Lower bounds from \socpc{} are between $0.87\%$ and $3.96\%$ larger and solver times are faster by between a factor of $1.24$ and $6.76$ than those from \socpr{}. 

\begin{table}[ht]
\centering
\caption{Real and Complex SOCP (Generation Cost Minimization)}
\begin{tabular}{|l|c|c|c|c|}
\hline
\multicolumn{1}{|c|}{\textbf{Case}} & \multicolumn{2}{c|}{\textbf{SOCP}-$\mathbb{R}$} & \multicolumn{2}{c|}{\textbf{SOCP}-$\mathbb{C}$} \\
\cline{2-5}
\multicolumn{1}{|c|}{\textbf{Name}} & \multicolumn{1}{c|}{Val. (\$/hr)} & \multicolumn{1}{c|}{\!\! Time (sec)\!\!} & \multicolumn{1}{c|}{Val. (\$/hr)} & \multicolumn{1}{c|}{\!\!Time (sec)\!\!} \\
\hline
GB-2224      & 1,855,393 &  \hphantom{1}3.5 & 1,925,723 &  \hphantom{1}1.4 \\
PL-2383wp    & 1,776,726 &  \hphantom{1}8.5 & 1,849,906 &  \hphantom{1}2.4 \\
PL-2736sp    & 1,278,926 &  \hphantom{1}4.8 & 1,303,958 &  \hphantom{1}1.7 \\
PL-2737sop   &  \hphantom{1,}765,184 &  \hphantom{1}5.5 &   \hphantom{1,}775,672 &  \hphantom{1}1.6 \\
PL-2746wop   & 1,180,352 &  \hphantom{1}5.1 & 1,203,821 &  \hphantom{1}1.7 \\
PL-2746wp    & 1,586,226 &  \hphantom{1}5.5 & 1,626,418 &  \hphantom{1}1.7 \\
PL-3012wp    & 2,499,097 &  \hphantom{1}5.9 & 2,571,422 &  \hphantom{1}2.0 \\
PL-3120sp    & 2,080,418 &  \hphantom{1}6.2 & 2,131,258 &  \hphantom{1}2.2 \\
PEGASE-89    &     \hphantom{1,11}5,744 &  \hphantom{1}0.5 &     \hphantom{1,11}5,810 &  \hphantom{1}0.4 \\
PEGASE-1354  &   \hphantom{1,1}73,102 &  \hphantom{1}3.4 &    \hphantom{1,1}73,999 &  \hphantom{1}1.5 \\
PEGASE-2869  &   \hphantom{1,}132,520 &  \hphantom{1}9.0 &   \hphantom{1,}133,869 &  \hphantom{1}2.7 \\
PEGASE-9241  &   \hphantom{1,}306,050 & 35.3 &   \hphantom{1,}309,309 & 10.0 \\
PEGASE-9241R &   \hphantom{1,}312,682 & 36.7 &   \hphantom{1,}315,411 &  \hphantom{1}5.4 \\
\hline
\end{tabular}
\label{tab:SOCP}
\end{table}


\subsubsection{Moment/Sum-of-Squares Hierarchy}
\label{subsubsec:Moment/Sum-of-Squares Hierarchy NUM}

The real hierarchy globally solves a broad class of optimal power flow problems~\cite{pscc2014,cedric_tps,mh_sparse_msdp,ibm_paper} by first converting them to real numbers. The dense real and complex hierarchies solve problems up to 10 buses while the sparse ones solve problems with up 40 buses. In order to solve large-scale instances, we apply the mismatch hierarchy of Section~\ref{subsec:Example of Multi-Ordered Hierarchy: the Mismatch Hierarchy} with the following parameters: $\epsilon := 1$ MVA; $h := 2$; and $\Delta^\text{max}_\text{min} := 2$. See Appendix~\ref{app:five-bus example} for a small example. The mismatches are taken to be the modulus of the complex number whose real part is the mismatch for constraint $k$ in \eqref{eq:opfP} and whose imaginary part is the mismatch for constraint $k$ in \eqref{eq:opfQ}. In other words, apparent power mismatches are considered rather than active and reactive power seperately. To improve numerics, $|y_{\alpha,\beta}+\overline{y_{\alpha,\beta}}| \leqslant 2(v^{\text{max}})^{\alpha+\beta}$ and $|y_{\alpha,\beta}-\overline{y_{\alpha,\beta}}| \leqslant 2(v^{\text{max}})^{\alpha+\beta}$ are added to the complex hierarchy and $|y_{\alpha}| \leqslant (v^{\text{max}})^{\alpha}$ is added to the real hierarchy for all $|\alpha|,|\beta|\leqslant \max H(d)$ where $v^{\text{max}} := (v_1^{\text{max}},\hdots,v_n^{\text{max}})$ (see~\eqref{eq:opf2}). A similar procedure can be found in~\cite{waki-2006}.

In Tables~\ref{tab:MSOSR} and~\ref{tab:MSOSC}, the mismatch hierarchy is applied until the solution obtained is feasible up to 0.005 p.u. at voltage constraints and 1 MVA at all other constraints\footnote{Typical violations are smaller than 1~MVA. For instance, with the complex hierarchy PL-3012wp has over 99\% of the buses with less than 0.02~MVA violation, and only 0.09\% of the buses with greater than 0.1~MVA violation. Maximum line flow viotation is 0.0006~MVA.}, and until the objective evaluated in the solution matches the bound within 0.05\% relative to the bound.
The optimal values in the two tables match to at least 0.007\%, which is within the expected solver tolerance. Further, they match the optimal values for the loss minimizing problems in Table~\ref{tab:size} to within 0.013\%, further proving that they are globally optimal. However, local solvers do not always globally solve the optimal power flow~\cite{bukhsh2013,mh_sparse_msdp,ferc5,iscas2015}.
Though both hierarchies solve many small- and medium-size test cases which minimize generation cost, the mismatch hierarchy requires too many higher-order constraints for larger generation-cost-minimizing test cases.

The feasible set of the optimal power flow problem is included in the ball of radius $ \sum_{k\in \mathcal{B}} \left(v_k^{\max}\right)^2$ so a slack variable and a sphere constraint may be added as suggested in Section \ref{subsec:Convergence of the Complex Hierarchy}. In order to preserve sparsity, a slack variable and a sphere constraint may be added \emph{for each maximal clique} of the chordal extension of the network graph. However, it tends to introduce numerical convergence challenges in problems with several thousand buses, resulting in higher-order constraints at more buses and correspondingly longer solver times. Interestingly, the results in Table~\ref{tab:MSOSC} were obtained without the slack variables and sphere constraints. A potential way to account for this would be to compute the Hermitian complexity~\cite{putinar-2012} of the ideal generated by the polynomials associated with equality constraints. A step in that direction would be to assess the greatest number of distinct points (possibly infinite) $v^i \in \mathbb{C}^n, 1 \leqslant i \leqslant p,$ such that $(v^i)^H (H_k + \textbf{i}\tilde{H_k}) v^j = -p_k^\text{dem} - \textbf{i}q_k^\text{dem}$ for all buses $k$ not connected to a generator and for all $1\leqslant i,j \leqslant p$. The Hermitian complexity of the ideal generated by $\sum_{i=1}^n |z_i|^2 + \sigma(z)+a$ as defined in \eqref{eq:weak} with $a<0$ is equal to 1.



\begin{table}[ht]
\centering
\caption{Real Moment/Sum-of-Squares Hierarchy (Active Power Loss Minimization)}
\begin{tabular}{|l|c|c|c|c|}
\hline
\multicolumn{1}{|c|}{\textbf{Case}} & \textbf{Num.} & \multicolumn{1}{c|}{\textbf{Global Obj.}} & \textbf{Max. Viol.} & \multicolumn{1}{c|}{\textbf{Solver}} \\
\multicolumn{1}{|c|}{\textbf{Name}} & \textbf{Iter.} & \multicolumn{1}{c|}{\textbf{Val. (MW)}} & \textbf{(MVA)} & \multicolumn{1}{c|}{\textbf{Time (sec)}} \\
\hline
PL-2383wp    & 3  & \hphantom{1}24,990 & 0.25 & \hphantom{1,}583.4 \\
PL-2736sp    & 1  & \hphantom{1}18,334 & 0.39 & \hphantom{1,1}44.0 \\
PL-2737sop   & 1  & \hphantom{1}11,397 & 0.45 & \hphantom{1,1}52.4 \\
PL-2746wop   & 2  & \hphantom{1}19,210 & 0.28 & 2,662.4 \\
PL-2746wp    & 1  & \hphantom{1}25,267 & 0.40 & \hphantom{1,1}45.9 \\
PL-3012wp    & 5  & \hphantom{1}27,642 & 1.00 & \hphantom{1,}318.7 \\
PL-3120sp    & 7  & \hphantom{1}21,512 & 0.77 & \hphantom{1,}386.6 \\
PEGASE-1354  & 5  & \hphantom{1}74,043 & 0.85 & \hphantom{1,}406.9 \\ 
PEGASE-2869  & 6 & 133,944 & 0.63 & \hphantom{1,}921.3 \\ 
\hline
\end{tabular}
\label{tab:MSOSR}
\end{table}

\begin{table}[ht]
\centering
\caption{Complex Moment/Sum-of-Squares Hierarchy (Active Power Loss Minimization)}
\begin{tabular}{|l|c|c|c|c|}
\hline
\multicolumn{1}{|c|}{\textbf{Case}} & \textbf{Num.} & \multicolumn{1}{c|}{\textbf{Global Obj.}} & \textbf{Max. Viol.} & \multicolumn{1}{c|}{\textbf{Solver}} \\
\multicolumn{1}{|c|}{\textbf{Name}} & \textbf{Iter.} & \multicolumn{1}{c|}{\textbf{Val. (MW)}} & \textbf{(MVA)} & \multicolumn{1}{c|}{\textbf{Time (sec)}} \\
\hline
PL-2383wp    & 3  & \hphantom{1}24,991 & 0.10 & \hphantom{1,1}53.9 \\
PL-2736sp    & 1  & \hphantom{1}18,335 & 0.11 & \hphantom{1,1}17.8 \\
PL-2737sop   & 1  & \hphantom{1}11,397 & 0.07 & \hphantom{1,1}25.7 \\
PL-2746wop   & 2  & \hphantom{1}19,212 & 0.12 & \hphantom{1,}124.3 \\
PL-2746wp    & 1  & \hphantom{1}25,269 & 0.05 & \hphantom{1,1}18.5 \\
PL-3012wp    & 7  & \hphantom{1}27,644 & 0.91 & \hphantom{1,}141.0 \\ 
PL-3120sp    & 9  & \hphantom{1}21,512 & 0.27 & \hphantom{1,}193.9 \\
PEGASE-1354  & 11  & \hphantom{1}74,042 & 1.00 & 1,132.6 \\ 
PEGASE-2869  & 9 & 133,939 & 0.97 & \hphantom{1,}700.8 \\ 
\hline
\end{tabular}
\label{tab:MSOSC}
\end{table}

Tables~\ref{tab:MSOSR} and~\ref{tab:MSOSC} show that the complex hierarchy has advantages over the real hierarchy. In all cases except PEGASE-1354, there is a speedup factor in solver time of between 1.31 and 21.42. The most significant improvements are seen for cases (e.g., \mbox{PL-2383wp} and \mbox{PL-2746wop} whose biggest maximal clique has 19 nodes) where the higher-order constraints account for a large portion of the solver times. This is due to fewer terms in the higher-order constraints. There is also a speedup in solver time of between 2.0 and 5.9 for 7 out of the 8 small- to moderate-size generation-cost-minimizing test cases in~\cite{mh_sparse_msdp}, the exception being case39Q due to numerical difficulties. For those 7 cases, the maximum violation for the complex hierarchy is 0.08~MVA, with the remaining case (case118Q) having a maximum violation of 0.32~MVA.


\mbox{PL-3012wp}, \mbox{PL-3120sp}, \mbox{PEGASE-1354}, and \mbox{PEGASE-2869} require more iterations in the complex case than the real one. However, the improved speed per iteration results in faster overall solution times for all of these test cases except for \mbox{PEGASE-1354}, for which 6 additional iterations result in a factor of 2.78 slower solver time. Interestingly, the dense versions of the real and complex hierarchies yield the same bounds at each order for small test cases ($\leqslant$ 10 buses) from~\cite{hicss2014,bukhsh2013,borden-demarco-lesieutre-molzahn-2011,iscas2015}.



\section{Conclusion}
\label{sec:Conclusion}
We construct a complex moment/sum-of-squares hierarchy for complex polynomial optimization and prove convergence toward the global optimum. Theoretical and experimental evidence suggest that relaxing non-convex constraints before converting from complex to real numbers is better than doing the operations in the opposite order. 
We conclude with the question: is it possible to gain efficiency by transposing convex optimization algorithms from real to complex numbers?

\section*{Acknowledgements}
We wish to thank the anonymous reviewers for their precious time and feedback.
Many thanks to Mihai Putinar for the fruitful discussions that helped us to improve this paper. We also wish to thank Didier Henrion, Jean Bernard Lasserre, and Markus Schweighofer for their insightful comments.


\appendix

\section{Rank-2 Condition}
\label{app:Rank-2 Condition}
It is proven here that a Hermitian matrix $Z$ is positive semidefinite and has rank~1 if and only if $\Lambda(Z)$ is positive semidefinite and has rank~2.

$(\Longrightarrow)$ Say $Z = z z^H$ where real and imaginary parts are defined by $z = x_1 + \textbf{i} x_2$ and $(x_1,x_2) \neq (0,0)$. Then 
\begin{subequations}
\label{eq:comp1}
\begin{align}
\Lambda(Z) & = 
\left(
\begin{array}{cc}
x_1 x_1^T + x_2 x_2^T & x_1 x_2^T - x_2 x_1^T \\
x_2 x_1^T - x_1 x_2^T & x_1 x_1^T + x_2 x_2^T
\end{array}
\right) \\ \label{eq:comp2}
& = 
\left(
\begin{array}{c}
x_1 \\
x_2 
\end{array}
\right)
\left(
\begin{array}{c}
x_1 \\
x_2 
\end{array}
\right)^T
+
\left(
\begin{array}{r}
-x_2 \\
x_1 
\end{array}
\right)
\left(
\begin{array}{r}
-x_2 \\
x_1 
\end{array}
\right)^T. 
\end{align}
\end{subequations}
The rank of $\Lambda(Z)$ is equal to 2 since $( ~ x_1^T ~ x_2^T ~)^T$ and $(~ (-x_2)^T ~ x_1^T ~ )^T$ are non-zero orthogonal vectors.

$(\Longleftarrow)$ Say $\Lambda(Z) = x x^T + y y^T$ where $x$ and $y$ are non-zero real vectors. Consider the block structure $x = ( ~ x_1^T ~ x_2^T ~ )^T$ and $y = ( ~ y_1^T ~ y_2^T ~ )^T$. For $i = 1, \hdots, n$, it must be that
\begin{subequations}
\begin{gather}
x_{1i}^2 + y_{1i}^2 = x_{2i}^2 + y_{2i}^2, \label{eq:sym} \\
x_{1i} x_{2i} + y_{1i} y_{2i} = 0. \label{eq:asym}
\end{gather}
\end{subequations}
Two cases can occur. The first is that $x_{1i} x_{2i} \neq 0$ in which case there exists a real number $\lambda_i \neq 0$ such that 
\begin{equation}
\left\{
\begin{array}{rcr}
y_{1i} & = & - \lambda_i ~ x_{2i}, \\
y_{2i} & = & \frac{1}{\lambda_i} ~ x_{1i}.
\end{array}
\right.
\end{equation}
Equation \eqref{eq:sym} implies that $(1-\lambda_i^2) x_{1i}^2 = (1-\frac{1}{\lambda_i^2}) x_{2i}^2$ thus $(1-\lambda_i^2)(1-\frac{1}{\lambda_i^2}) \geqslant 0$ and $\lambda_i = \pm 1$. The second case is that $x_{1i} x_{2i} = 0$. Then, according to \eqref{eq:asym}, $y_{1i} y_{2i} = 0$. If either $x_{1i} = y_{1i} = 0$ or $x_{2i} = y_{2i} = 0$, then \eqref{eq:sym} implies that $x_{1i} = x_{2i} = y_{1i} = y_{2i} = 0$. If $x_{1i} = y_{2i} = 0$, then \eqref{eq:sym} implies that $y_{1i} = \pm x_{2i}$. If $x_{2i} = y_{1i} = 0$, then \eqref{eq:sym} implies that $y_{2i} = \pm x_{1i}$.

In any case, there exists $\epsilon_i = \pm 1$ such that
\begin{equation}
\left\{
\begin{array}{rcr}
y_{1i} & = & - \epsilon_i ~ x_{2i}, \\
y_{2i} & = & \epsilon_i ~ x_{1i}.
\end{array}
\right.
\end{equation}
For $i,j=1,\hdots, n$ it must be that 
\begin{subequations}
\begin{gather}
(1-\epsilon_i \epsilon_j)(x_{1i} x_{1j}-x_{2i} x_{2j} ) = 0, \label{eq:sym1} \\
(1-\epsilon_i \epsilon_j)(x_{1j}x_{2i} + x_{1i} x_{2j}) = 0. \label{eq:asym1}
\end{gather}
\end{subequations}
Moreover 
\begin{equation}
\left\{
\begin{array}{rcr}
x_{1i} x_{1j} + y_{1i} y_{1j} & = & x_{1i} x_{1j} + \epsilon_i \epsilon_j x_{2i} x_{2j}, \\
x_{1i} x_{2j} + y_{1i} y_{2j} & = & x_{1i} x_{2j} -\epsilon_i \epsilon_j x_{2i} x_{1j}.
\end{array}
\right.
\end{equation} 
It will now be shown that
\begin{equation}
\left\{
\begin{array}{rcr}
x_{1i} x_{1j} + y_{1i} y_{1j} & = & x_{1i} x_{1j} + x_{2i} x_{2j}, \\
x_{1i} x_{2j} + y_{1i} y_{2j} & = & x_{1i} x_{2j} - x_{2i} x_{1j}.
\end{array}
\right.
\label{eq:conclusion}
\end{equation}
It is obvious if $\epsilon_i \epsilon_j = 1$. If $\epsilon_i \epsilon_j = -1$, then \eqref{eq:sym1}--\eqref{eq:asym1} imply
\begin{subequations}
\begin{gather}
x_{1i} x_{1j}-x_{2i} x_{2j} = 0, \label{eq:sym2}\\
x_{1j}x_{2i} + x_{1i} x_{2j} = 0. \label{eq:asym2}
\end{gather}
\end{subequations}
If $x_{1i} x_{1j} x_{2i} x_{2j} = 0$, it can be seen that \eqref{eq:conclusion} holds. If not, \eqref{eq:sym2} implies that there exists a real number $\mu_{ij} \neq 0$ such that 
\begin{equation}
\left\{
\begin{array}{rcr}
x_{2i} & = & \mu_{ij} ~ x_{1i}, \\
x_{2j} & = & \frac{1}{\mu_{ij}} ~ x_{1j}.
\end{array}
\right.
\end{equation}
Further, \eqref{eq:asym2} implies that $(\mu_{ij} + \frac{1}{\mu_{ij}}) x_{1j} x_{2i} = 0$. This is impossible ($\mu_{ij} + \frac{1}{\mu_{ij}} \neq 0$ and $x_{1j} x_{2i} \neq 0$). Thus,~\eqref{eq:conclusion} holds.

With the left hand side corresponding to~$\Lambda(Z) = x x^T + y y^T$ and the right hand side corresponding to~\eqref{eq:comp2}, equation \eqref{eq:conclusion} implies that $\Lambda(Z)$ is equal to \eqref{eq:comp2}. Since the function $\Lambda$ is injective, it must be that $Z = (x_1+\textbf{i}x_2)(x_1+\textbf{i}x_2)^H$.

\section{Invariance of Shor Relaxation Bound}
\label{app:Invariance of Shor Relaxation Bound}
We have val(\csdpr{}) $\geqslant$ val(\sdpr{}) since the feasible set is more tightly constrained due to $\eqref{eq:csdpR4}$. To prove the opposite inequality, define $\tilde{\Lambda}(X) := (A + C)/2 + \textbf{i}(B-B^T)/2$ for all $X \in \mathbb{S}_{2n}$
using the block decomposition in the left hand part of \eqref{eq:csdpR4}.
It is proven here that if $X$ is a feasible point of \sdpr{}, then $\Lambda \circ \tilde{\Lambda}(X)$ is a feasible point of \csdpr{} with same objective value as $X$. Firstly, $\Lambda \circ \tilde{\Lambda}(X)$ satisfies \eqref{eq:csdpR4} because $\tilde{\Lambda}(X)$ is a Hermitian matrix.
Secondly, in order to show that $\Lambda \circ \tilde{\Lambda}(X)$ satisfies \eqref{eq:csdpR3}, notice that if $x = ( ~ x_1^T ~ x_2^T ~ )^T$ then
\begin{equation}\small
\begin{array}{c}
\left(
\begin{array}{r}
x_1 \\
x_2
\end{array}
\right)^T
\left(
\begin{array}{lc}
\hphantom{-}C & -B \\
-B^T & \hphantom{-}A
\end{array}
\right)
\left(
\begin{array}{r}
x_1 \\
x_2
\end{array}
\right) 
\\
= 
\\
\left(
\begin{array}{r}
-x_2 \\
x_1
\end{array}
\right)^T
\left(
\begin{array}{cl}
A & B^T \\
B & C
\end{array}
\right)
\left(
\begin{array}{r}
-x_2 \\
x_1
\end{array}
\right). 
\end{array}
\end{equation}
Hence $\Lambda \circ \tilde{\Lambda}(X)$ is equal to the sum of two positive semidefinite matrices.
Finally, to prove that $\Lambda \circ \tilde{\Lambda}(X)$ satisfies \eqref{eq:csdpR2} and has same objective value as $X$, notice that if $H \in \mathbb{H}_n$ and $Y\in \mathbb{S}_{2n}$, then $\text{Tr} \left[ \Lambda(H) Y \right] = \sum_{1 \leqslant i,j \leqslant 2n} \Lambda(H)_{ij} Y_{ji} = \sum_{1 \leqslant i,j \leqslant 2n} \Lambda(H)_{ij} Y_{ij} =  \sum_{1 \leqslant i,j \leqslant n} \text{Re}(H)_{ij} A_{ij} + \text{Im}(H)_{ij} B_{ij} + (-\text{Im}(H)_{ij}) (B^T)_{ij} + \text{Re}(H)_{ij} C_{ij} = \sum_{1 \leqslant i,j \leqslant n} \text{Re}(H_{ij}) (A + C)_{ij} + \text{Im}(H_{ij})(B - B^T)_{ij} = \hdots$ 
\\
$ 2 \sum_{1 \leqslant i,j \leqslant n} \text{Re} [ H_{ij} (\tilde{\Lambda}(Y)_{ij})^H ] = 2 \sum_{1 \leqslant i,j \leqslant n} H_{ij} (\tilde{\Lambda}(Y)_{ij})^H = 2\text{Tr} [ H \tilde{\Lambda}(Y) ]$. Completing the proof, for all $H \in \mathbb{H}_n$, $\text{Tr} [ \Lambda(H) ~ \Lambda \circ \tilde{\Lambda}(X) ] = 2 \text{Tr} [ H \tilde{\Lambda}(X) ] = \text{Tr} \left[ \Lambda(H) X \right]$. 

\section{Invariance of SDP-$\mathbb{R}$ Relaxation Bound}
\label{app:Invariance of SDP-R Relaxation Bound}
We assume that $X$ is a feasible point of SDP-$\mathbb{R}$ and construct a feasible point of SDP-$\mathbb{R}$ with same objective value and first diagonal entry equal to 0. Consider the eigenvalue decomposition $X = \sum_{k=1}^p x_k x_k^T$ for some $x_k \in \mathbb{R}^{2n}$ and $p\in \mathbb{N}$. For all $\theta \in \mathbb{R}$, define
\begin{equation}
R_\theta := \Lambda [ \cos (\theta) I_n + \textbf{i} \sin (\theta) I_n] =
\left(
\begin{array}{cr}
\cos (\theta) I_n & -\sin (\theta) I_n \\
\sin (\theta) I_n & \cos (\theta) I_n
\end{array}
\right).
\end{equation}
For $k=1,\hdots,p$, define $\theta_k \in \mathbb{R}$ such that $x_{k,n+1} + \textbf{i} x_{k,1} =: \sqrt{x_{k, n+1}^2+ x_{k1}^2} e^{\textbf{i}\theta_k}$. Construct $\tilde{X} := \sum_{k=1}^p (R_{\theta_k} x_k) (R_{\theta_k} x_k)^T \succcurlyeq 0$ whose first diagonal entry is equal to 0. If $H \in \mathbb{H}_n$, $\text{Tr}(\Lambda(H)\tilde{X}) = \sum_{k=1}^p \text{Tr}[ \Lambda(H) R_{\theta_k} x_k x_k^T R_{\theta_k}^T ] = \sum_{k=1}^p \text{Tr}[ R_{\theta_k}^T \Lambda(H) R_{\theta_k} x_k x_k^T  ] = \sum_{k=1}^p \text{Tr}[ \Lambda\{(\cos (\theta_k) I_n - \textbf{i} \sin (\theta_k) I_n) H (\cos (\theta_k) I_n + \textbf{i} \sin (\theta_k) I_n)\} x_k x_k^T ] = \hdots$
\\
$\sum_{k=1}^p \text{Tr}[ \Lambda( H ) x_k x_k^T  ] = \text{Tr}(\Lambda(H)X)$.

\section{Discrepancy Between Second-Order Conic Relaxation Bounds}
\label{app:Discrepancy Between Second-Order Conic Relaxation Bounds}
We have val(\csocpr{}) $\geqslant$ val(\socpr{}) since the feasible set is more tightly constrained.
The opposite inequality between optimal values does not hold, and this can be proven by considering the example QCQP-$\mathbb{C}$ defined by $\inf_{z_1,z_2 \in \mathbb{C}} ~ (1+\textbf{i}) \bar{z}_1 z_2 + (1-\textbf{i}) \bar{z}_2 z_1 ~ \mathrm{s.t.} ~ \bar{z}_1 z_1 \leqslant 1, ~ \overline{z}_2 z_2 \leqslant 1$. \csocpr{} yields the globally optimal value of $-2\sqrt{2}$, while \socpr{} yields $-4$. 

\section{Five-Bus Illustrative Example for Exploiting Sparsity}
\label{app:five-bus example}
Consider the five-bus optimal power flow problem in~\cite{bukhsh2013} which is an instance of QCQP-$\mathbb{C}$. Let $\text{ind}(\cdot)$ denote the set of indices corresponding to monomials of either the objective $f$ or constraint functions $(g_{i})_{1\leqslant i \leqslant 20}$. We have
\begin{align*}
\text{ind}(f) =\; & \{ (1,1),(1,2),(1,3),(3,5),(4,5),(5,5) \}, \\
\text{ind}(g_{1}) = \text{ind}(g_{2}) =\; & \{ (1,1),(1,2),(1,3)\} & \left[P_1^{\min}, Q_1^{\min}\right], \\
\text{ind}(g_{3}) = \text{ind}(g_{4}) =\; & \{ (1,2),(2,2),(2,3),(2,4) \} & \left[P_2, Q_2\right],\\
\text{ind}(g_{5}) = \text{ind}(g_{6}) =\; & \{ (1,3),(2,3),(3,3),(3,5) \} & \left[P_3,Q_3\right],\\
\text{ind}(g_{7}) = \text{ind}(g_{8}) =\; & \{ (2,4),(4,4),(4,5) \} & \left[P_4,Q_4\right], \\ \addtocounter{equation}{1}\tag{\theequation}\label{eq:5bus}
\text{ind}(g_{9}) = \text{ind}(g_{10}) =\; & \{ (3,5),(4,5),(5,5) \} & \left[P_5^{\min},Q_5^{\min}\right],\\
\text{ind}(g_{11}) = \text{ind}(g_{12}) =\; & \{ (1,1) \} & \left[V_1^{\min},V_1^{\max}\right], \\
\text{ind}(g_{13}) = \text{ind}(g_{14}) =\; & \{ (2,2) \} & \left[V_2^{\min},V_2^{\max}\right], \\
\text{ind}(g_{15}) = \text{ind}(g_{16}) =\; & \{ (3,3) \} & \left[V_3^{\min},V_3^{\max}\right], \\
\text{ind}(g_{17}) = \text{ind}(g_{18}) =\; & \{ (4,4) \} & \left[V_4^{\min},V_4^{\max}\right], \\ 
\text{ind}(g_{19}) = \text{ind}(g_{20}) =\; & \{ (5,5) \} & \left[V_5^{\min},V_5^{\max}\right],
\end{align*}
\noindent where the text in brackets indicates the origin of the constraint: $P_i$ and $Q_i$ for active and reactive power injection equality constraints, $P_i^{\min}$ and $Q_i^{\min}$ for lower limits on active and reactive power injections, and $V_i^{\min}$ and $V_i^{\max}$ for squared voltage magnitude limits at bus~$i$.
For brevity, the sphere constraints discussed in Section~\ref{subsec:Convergence of the Complex Hierarchy} are not enforced in this example. Regardless, the complex hierarchy with $d_i = 1,\; \forall i \in \left\lbrace 1,2,3,4,5,6,11,12,13,14,15,16 \right\rbrace$, $d_i = 2,\; \forall i \in \left\lbrace 7,8,9,10,17,18,19,20 \right\rbrace$, yields the global solution. Second-order constraints are identified using the mismatch hierarchy.

The graph $\left(\{1,\hdots,5\},\mathcal{E}^\text{mono}\right)$ corresponding to~\eqref{eq:5bus} is shown in Fig.~\ref{fig:5bus graph} where each node $i$ corresponds to a complex variable $z_i$. Edges $\mathcal{E}^\text{mono}$, which are denoted by solid lines in Fig.~\ref{fig:5bus graph}, connect variables that appear in the same monomial in any of the constraint equations or objective function. The supergraph $\left(\{1,\hdots,5\},\mathcal{E}^\text{con}\right)$ has edges $\mathcal{E}^\text{con}$ comprised of $\mathcal{E}^\text{mono}$ (solid lines in Fig.~\ref{fig:5bus graph}) augmented with edges connecting all variables within each constraint with $d_i > 1$ (dashed lines in Fig.~\ref{fig:5bus graph}). In this case, the supergraph is already chordal, so there is no need to form a chordal extension $\mathcal{G}^{\text{ch}}$. 
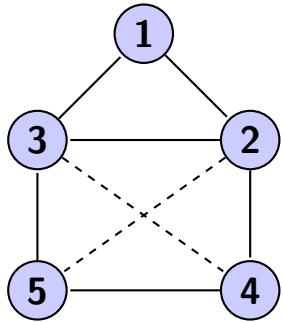
\begin{figure}[ht]
\centering
\begin{tikzpicture}[-,>=stealth',shorten >=1pt,auto,node distance=2cm,
  thick,main node/.style={circle,fill=blue!20,draw,font=\sffamily\Large\bfseries}]

  \node[main node] (1) {1};
  \node[main node] (3) [below left of=1] {3};
  \node[main node] (5) [below of=3] {5};
  \node[main node] (2) [below right of=1] {2};
  \node[main node] (4) [below of=2] {4};

  \draw (1) edge (2);
  \draw (2) edge (4);
  \draw (4) edge (5);
  \draw (5) edge (3);
  \draw (3) edge (1);
  \draw (2) edge (3);
  \draw[dashed] (2) edge (5);
  \draw[dashed] (3) edge (4);
\end{tikzpicture}
\caption{Graph Corresponding to Equations~\eqref{eq:5bus} from Five-Bus System in~\cite{bukhsh2013}}
\label{fig:5bus graph}
\end{figure}
The maximal cliques of the supergraph are $C_1 = \left\lbrace 1, 2, 3\right\rbrace$ and $C_2 = \left\lbrace 2, 3, 4, 5\right\rbrace$. Clique $C_2$ is the minimal covering clique for all second-order constraints $ g_i\left(z\right)\geqslant 0,\, \forall i \in \left\lbrace 7,8,9,10,17,18,19,20\right\rbrace$. The order associated with $C_2$ is two ($d^{\text{cl}}_2 = 2$) since the highest order $d_i$ among all constraints for which $C_2$ is the minimal covering clique is two. Clique $C_1$ is not the minimal covering clique for any constraints with $d_i > 1$, so $d^{\text{cl}}_1 = 1$. The globally optimal objective value obtained from the complex hierarchy specified above is 946.8 with corresponding decision variable $z = (1.0467 + 0.0000\mathbf{i}, 0.9550 - 0.0578\mathbf{i}, 0.9485 - 0.0533\mathbf{i}, 0.7791 + 0.6011\mathbf{i}, 0.7362 + 0.7487\mathbf{i})^T$.

\bibliography{mybib}{}

\begin{thebibliography}{10}

\bibitem{aittomaki-2009}
{\sc T.~Aittomaki and V.~Koivunen}, {\em {Beampattern Optimization by
  Minimization of Quartic Polynomial}}, IEEE/SP 15th W. Stat. Signal Process.,
  51 (2009), pp.~437--–440.

\bibitem{aliprantis-1999}
{\sc C.D. Aliprantis and K.~Border}, {\em Infinite Dimensional Analysis}, A
  Hitchhiker's guide, Second Edition, Springer-Verlag Berlin Heidelberg, 1999.

\bibitem{andersen2014}
{\sc M.S. Andersen, A.~Hansson, and L.~Vandenberghe}, {\em {Reduced-Complexity
  Semidefinite Relaxations of Optimal Power Flow Problems}}, IEEE Trans. Power
  Syst., 29 (2014), pp.~1855--1863.

\bibitem{anderson-1987}
{\sc E.J. Anderson and P.~Nash}, {\em Linear Programming in
  Infinite-Dimensional Spaces, Theory and Applications}, Wiley Int. Ser. Disc.
  Math. Optim., 1987.

\bibitem{atzmon-1975}
{\sc A.~Atzmon}, {\em {A Moment Problem for Positive Measures on the Unit
  Disc}}, Pacific J. Math., 59 (1975), pp.~317--325.

\bibitem{aubry-2013}
{\sc A.~Aubry, A.~{De Maio}, B.~Jiang, and S.~Zhang}, {\em {Ambiguity Function
  Shaping for Cognitive Radar via Complex Quartic Optimization}}, IEEE Trans.
  Signal Process., 61 (2013), pp.~5603–--5619.

\bibitem{bai-fujisawa-wang-wei-2008}
{\sc X.~Bai, H.~Wei, K.~Fujisawa, and Y.~Wang}, {\em {Semidefinite Programming
  for Optimal Power Flow Problems}}, Int. J. Elec. Power, 30 (2008),
  pp.~383--392.

\bibitem{bandeira-2014}
{\sc A.S. Bandeira, N.~Boumal, and A.~Singer}, {\em {Tightness of the Maximum
  Likelihood Semidefinite Relaxation for Angular Synchronization}}, Math.
  Program.,  (2016), pp.~1--–23.

\bibitem{bose-chandy-gayme-low-2015}
{\sc S.~Bose, D.F. Gayme, K.M. Chandy, and S.H. Low}, {\em {Quadratically
  Constrained Quadratic Programs on Acyclic Graphs with Application to Power}},
  IEEE Trans. Contr. Network Syst.,  (2015).

\bibitem{bose2014}
{\sc S.~Bose, S.H. Low, T.~Teeraratkul, and B.~Hassibi}, {\em {Equivalent
  Relaxations of Optimal Power Flow}}, IEEE Trans. Automat. Control,  (2014),
  p.~99.

\bibitem{boyd2009}
{\sc S.~Boyd and L.~Vandenberghe}, {\em {Convex Optimization}}, Cambridge
  University Press, 2009.

\bibitem{bukhsh2013}
{\sc W.A Bukhsh, A.~Grothey, K.I. {McKinnon}, and P.A. Trodden}, {\em {Local
  Solutions of the Optimal Power Flow Problem}}, IEEE Trans. Power Syst., 28
  (2013), pp.~4780–--4788.

\bibitem{candes-2013}
{\sc E.J. Cand\`es, Y.~C. Eldar, T.~Strohmer, and V.~Voroninski}, {\em {Phase
  Retrieval via Matrix Completion}}, SIAM J. Imaging Sci., 6 (2013),
  pp.~199–--225.

\bibitem{carpentier-1962}
{\sc M.J. Carpentier}, {\em {Contribution \`a l'\'Etude du Dispatching
  \'Economique}}, Bull. de la Soc. Fran. des \'Elec., 8 (1962),
  pp.~431–--447.

\bibitem{ferc5}
{\sc A.~Castillo and R.P. O'Neill}, {\em {Computational Performance of Solution
  Techniques Applied to the ACOPF (OPF Paper 5)}}, tech. report, US FERC, Jan.
  2013.

\bibitem{castillo-2013}
\leavevmode\vrule height 2pt depth -1.6pt width 23pt, {\em {Survey of
  Approaches to Solving the ACOPF (OPF Paper 4)}}, tech. report, US FERC, Mar.
  2013.

\bibitem{catlin-1996}
{\sc D.W. Catlin and J.P. D’Angelo}, {\em {A Stabilization Theorem for
  Hermitian Forms and Applications to Holomorphic Mappings}}, Math. Res. Lett.,
  3 (1996), pp.~149--–166.

\bibitem{chen-2009}
{\sc C.~Chen and P.P. Vaidyanathan}, {\em {MIMO Radar Waveform Optimization
  With Prior Information of the Extended Target and Clutter}}, IEEE Trans.
  Signal Process., 57 (2009), pp.~3533--3544.

\bibitem{cimpric-2009}
{\sc J.~Cimpric, S.~Kuhlmann, and C.~Scheiderer}, {\em {Sums of Squares and
  Moment Problems in Equivariant Situations}}, Trans. Am. Math. Soc., 361
  (2009), pp.~735--765.

\bibitem{coffrin2015}
{\sc C.~Coffrin, H.L. Hijazi, and P.~{Van Hentenryck}}, {\em {The QC
  Relaxation: Theoretical and Computational Results on Optimal Power Flow}},
  IEEE Trans. Power Syst., 31 (2016), pp.~3008--3018.

\bibitem{curto-2005}
{\sc R.~Curto and L.~Fialkow}, {\em {Truncated K-Moment Problems in Several
  Variables}}, J. Operator Theory, 54 (2005), pp.~189--226.

\bibitem{curto-2010}
{\sc R.~Curto and M.~Putinar}, {\em {Polynomially Hyponormal Operators}},
  Operator Theory: Advances and Applications, 207 (2010), pp.~195--207.

\bibitem{angelo-2002}
{\sc J.P. D'Angelo}, {\em {Inequalities from Complex Analysis}}, Carus Math.
  Monogr., MAA, 2002.

\bibitem{angelo-2010}
\leavevmode\vrule height 2pt depth -1.6pt width 23pt, {\em {Hermitian Analogues
  of Hilbert's 17th Problem}}, Adv. Math., 226 (2011), pp.~4607--4637.

\bibitem{angelo-2008}
{\sc J.P. D'Angelo and M.~Putinar}, {\em {Polynomial Optimization on
  Odd-Dimensional Spheres}}, in Emerging Applications of Algebraic Geometry,
  Springer New York, 2008.

\bibitem{putinar-2012}
\leavevmode\vrule height 2pt depth -1.6pt width 23pt, {\em {Hermitian
  Complexity of Real Polynomial Ideals}}, Int. J. Math., 23 (2012).

\bibitem{klerk-2000}
{\sc E.~{de Klerk}, T.~Terlaky, and K.~Roos}, {\em {Self-Dual Embeddings}}, in
  Handbook of Semidefinite Programming -- Theory, Algorithms, and Applications,
  H.~Wolkowicz, R.~Saigal, and L.~Vandenberghe, eds., Kluwer Acad. Publ.,
  Boston, 2000.

\bibitem{pegase}
{\sc S.~Fliscounakis, P.~Panciatici, F.~Capitanescu, and L.~Wehenkel}, {\em
  {Contingency Ranking with Respect to Overloads in Very Large Power Systems
  Taking into Account Uncertainty, Preventive and Corrective Actions}}, IEEE
  Trans. Power Syst., 28 (2013), pp.~4909--4917.

\bibitem{fogel-2014}
{\sc F.~Fogel, I.~Waldspurger, and A.~{d'Aspremont}}, {\em {Phase Retrieval for
  Imaging Problems}}, Math. Program. Comp.,  (2016), pp.~311--335.

\bibitem{ibm_paper}
{\sc B.~Ghaddar, J.~Marecek, and M.~Mevissen}, {\em {Optimal Power Flow as a
  Polynomial Optimization Problem}}, IEEE Trans. Power Syst.,  (2015).

\bibitem{schweighofer-2007}
{\sc D.~Grimm, T.~Netzer, and M.~Schweighofer}, {\em {A Note on the
  Representation of Positive Polynomials with Structured Sparsity}}, Arch.
  Math., 89 (2007), pp.~399--403.

\bibitem{grone}
{\sc R.~Grone, C.R. Johnson, E.M. S\'a, and H.~Wolkowicz}, {\em {Positive
  Definite Completions of Partial Hermitan Matrices}}, Linear Algebra Appl., 58
  (1984), pp.~109--124.

\bibitem{hilling-2010}
{\sc J.J. Hilling and A.~Sudbery}, {\em {The Geometric Measure of Multipartite
  Entanglement and the Singular Values of a Hypermatrix}}, J. Math. Phys., 51
  (2010).

\bibitem{jabr2006}
{\sc R.A. Jabr}, {\em {Radial Distribution Load Flow using Conic Programming}},
  IEEE Trans. Power Syst., 21 (2006), pp.~1458–--1459.

\bibitem{jiang-2014}
{\sc B.~Jiang, Z.~Li, and S.~Zhang}, {\em {Approximation Methods for Complex
  Polynomial Optimization}}, Springer Comput. Optim. Appl., 59 (2014),
  pp.~219--248.

\bibitem{jiang-2015}
\leavevmode\vrule height 2pt depth -1.6pt width 23pt, {\em {Characterizing
  Real-Valued Multivariate Complex Polynomials and Their Symmetric Tensor
  Representations}}, SIAM J. Matrix Anal. Appl.,  (2016), pp.~381--–408.

\bibitem{josz-2016}
{\sc C.~Josz, S.~Fliscounakis, J.~Maeght, and P.~Panciatici}, {\em {AC Power
  Flow Data in MATPOWER and QCQP format: iTesla, RTE Snapshots, and PEGASE}},
  \url{https://arxiv.org/abs/1603.01533},  (2016).

\bibitem{josz-2015}
{\sc C.~Josz and D.~Henrion}, {\em {Strong Duality in Lasserre's Hierarchy for
  Polynomial Optimization}}, Springer Optim. Lett.,  (2015).

\bibitem{cedric_tps}
{\sc C.~Josz, J.~Maeght, P.~Panciatici, and J.C. Gilbert}, {\em {Application of
  the Moment-SOS Approach to Global Optimization of the OPF Problem}}, IEEE
  Trans. Power Syst., 30 (2015), pp.~463--470.

\bibitem{kuhlmann-2007}
{\sc S.~Kuhlmann and M.~Putinar}, {\em {Positive Polynomials on Fibre
  Products}}, C. R. Acad. Sci. Paris, 344 (2007), pp.~681--684.

\bibitem{lasserre-2001}
{\sc J.~B. Lasserre}, {\em {Global Optimization with Polynomials and the
  Problem of Moments}}, SIAM J. Optim., 11 (2001), pp.~796--817.

\bibitem{lasserre-2010}
\leavevmode\vrule height 2pt depth -1.6pt width 23pt, {\em {Moments, Positive
  Polynomials and Their Applications}}, no.~1 in Imperial College Press
  Optimization Series, Imperial College Press, 2010.

\bibitem{lavaei-low-2012}
{\sc J.~Lavaei and S.H. Low}, {\em {Zero Duality Gap in Optimal Power Flow
  Problem}}, IEEE Trans. Power Syst., 27 (2012), pp.~92--107.

\bibitem{borden-demarco-lesieutre-molzahn-2011}
{\sc B.C. Lesieutre, D.K. Molzahn, A.R. Borden, and C.L. DeMarco}, {\em
  {Examining the Limits of the Application of Semidefinite Programming to Power
  Flow Problems}}, in 49th Annu. Allerton Conf. Commun., Control, Comput.,
  2011, pp.~28--30.

\bibitem{li-2012}
{\sc Z.~Li, S.~He, and S.~Zhang}, {\em {Approximation Methods for Polynomial
  Optimization: Models, Algorithms, and Applications}}, Comput. Optim. Appl.,
  Springer, New York, 2012.

\bibitem{yalmip}
{\sc J.~L\"{o}fberg}, {\em {YALMIP: A Toolbox for Modeling and Optimization in
  MATLAB}}, in {IEEE Int. Symp. Comput. Aided Contr. Syst. Des.}, 2004,
  pp.~284--289.

\bibitem{low_tutorial}
{\sc S.H. Low}, {\em {Convex Relaxation of Optimal Power Flow: Parts I \& II}},
  {IEEE Trans. Control Network Syst.}, 1 (2014), pp.~15--27.

\bibitem{luo-2010}
{\sc Z.~Luo, W.-K. Ma, A.M.-C. So, Y.~Ye, and S.~Zhang}, {\em {Semidefinite
  Relaxation of Quadratic Optimization Problems}}, IEEE Signal Process. Mag.,
  27 (2010), pp.~20--–34.

\bibitem{maricic-2003}
{\sc B.~Maricic, Z.-Q. Luo, and T.N. Davidson}, {\em {Blind Constant Modulus
  Equalization via Convex Optimization}}, IEEE Trans. Signal Process., 51
  (2003), pp.~805--–818.

\bibitem{iscas2015}
{\sc D.K. Molzahn, S.S. Baghsorkhi, and I.A. Hiskens}, {\em {Semidefinite
  Relaxations of Equivalent Optimal Power Flow Problems: An Illustrative
  Example}}, in IEEE Int. Symp. Circ. Syst. (ISCAS), May 24-27 2015.

\bibitem{pscc2014}
{\sc D.K. Molzahn and I.A. Hiskens}, {\em {Moment-Based Relaxation of the
  Optimal Power Flow Problem}}, 18th Power Syst. Comput. Conf. (PSCC),  (2014).

\bibitem{dan2015}
\leavevmode\vrule height 2pt depth -1.6pt width 23pt, {\em {Mixed SDP/SOCP
  Moment Relaxations of the Optimal Power Flow Problem}}, in IEEE Eindhoven
  PowerTech, 29 June--2 July 2015.

\bibitem{mh_sparse_msdp}
\leavevmode\vrule height 2pt depth -1.6pt width 23pt, {\em {Sparsity-Exploiting
  Moment-Based Relaxations of the Optimal Power Flow Problem}}, IEEE Trans.
  Power Syst., 30 (2015), pp.~3168--3180.

\bibitem{dan2013}
{\sc D.K. Molzahn, J.T. Holzer, B.C. Lesieutre, and C.L. DeMarco}, {\em
  {Implementation of a Large-Scale Optimal Power Flow Solver Based on
  Semidefinite Programming}}, IEEE Trans. Power Syst., 28 (2013),
  pp.~3987--3998.

\bibitem{hicss2014}
{\sc D.K. Molzahn, B.C. Lesieutre, and C.L. DeMarco}, {\em {Investigation of
  Non-Zero Duality Gap Solutions to a Semidefinite Relaxation of the Power Flow
  Equations}}, in 47th Hawaii Int. Conf. Syst. Sci. (HICSS), 6-9 Jan. 2014.

\bibitem{parrilo-2003}
{\sc P.A. Parrilo}, {\em {Semidefinite Programming Relaxations for
  Semialgebraic Problems}}, Math. Program., 96 (2003), pp.~293--320.

\bibitem{putinar-1993}
{\sc M.~Putinar}, {\em {Positive Polynomials on Compact Semi-Algebraic Sets}},
  Indiana Univ. Math. J., 42 (1993), pp.~969--984.

\bibitem{putinar-2006}
\leavevmode\vrule height 2pt depth -1.6pt width 23pt, {\em {On Hermitian
  Polynomial Optimization}}, Arch. Math., 87 (2006), pp.~41--51.

\bibitem{putinar-2013}
{\sc M.~Putinar and C.~Scheiderer}, {\em {Quillen Property of Real Algebraic
  Varieties}}, To appear in Muenster J. Math.

\bibitem{putinar-scheiderer-2012}
\leavevmode\vrule height 2pt depth -1.6pt width 23pt, {\em {Hermitian Algebra
  on the Ellipse}}, Illinois J. Math., 56 (2012), pp.~213--220.

\bibitem{quillen-1968}
{\sc D.G. Quillen}, {\em {On the Representation of Hermitian Forms as Sums of
  Squares}}, Invent. Math., 5 (1968), pp.~237--242.

\bibitem{riener-2013}
{\sc C.~Riener, T.~Theobald, L.~J. Andr\'en, and J.~B. Lasserre}, {\em
  {Exploiting Symmetries in SDP-Relaxations for Polynomial Optimization}},
  Math. of Operations Research, 38 (2013), pp.~122--141.

\bibitem{rudin-1987}
{\sc W.~Rudin}, {\em {Real and Complex Analysis}}, Math. Ser., Third Edition,
  McGraw Hill Int. Ed., 1987.

\bibitem{schweighofer-2005}
{\sc M.~Schweighofer}, {\em {Optimization of Polynomials on Compact
  Semialgebraic Sets}}, SIAM J. Optim., 15 (2005), pp.~805--825.

\bibitem{shor-1987b}
{\sc N.Z. Shor}, {\em {Quadratic Optimization Problems}}, Sov. J. Comput. Syst.
  Sci., 25 (1987), pp.~1--11.

\bibitem{singer-2011}
{\sc A.~Singer}, {\em {Angular Synchronization by Eigenvectors and Semidefinite
  Programming}}, Appl. Comput. Harmon. Anal., 30 (2011), pp.~20--–36.

\bibitem{sorber-2012}
{\sc L.~Sorber, M.V. Barel, and L.~{De Lathauwer}}, {\em {Unconstrained
  Optimization of Real Functions in Complex Variables}}, SIAM J. Optim., 22
  (2012), pp.~879--898.

\bibitem{tarjan}
{\sc R.E. Tarjan and M.~Yannakakis}, {\em {Simple Linear-Time Algorithms to
  Test Chordality of Graphs, Test Acyclicity of Hypergraphs, and Selectively
  Reduce Acyclic Hypergraphs}}, {SIAM J. Comput.}, 13 (1984), p.~566.

\bibitem{taylor2015}
{\sc J.A. Taylor}, {\em {Convex Optimization of Power Systems}}, Cambridge
  University Press, 2015.

\bibitem{toker-1998}
{\sc O.~Toker and H.~Ozbay}, {\em {On the Complexity of Purely Complex Mu
  Computation and Related Problems in Multidimensional Systems}}, IEEE Trans.
  Automat. Control, 43 (1998), pp.~409--414.

\bibitem{maria-2005}
{\sc M.~Trnovsk\'a}, {\em {Strong Duality Conditions in Semidefinite
  Programming}}, J. Electr. Eng., 56 (2005), pp.~1--5.

\bibitem{gbnetwork}
{\sc {University of Edinburgh Power Systems Test Case Archive}}, {\em {GB
  Network}}.

\bibitem{vandenberghe2014}
{\sc L.~Vandenberghe and M.S. Andersen}, {\em {Chordal Graphs and Semidefinite
  Optimization}}, {Found. Trends Optim.}, 1 (2015), pp.~241--433.

\bibitem{waki-2006}
{\sc H.~Waki, S.~Kim, M.~Kojima, and M.~Muramatsu}, {\em {Sums of Squares and
  Semidefinite Program Relaxations for Polynomial Optimization Problems with
  Structured Sparsity}}, SIAM J. Optim., 17 (2006), pp.~218--242.

\bibitem{wirtinger-1927}
{\sc W.~Wirtinger}, {\em {Zur Formalen Theorie der Funktionen von Mehr
  Komplexen Ver\"anderlichen}}, Math. Ann., 97 (1927), pp.~357--375.

\bibitem{zhang2013}
{\sc B.~Zhang and D.~Tse}, {\em {Geometry of Feasible Injection Region of Power
  Networks}}, IEEE Trans. Power Syst., 28 (2013), pp.~788--797.

\bibitem{murillosanchez-thomas-zimmerman-2011}
{\sc R.~Zimmerman, C.~Murillo-S\'anchez, and R.~Thomas}, {\em {MATPOWER:
  Steady-State Operations, Planning, and Analysis Tools for Power Systems
  Research and Education}}, IEEE Trans. Power Syst., 99 (2011), pp.~1--8.

\end{thebibliography}
\bibliographystyle{siam}

\end{document}